\documentclass[10pt,reqno,final]{amsart}
\usepackage{lineno,hyperref}
\usepackage{epsfig,amssymb,amsmath,version,amsthm}
\usepackage{amssymb,version,graphicx,fancybox,mathrsfs,multirow}
\usepackage[notcite,notref]{showkeys}
\usepackage{url,hyperref}
\usepackage{subfigure,epstopdf,bm}
\usepackage{color}
\usepackage{appendix}
\usepackage{algorithm}
\usepackage{algorithmicx}
\usepackage{algpseudocode}
\usepackage{amsmath}
\usepackage{amsfonts}
\usepackage{leftidx}
\usepackage{amssymb}
\usepackage{graphicx}%

 \catcode`\@=11 \theoremstyle{plain}
 \@addtoreset{equation}{section}   
 
\@addtoreset{figure}{section}
\renewcommand\thefigure{\thesection.\@arabic\c@figure}
\@addtoreset{table}{section}
\renewcommand\thetable{\thesection.\@arabic\c@table}
 \newtheorem{thm}{\bf Theorem}

 \newenvironment{theorem}{\begin{thm}} {\end{thm}}
 \newtheorem{cor}{\bf Corollary}
 \newtheorem{prop}{Proposition}[section]

 \newtheorem{lmm}{\bf Lemma}
 
 \newenvironment{lemma}{\begin{lmm}}{\end{lmm}}
 \theoremstyle{remark}
 \newtheorem{rem}{Remark}[section]

 \def \ri {{\rm i}}
 \def \widebar {\accentset{{\cc@style\underline{\mskip10mu}}}}
 \newcommand{\bs}[1]{\boldsymbol{#1}}

 \DeclareSymbolFont{ugmL}{OMX}{mdugm}{m}{n}
 \SetSymbolFont{ugmL}{bold}{OMX}{mdugm}{b}{n}
 \DeclareMathAccent{\wideparen}{\mathord}{ugmL}{"F3}

 \renewcommand \wedge \times

\begin{document}
\graphicspath{{./figs/}}
{\title[ME, LE, AND M2L FOR CHARGES IN LAYERED MEDIA]{Exponential convergence for multipole and local expansions and their  translations for sources in layered media: three-dimensional Laplace equation}}

\author{Bo Wang}
\address{LCSM(MOE), School of Mathematics and Statistics, Hunan Normal University, Changsha, Hunan, 410081, P. R. China.}
\email{bowang@hunnu.edu.cn}

%
%

\author{Wen Zhong Zhang}
\address{Department of Mathematics, Southern Methodist University, Dallas, TX 75275, USA.}
\email{wenzhongz@smu.edu}

\author{Wei Cai}
\address{Department of Mathematics, Southern Methodist University, Dallas, TX 75275, USA.}
\email{cai@smu.edu}

\maketitle

\begin{abstract}
	In this paper, we prove the exponential convergence of the multipole and local expansions, shifting and translation operators used in fast multipole methods (FMMs) for 3-dimensional Laplace equations in layered media. These theoretical results ensure the exponential convergence of the FMM which has been shown by the numerical results recently reported in \cite{wang2019fastlaplace}. As the free space components are calculated by the classic FMM, this paper will focus on the analysis for the reaction components of the Green's function for the Laplace equation in layered media. We first prove that the density functions in the integral representations of the reaction components are analytic and bounded in the right half complex plane. Then, using the Cagniard-de Hoop transform and contour deformations,  estimate for the remainder terms of the truncated expansions is given, and, as a result, the exponential convergence for the expansions and translation operators is proven.
\end{abstract}


\section{Introduction}

The well-known fast multipole method (FMM) proposed by
Greengard and Rohklin \cite{greengard1987fast,greengard1997new} for particles in free spaces has been a revolutionary development for scientific and engineering computing. The algorithm was based on low rank approximations for the far field of sources, which are obtained by using truncated multipole expansions (MEs) and local expansions (LEs) with a truncation number $p$. The capability of using a small number $p$ to achieve high accuracy is due to the exponential convergence of the MEs and LEs, as well as the shifting and translation operators for multipole to multipole (M2M), local to local (L2L), and multipole to local (M2L) conversions.
Recently, we have extended the FMMs of the Helmholtz, Laplace and Poisson-Botzmann equations from free space to layered media (cf. \cite{wang2019fast, zhang2020exponential, wang2019fastlaplace, wang2020fast}). The new FMMs significantly enlarge the application area of the classic ones. Many important applications, e.g. parasitic parameter extraction of very large-scale integrated (VLSI) circuits \cite{oh1994capacitance, zhao1998efficient}, complex scattering problem in meta-materials \cite{chen2018accurate}, electrical potential computation in ion channel simulation \cite{lin2014accuracy}, etc, can be solved more efficiently and accurately by using the FMMs for layered media.

The mathematical proof for the exponential convergence of the MEs and LEs and the corresponding translation operators is one of the key issues in
developing FMMs for the aforementioned equations in layered media.
As the reaction components of the Green's functions in layered media do not have a closed form in the physical space, Sommerfeld-type integral representations and extended Funck-Hecke formula are used to derive the MEs, LEs and translation operators for the FMMs mentioned above. The distinct feature of the expansions and translation operators for reaction components is that they involve Sommerfeld-type integrals with integrands depending on the layered structure of the media. Hence, the main difficulty in the convergence analysis is how to give a delicate estimate on the Sommerfeld-type integrals. Recently, we have proved the exponential convergence for the 2-dimensional Helmholtz equation case \cite{zhang2020exponential} and numerically showed that the MEs in 3-dimensional cases also have exponential convergence similarly as in 2-dimensional cases. However, the theoretical proof for 3-dimensional cases are much more involved technically due to the double improper integrals induced by the 2-dimensional inverse Fourier transform used in the derivation of the 3-dimensional layered Green's function while only 1-dimensional inverse Fourier transform is needed in 2-dimensional cases.

In this paper, we will continue our previous work on 2-D Helmholtz equation \cite{zhang2020exponential} and prove the exponential convergence of the MEs, LEs and corresponding translation operators for the Green's function of 3-dimensional Laplace equation in layered media. First, we consider the direct MEs for reaction components. We prove that they have exponential convergence but with convergence rates depend on polarization distances as defined in \cite{wang2019fast} and then used by \cite{wang2019fastlaplace, wang2020fast}. Therefore, the concept of equivalent polarization sources  is crucial for the development of the FMMs for reaction components. By introducing the equivalent polarization sources, the reaction components have been reformulated and the MEs, LEs and translation operators are re-derived according to the new formulations. In this paper, we further give theoretical proof for their exponential convergence and show that the convergence rates are determined by the Euclidean distance between the targets and corresponding equivalent polarization sources. As a result, we validate the idea of using the re-derived MEs, LEs and translation operators with equivalent polarization sources in the FMMs for the reaction components. All theoretical results proved in this paper show that the FMM for Laplace equation in layered media developed in \cite{wang2019fastlaplace} is a highly accurate and error controllable algorithm as same as the free space FMM.

The rest of the paper is organized as follows. In section 2, we review the integral representation of the Green's function of Laplace equation in layered media and a recursive algorithm for a stable and efficient calculation of the reaction densities of general multi-layered media. Based on the recursive formulas, we prove that the reaction densities are bounded and analytic in the right half complex plane, which is important for the estimate of the Sommerfeld-type integrals. Section 3 will review the derivation of MEs, LEs, shifting and translation operators for layered Green's function. In section 4, we first review the exponential convergence of the ME, LE, shifting and translation operators for the free space components of layered Green's function. Then, proofs for the exponential convergence of the MEs, LEs and translation operators for the reaction components are presented. Finally, a conclusion is given in Section 5.

\section{Green's function of 3-dimensional Laplace equation in layered media}
Consider a layered medium consisting of $L$-interfaces located at $z=d_{\ell
},\ell=0,1,\cdots,L-1$, see Fig. \ref{layerstructure}. The piece wise constant material parameter is described by $\{\varepsilon_{\ell}\}_{\ell=0}^L$.
Suppose we have a point
source at $\boldsymbol{r}^{\prime}=(x^{\prime},y^{\prime},z^{\prime})$ in the
$\ell^{\prime}$th layer ($d_{\ell^{\prime}}<z^{\prime}<d_{\ell^{\prime}-1}$), then, the layered media Green's function $u_{\ell\ell'}(\bs r, \bs r')$ for the Laplace equation satisfies
\begin{equation}\label{Laplaceeqlayer}
\boldsymbol{\Delta}u_{\ell\ell'}(\boldsymbol{r},\boldsymbol{r}^{\prime
})=-\delta(\boldsymbol{r},\boldsymbol{r}^{\prime}),
\end{equation}
at field point $\boldsymbol{r}=(x,y,z)$ in the $\ell$th layer ($d_{\ell
}<z<d_{\ell}-1$) where $\delta(\boldsymbol{r},\boldsymbol{r}^{\prime})$ is the
Dirac delta function.
By using Fourier transforms along $x-$ and $y-$directions, the problem can be solved analytically for each layer in $z$ by imposing
transmission conditions at the interface between $\ell$th and $(\ell-1)$th
layer ($z=d_{\ell-1})$, \textit{i.e.},
\begin{equation}\label{transmissioncond}
u_{\ell-1,\ell'}(x,y,z)=u_{\ell\ell'}(x,y,z),\quad \varepsilon_{\ell-1}\frac{\partial  u_{\ell-1,\ell'}(x,y,z)}{\partial z}=\varepsilon_{\ell}\frac{\partial \widehat u_{\ell\ell'}(k_{x},k_{y},z)}{\partial z},
\end{equation}
as well as the decaying conditions in the top and bottom-most layers as
$z\rightarrow\pm\infty$.
\begin{figure}[ht!]\label{layerstructure}
	\centering
	\includegraphics[scale=0.7]{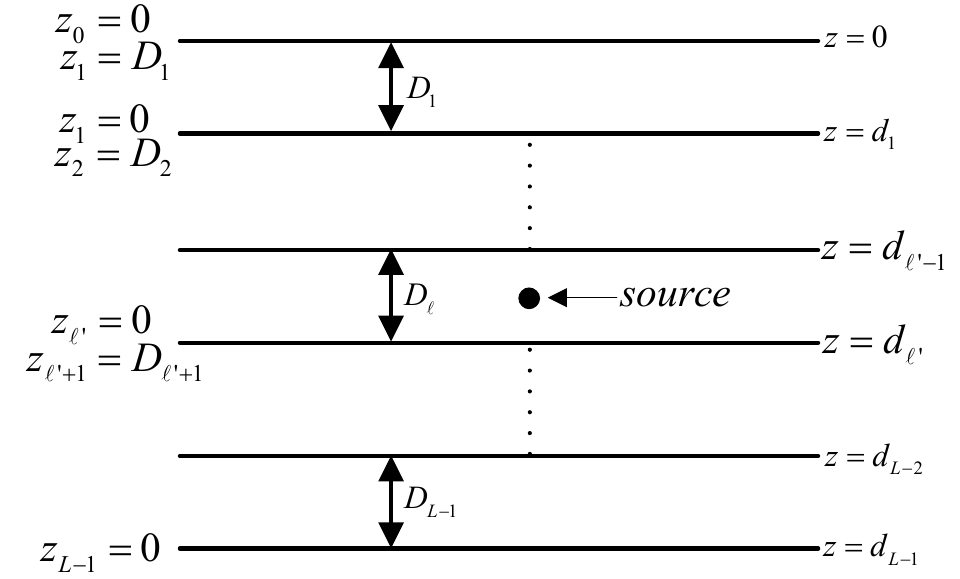}
	\caption{Sketch of the layer structure for general multi-layer media.}
\end{figure}

By applying Fourier transform in $x, y$ direction and solving the resulted ODE with interface conditions, we can obtain the expression of the Green's function in the physical domain as (cf. \cite{wang2019fastlaplace} Appendix B.)
\begin{equation}\label{layeredGreensfun}
u_{\ell\ell^{\prime}}(\boldsymbol{r},\boldsymbol{r}^{\prime})=%
\begin{cases}
\displaystyle u_{0\ell^{\prime}}^{11}(\boldsymbol{r},\boldsymbol{r}^{\prime})+u_{0\ell^{\prime}}^{12}%
(\boldsymbol{r},\boldsymbol{r}^{\prime}),\\
\displaystyle u_{\ell\ell^{\prime}}^{11}%
(\boldsymbol{r},\boldsymbol{r}^{\prime})+u_{\ell\ell^{\prime}}^{12}%
(\boldsymbol{r},\boldsymbol{r}^{\prime})+u_{\ell\ell^{\prime}}^{21
}(\boldsymbol{r},\boldsymbol{r}^{\prime})+u_{\ell\ell^{\prime}}^{22
}(\boldsymbol{r},\boldsymbol{r}^{\prime}), & \ell\neq\ell^{\prime},\\
\displaystyle u_{\ell\ell^{\prime}}^{11}%
(\boldsymbol{r},\boldsymbol{r}^{\prime})+u_{\ell\ell^{\prime}}^{12}%
(\boldsymbol{r},\boldsymbol{r}^{\prime})+u_{\ell\ell^{\prime}}^{21
}(\boldsymbol{r},\boldsymbol{r}^{\prime})+u_{\ell\ell^{\prime}}^{22
}(\boldsymbol{r},\boldsymbol{r}^{\prime})+\frac{1}{4\pi
	|\bs r-\bs r^{\prime}|}, &
\ell=\ell^{\prime},\\
\displaystyle u_{L\ell^{\prime}}^{21}%
(\boldsymbol{r},\boldsymbol{r}^{\prime})+u_{L\ell^{\prime}}^{22}%
(\boldsymbol{r},\boldsymbol{r}^{\prime}),
\end{cases}
\end{equation}
with reaction components given by Sommerfeld-type integrals:
\begin{equation}\label{generalcomponents}
u_{\ell\ell'}^{\mathfrak{ab}}(\bs r, \bs r')=\frac{1}{8\pi^2 }\int_{-\infty}^{\infty}\int_{-\infty}^{\infty}\frac{1}{k_{\rho}}e^{\ri \bs k\cdot\bs\tau_{\ell\ell'}^{\mathfrak{ab}}(\bs r, \bs r')}\sigma_{\ell\ell'}^{\mathfrak{ab}}(k_{\rho})dk_x dk_y, \quad\mathfrak{a}, \mathfrak{b}=1, 2,
\end{equation}
where $k_{\rho}=\sqrt{k_x^2+k_y^2}$, $\bs k=(k_x, k_y, \ri k_{\rho})$,
\begin{equation}\label{coordmapping}
\begin{split}
\bs\tau_{\ell\ell'}^{11}(\bs r, \bs r')=&(x-x', y-y', z-d_{\ell}+z'-d_{\ell'}),\\
\bs\tau_{\ell\ell'}^{12}(\bs r, \bs r')=&(x-x', y-y', z-d_{\ell}+d_{\ell'-1}-z'),\\
\bs\tau_{\ell\ell'}^{21}(\bs r, \bs r')=&(x-x', y-y', d_{\ell-1}-z+z'-d_{\ell'}),\\
\bs\tau_{\ell\ell'}^{22}(\bs r, \bs r')=&(x-x', y-y', d_{\ell-1}-z+d_{\ell'-1}-z'),
\end{split}
\end{equation}
are coordinate mappings depend on interfaces and $\sigma_{\ell\ell'}^{\mathfrak{ab}}(k_{\rho})$ are the reaction densities in Fourier spectral space. The expression \eqref{layeredGreensfun} is a general formula for source $\bs r'$ in the middle layer. In the cases of the source $\bs r'$ in the top or bottom most layer, the reaction components $\{u_{\ell\ell'}^{\mathfrak a2}(\bs r, \bs r')\}_{\mathfrak a=1}^2$ and $\{u_{\ell\ell'}^{\mathfrak a1}(\bs r, \bs r')\}_{\mathfrak a=1}^2$ will vanish, respectively.

The reaction densities $\sigma_{\ell\ell'}^{\mathfrak{ab}}(k_{\rho})$ only depend on the layered structure and the material parameter $\varepsilon_{\ell}$ in each layer. According to the derivation in \cite[Appendix B]{wang2019fastlaplace}, a stable recurrence formula is available for more  general interface conditions
\begin{equation}
\label{generalinterfacecond}
a_{\ell-1}u_{\ell-1,\ell'}(x,y,z)=a_{\ell}u_{\ell\ell'}(x,y,z),\quad b_{\ell-1}\frac{\partial  u_{\ell-1,\ell'}(x,y,z)}{\partial z}=b_{\ell}\frac{\partial \widehat u_{\ell\ell'}(k_{x},k_{y},z)}{\partial z},
\end{equation}
where $\{a_{\ell}, b_{\ell}\}$ are given constants. In order to prove some key properties of the densities, we will review the recurrence formula. For this purpose, let us define
\begin{equation}\label{expdef}
\begin{split}
&d_{-1}:=d_0,\quad d_{L+1}:=d_{L},\quad D_{\ell}:=d_{\ell-1}-d_{\ell},\quad e_{\ell}:=e^{- k_{\rho}D_{\ell}},\quad \ell=0, 1, \cdots, L,\\
&\gamma_{\ell}^+=\frac{a_{\ell}}{a_{\ell-1}}+\frac{b_{\ell}}{b_{\ell-1}},\quad \gamma_{\ell}^-=\frac{a_{\ell}}{a_{\ell-1}}-\frac{b_{\ell}}{b_{\ell-1}},\quad C^{(\ell)}=\prod\limits_{j=0}^{\ell-1}\frac{1}{2e_j},\quad \ell=1, 2, \cdots, L,
\end{split}
\end{equation}
and matrices
\begin{equation}\label{transmissionmat}
\begin{split}
&\mathbb T^{\ell-1,\ell}:=\begin{pmatrix}
T_{11}^{\ell-1,\ell} & T_{12}^{\ell-1,\ell}\\
T_{21}^{\ell-1,\ell} & T_{22}^{\ell-1,\ell}
\end{pmatrix}=\frac{1}{2e_{\ell-1}}\begin{pmatrix}
e_{\ell-1} & 0\\
0 & 1
\end{pmatrix}
\begin{pmatrix}
\displaystyle \gamma_{\ell}^+ & \displaystyle\gamma_{\ell}^-\\[7pt]
\displaystyle \gamma_{\ell}^- & \displaystyle\gamma_{\ell}^+
\end{pmatrix}\begin{pmatrix}
e_{\ell} & 0\\
0 & 1
\end{pmatrix}=\frac{1}{2e_{\ell-1}}\widetilde{\mathbb T}^{\ell-1,\ell},\\
&\breve{\mathbb S}^{(\ell)}:=\begin{pmatrix}
\breve S_{11}^{(\ell)} & \breve S_{12}^{(\ell)}\\
\breve S_{21}^{(\ell)} & \breve S_{22}^{(\ell)}
\end{pmatrix}=\frac{1}{2}\begin{pmatrix}
\displaystyle\frac{1}{a_{\ell}} & \displaystyle\frac{1}{b_{\ell}}\\[7pt]
\displaystyle\frac{1}{a_{\ell}e_{\ell}} & \displaystyle-\frac{1}{b_{\ell}e_{\ell}}
\end{pmatrix},\quad \mathbb A^{(\ell)}:=\begin{pmatrix}
\alpha_{11}^{(\ell)} & \alpha_{12}^{(\ell)}\\[7pt]
\alpha_{21}^{(\ell)} & \alpha_{22}^{(\ell)}
\end{pmatrix}=\widetilde{\mathbb T}^{01}\cdots \widetilde{\mathbb T}^{\ell-1,\ell}.
\end{split}
\end{equation}
Then, the recursive algorithm is summarized as follow:
\begin{algorithm}\label{algorithm2}
	\caption{Stable and efficient algorithm for reaction densities $\sigma_{\ell\ell'}^{\mathfrak{ab}}(k_{\rho})$ in \eqref{generalcomponents}.}
	\begin{algorithmic}
		\For{$\ell'=0 \to L$}
		\If{$\ell'<L$}
		\begin{equation}\label{bottommostdensity1}
		\begin{split}
		&\sigma_{L\ell'}^{21}(k_{\rho})=-\frac{C^{(\ell'+1)}}{C^{(L)}\alpha_{22}^{(L)}}\begin{pmatrix}
		\alpha^{(\ell')}_{21} & \alpha^{(\ell')}_{22}
		\end{pmatrix}2e_{\ell'}\breve{\mathbb S}^{(\ell')}\begin{pmatrix}
		-a_{\ell'}\\
		b_{\ell'}
		\end{pmatrix},
		\end{split}
		\end{equation}
		\EndIf
		\If{$\ell'>0$}
		\begin{equation}\label{bottommostdensity2}
		\sigma_{L\ell'}^{22}(k_{\rho})=-\frac{C^{(\ell')}}{C^{(L)}\alpha_{22}^{(L)}}\begin{pmatrix}
		\alpha^{(\ell'-1)}_{21} & \alpha^{(\ell'-1)}_{22}
		\end{pmatrix}2e_{\ell'-1}\breve{\mathbb S}^{(\ell'-1)}\begin{pmatrix}
		a_{\ell'}\\
		b_{\ell'}
		\end{pmatrix},
		\end{equation}
		\EndIf
		\For{$\ell= L-1 \to 0$ }
		\If{$\ell=\ell'$}
		\begin{equation}\label{middledensity}
		\sigma_{\ell\ell'}^{11}(k_{\rho})=T^{\ell'\ell'+1}_{11}\sigma_{\ell'+1,\ell'}^{11}+T^{\ell'\ell'+1}_{12}\sigma_{\ell'+1,\ell'}^{21}-\breve S_{11}^{(\ell')}a_{\ell'}+\breve S_{12}^{(\ell')}b_{\ell'},
		\end{equation}
		\Else
		\begin{equation}
		\sigma_{\ell\ell'}^{12}(k_{\rho})=T^{\ell\ell+1}_{11}\sigma_{\ell+1,\ell'}^{11}+T^{\ell\ell+1}_{12}\sigma_{\ell+1,\ell'}^{21},
		\end{equation}
		\EndIf
		\If{$\ell=\ell'-1$}
		\begin{equation}
		\sigma_{\ell\ell'}^{12}(k_{\rho})=T^{\ell'-1,\ell'}_{11}\sigma_{\ell'\ell'}^{12}+T^{\ell'-1,\ell'}_{12}\sigma_{\ell'\ell'}^{22}+\breve S_{11}^{(\ell'-1)}a_{\ell'}+\breve S_{12}^{(\ell'-1)}b_{\ell'},
		\end{equation}
		\Else
		\begin{equation}
		\sigma_{\ell\ell'}^{12}(k_{\rho})=T^{\ell\ell+1}_{11}\sigma_{\ell+1,\ell'}^{12}+T^{\ell\ell+1}_{12}\sigma_{\ell+1,\ell'}^{22},
		\end{equation}
		\EndIf
		\If{$\ell>0$}
		\If{$\ell>\ell'$}
		\begin{equation}
		\sigma_{\ell\ell'}^{21}(k_{\rho})=\frac{-1}{\alpha_{22}^{(\ell)}}\begin{pmatrix}
		0 & 1
		\end{pmatrix}\left[\frac{C^{(\ell'+1)}}{C^{(\ell)}}\mathbb A^{(\ell')}2e_{\ell'}\breve{\mathbb S}^{(\ell')}\begin{pmatrix}
		-a_{\ell'}\\
		b_{\ell'}
		\end{pmatrix}+\mathbb A^{(\ell)}\begin{pmatrix}
		\sigma_{\ell\ell'}^{11}\\
		0
		\end{pmatrix}\right],
		\end{equation}
		\Else
		\begin{equation}
		\sigma_{\ell\ell'}^{21}(k_{\rho})=-\alpha_{21}^{(\ell)}\sigma_{\ell\ell'}^{11}(k_{\rho})/\alpha_{22}^{(\ell)},
		\end{equation}
		\EndIf
		\If{$\ell>\ell'-1$}
		\begin{equation}
		\sigma_{\ell\ell'}^{22}(k_{\rho})=\frac{-1}{\alpha_{22}^{(\ell)}}\begin{pmatrix}
		0 & 1
		\end{pmatrix}\left[\frac{C^{(\ell')}}{C^{(\ell)}}\mathbb A^{(\ell'-1)}2e_{\ell'-1}\breve{\mathbb S}^{(\ell'-1)}\begin{pmatrix}
		a_{\ell'}\\
		b_{\ell'}
		\end{pmatrix}+\mathbb A^{(\ell)}\begin{pmatrix}
		\sigma_{\ell\ell'}^{12}\\
		0
		\end{pmatrix}\right],
		\end{equation}
		\Else
		\begin{equation}\label{upperlayerdensity}
		\sigma_{\ell\ell'}^{22}(k_{\rho})=-\alpha_{21}^{(\ell)}\sigma_{\ell\ell'}^{12}(k_{\rho})/\alpha_{22}^{(\ell)},
		\end{equation}
		\EndIf
		\EndIf
		\EndFor
		\EndFor
	\end{algorithmic}
\end{algorithm}

According to the formulas used in the {\bf Algorithm 1}, we are able to prove some important properties of the reaction densities, which will play a key role in the analysis in the rest of this paper. First, we have the following lemma for the matrices $\mathbb A^{(\ell)}$ defined in \eqref{transmissionmat}.
\begin{lemma}\label{matrixentrylemma}
	Suppose $a_{\ell}, b_{\ell}>0$ for all $\ell=0, 1, \cdots, L$, then the entries in the second row of the matrices $\mathbb A^{(\ell)}$ satisfy
	\begin{equation}\label{keyinequality}
	|\alpha_{22}^{(\ell)}|^2-|\alpha_{21}^{(\ell)}|^2\geq (|\gamma_1^+|^2-|\gamma_1^-|^2)(|\gamma_2^+|^2-|\gamma_2^-|^2)\cdots(|\gamma_{\ell}^+|^2-|\gamma_{\ell}^-|^2)>0,
	\end{equation}
	for $\ell=1, 2, \cdots, L,$ and any $k_{\rho}\in\{z\in\mathbb C|\mathfrak{Re}z\geq 0\}$.
\end{lemma}
\begin{proof}
	By the definition of $\widetilde{\mathbb T}^{\ell-1,\ell}$, $\mathbb A^{(\ell)}$ in \eqref{transmissionmat}, the entries $\{\alpha_{21}^{(\ell)}, \alpha_{22}^{(\ell)}\}_{\ell=1}^{L}$ can be calculated recursively as
	\begin{equation}\label{alpharecursion}
	\begin{split}
	\alpha_{21}^{(1)}=e_1\gamma_1^-,\quad \alpha_{22}^{(1)}=\gamma_1^+,\\
	\alpha_{21}^{(2)}=\alpha_{21}^{(1)}\gamma_2^+e_1e_2+\alpha_{22}^{(1)}\gamma_2^-e_2,\quad \alpha_{22}^{(2)}=\alpha_{21}^{(1)}\gamma_2^-e_1+\alpha_{22}^{(1)}\gamma_2^+,\\
	\cdots\cdots\cdots\\
	\alpha_{21}^{(\ell)}=\alpha_{21}^{(\ell-1)}\gamma_{\ell}^+e_{\ell-1}e_{\ell}+\alpha_{22}^{(\ell-1)}\gamma_{\ell}^-e_{\ell},\quad \alpha_{22}^{(\ell)}=\alpha_{21}^{(\ell-1)}\gamma_{\ell}^-e_{\ell-1}+\alpha_{22}^{(\ell-1)}\gamma_{\ell}^+.
	\end{split}
	\end{equation}
	Naturally, we will prove the conclusion \eqref{keyinequality} by induction. As  $a_{\ell}, b_{\ell}>0$ by assumption and $|e_{\ell}|\leq 1$ for all  $k_{\rho}\in\{z\in\mathbb C|\mathfrak{Re}z\geq 0\}$, we obtain
	\begin{equation}
	|\gamma_{\ell}^+|=\Big|\frac{a_{\ell}}{a_{\ell-1}}+\frac{b_{\ell}}{b_{\ell-1}}\Big|>\Big|\frac{a_{\ell}}{a_{\ell-1}}-\frac{b_{\ell}}{b_{\ell-1}}\Big|=|\gamma^{-}_{\ell}|\geq|\gamma_{\ell}^{-}e_{\ell}|,\quad {\rm if}\;\;\mathfrak{Re}k_{\rho}\geq 0.
	\end{equation}
	Therefore, \eqref{keyinequality} is true for $\ell=1$ as
	\begin{equation}
	\big|\alpha_{22}^{(1)}\big|^2-\big|\alpha_{21}^{(1)}\big|^2=|\gamma_{1}^+|^2-|\gamma_{1}^{-}e_{1}|^2\geq |\gamma_{1}^+|^2-|\gamma_{1}^{-}|^2>0.
	\end{equation}
	Assume
	\begin{equation}\label{assumption}
	\big|\alpha_{22}^{(s)}\big|^2-\big|\alpha_{21}^{(s)}\big|^2\geq(|\gamma_1^+|^2-|\gamma_1^-|^2)(|\gamma_2^+|^2-|\gamma_2^-|^2)\cdots(|\gamma_s^+|^2-|\gamma_s^-|^2),
	\end{equation}
	is true for all $s=2, 3, \cdots, \ell-1$. By recursion \eqref{alpharecursion}, we have
	\begin{equation}\label{modules}
	\big|\alpha_{21}^{(\ell)}\big|=|\beta_{\ell-1}\gamma_{\ell}^+e_{\ell-1}+\gamma_{\ell}^-|\big|\alpha_{22}^{(\ell-1)}e_{\ell}\big|,\quad \big|\alpha_{22}^{(\ell)}\big|=|\beta_{\ell-1}\gamma_{\ell}^-e_{\ell-1}+\gamma_{\ell}^+||\alpha_{22}^{(\ell-1)}|,
	\end{equation}
	where $\beta_{\ell}:=\alpha_{21}^{(\ell)}/\alpha_{22}^{(\ell)}$. Noting that $\gamma_{\ell}^{\pm}$ are real, then
	\begin{equation*}
	\begin{split}
	|\beta_{\ell-1}\gamma_{\ell}^+e_{\ell-1}+\gamma_{\ell}^-|^2=|\beta_{\ell-1}e_{\ell-1}|^2(\gamma_{\ell}^+)^2+2\gamma_{\ell}^+\gamma_{\ell}^-\mathfrak{Re}\{\beta_{\ell-1}e_{\ell-1}\}+(\gamma_{\ell}^-)^2,\\
	|\beta_{\ell-1}\gamma_{\ell}^-e_{\ell-1}+\gamma_{\ell}^+|^2=|\beta_{\ell-1}e_{\ell-1}|^2(\gamma_{\ell}^-)^2+2\gamma_{\ell}^+\gamma_{\ell}^-\mathfrak{Re}\{\beta_{\ell-1}e_{\ell-1}\}+(\gamma_{\ell}^+)^2.
	\end{split}
	\end{equation*}
	Therefore
	\begin{equation*}
	|\beta_{\ell-1}\gamma_{\ell}^-e_{\ell-1}+\gamma_{\ell}^+|^2-|\beta_{\ell-1}\gamma_{\ell}^+e_{\ell-1}+\gamma_{\ell}^-|^2=[(\gamma_{\ell}^+)^2-(\gamma_{\ell}^-)^2](1-|\beta_{\ell-1}e_{\ell-1}|^2).
	\end{equation*}
	Together with \eqref{modules} and the fact $|e_{\ell}|\leq 1$ for all  $k_{\rho}\in\{z\in\mathbb C|\mathfrak{Re}z\geq 0\}$, we obtain
	\begin{equation*}
	\big|\alpha_{22}^{(\ell)}\big|^2-\big|\alpha_{21}^{(\ell)}\big|^2\geq[(\gamma_{\ell}^+)^2-(\gamma_{\ell}^-)^2](1-|\beta_{\ell-1}|^2)\big|\alpha_{22}^{(\ell-1)}\big|^2=[(\gamma_{\ell}^+)^2-(\gamma_{\ell}^-)^2](\big|\alpha_{22}^{(\ell-1)}\big|^2-\big|\alpha_{21}^{(\ell-1)}\big|^2).
	\end{equation*}
	Then, we complete the proof by applying the assumption \eqref{assumption}.
\end{proof}
\begin{prop}\label{Prop:densityprop}
	Suppose $a_{\ell}, b_{\ell}>0$ for all $\ell=0, 1, \cdots, L$, then all reaction densities  $\sigma_{\ell\ell'}^{\mathfrak{ab}}(k_{\rho})$ in \eqref{generalcomponents} are continuous and bounded in $\{k_{\rho}|\mathfrak{Re}k_{\rho}\geq 0\}$. Moreover, they are analytic in the right half complex plane $\{k_{\rho}|\mathfrak{Re}k_{\rho}>0\}$
\end{prop}
\begin{proof}
	From the definition \eqref{expdef} and \eqref{transmissionmat}, we have
	\begin{equation*}
	\begin{split}
	&T^{\ell\ell+1}_{11}=\frac{a_{\ell+1}b_{\ell}+a_{\ell}b_{\ell+1}}{2a_{\ell}b_{\ell}}e_{\ell+1},\quad  T^{\ell\ell+1}_{12}= \frac{a_{\ell+1}b_{\ell}-a_{\ell}b_{\ell+1}}{2a_{\ell}b_{\ell}},\quad \widetilde{\mathbb T}^{\ell-1,\ell}=\begin{pmatrix}
	\displaystyle \gamma_{\ell}^+e_{\ell-1}e_{\ell} & \displaystyle\gamma_{\ell}^-e_{\ell-1}\\[7pt]
	\displaystyle \gamma_{\ell}^-e_{\ell} & \displaystyle\gamma_{\ell}^+
	\end{pmatrix},\\
	&2e_{\ell}\breve{\mathbb S}^{(\ell)}=\begin{pmatrix}
	\displaystyle a_{\ell}^{-1}e_{\ell} & \displaystyle b_{\ell}^{-1}e_{\ell}\\
	\displaystyle a_{\ell}^{-1} & \displaystyle-b_{\ell}^{-1}
	\end{pmatrix},\quad
	\frac{C^{(\ell_1)}}{C^{(\ell_2)}}=\begin{cases}
	1 & \ell_1=\ell_2,\\
	2^{\ell_2-\ell_1}e^{-k_{\rho}(d_{\ell_1-1}-d_{\ell_2-1})} & 0\leq\ell_1<\ell_2.
	\end{cases}
	\end{split}
	\end{equation*}
	As they consist of constants $\{a_{\ell}, b_{\ell}\}_{\ell=0}^L$ and their product with exponential functions of $k_{\rho}$, $\{T_{11}^{\ell\ell+1}, T_{12}^{\ell\ell+1}\}_{\ell=0}^{L-1}$, $\{C^{(\ell_1)}/C^{(\ell_2)}\}_{\ell_1\leq \ell_2}$ and the entries of matrices  $\{\widetilde{\mathbb T}^{\ell-1,\ell}\}_{\ell=1}^L$, $\{2e_{\ell}\breve{\mathbb S}^{(\ell)}\}_{\ell=0}^L$ and $\mathbb A^{(\ell)}=\widetilde{\mathbb T}^{01}\cdots \widetilde{\mathbb T}^{\ell-1,\ell}$ are all continuous and bounded in $\{k_{\rho}|\mathfrak{Re}k_{\rho}\geq 0\}$. Moreover, by lemma \ref{matrixentrylemma}, the module of the denominators $\{\alpha_{22}^{(\ell)}\}_{\ell=1}^L$ in \eqref{bottommostdensity1}-\eqref{upperlayerdensity} are bounded below in $\{k_{\rho}|\mathfrak{Re}k_{\rho}\geq 0\}$ by some positive constants determined by $\{a_{\ell}, b_{\ell}\}_{\ell=0}^L$.
	Therefore, checking the formulas \eqref{bottommostdensity1}-\eqref{upperlayerdensity} with the discussions above, it is not difficult to conclude that all reaction densities are continuous and bounded in $\{k_{\rho}|\mathfrak{Re}k_{\rho}\geq 0\}$ and analytic in the right half complex plane.	
\end{proof}

\section{Multipole and local expansions, shifting and translation operators for the Green's function of 3-dimensional Laplace equation in layered media}
In this section, we review the derivation of the ME, LEs and shifting and translation operators used in the FMM for charge interaction in multi-layered media(cf. \cite{wang2019fastlaplace}).

\subsection{Multipole and local expansions, shifting and translation operators for free space Green's function} According to the expression \eqref{layeredGreensfun}, the layered media Green's function consists of free space and reaction field components. The FMM for Laplace equation in layered media (cf. \cite{wang2019fastlaplace}) is a combination of the classic FMM and new FMMs for the free space and reaction field components, respectively. Before we go to the expansions for the reaction field components, let us first review the classic formulas on which the classic FMM rely.

Given source and target centers $\bs r_c^s$ and $\bs r_c^t$ close to source $\bs r'$ and target $\bs r$, i.e,  $|\bs r'-\bs r_c^s|<|\bs r-\bs r_c^s|$ and $|\bs r'-\bs r_c^t|>|\bs r-\bs r_c^t|$, the free space Green's function has Taylor expansions
\begin{equation}\label{expansionbeforeme}
\frac{1}{4\pi|\bs r-\bs r'|}=\frac{1}{4\pi|(\bs r-\bs r_c^s)-(\bs r'-\bs r_c^s)|}=\frac{1}{4\pi}\sum\limits_{n=0}^{\infty}\frac{P_n(\cos\gamma_s)}{ r_s}\Big(\frac{r'_s}{ r_s}\Big)^n,
\end{equation}
and
\begin{equation}\label{localexpansionbeforeme}
\frac{1}{4\pi|\bs r-\bs r'|}=\frac{1}{4\pi|(\bs r-\bs r_c^t)-(\bs r'-\bs r_c^t)|}=\frac{1}{4\pi}\sum\limits_{n=0}^{\infty}\frac{P_n(\cos\gamma_t)}{r_t'}\Big(\frac{r_t}{r_t'}\Big)^n,
\end{equation}
where  $(r_s, \theta_s,\varphi_s)$, $(r_t,\theta_t,\varphi_t)$ are the spherical coordinates of $\bs r-\bs r_c^s$ and $\bs r-\bs r_c^t$, $(r'_s, \theta'_s,\varphi'_s)$, $(r'_t,\theta'_t,\varphi'_t)$ are the spherical coordinates of $\bs r'-\bs r_c^s$ and $\bs r'-\bs r_c^t$( see Fig. \ref{3dspherical}) and
\begin{equation}
\begin{split}
&\cos\gamma_s=\cos\theta_s\cos\theta'_s+\sin\theta_s\sin\theta'_s\cos(\varphi_s-\varphi'_s),\\
&\cos\gamma_t=\cos\theta_t\cos\theta_t'+\sin\theta_t\sin\varphi_t'\cos(\varphi_t-\varphi_t').
\end{split}
\end{equation}
Moreover, the following error estimates
\begin{equation}\label{meerror}
\left|\frac{1}{4\pi|\bs r-\bs r'|}-\frac{1}{4\pi}\sum\limits_{n=0}^{p}\frac{P_n(\cos\gamma_s)}{ r_s}\Big(\frac{r'_s}{ r_s}\Big)^n\right|\leq \frac{1}{4\pi(r_s-r_s')}\Big(\frac{r'_s}{r_s}\Big)^{p+1}, \quad r_s>r_s',
\end{equation}
and
\begin{equation}\label{leerror}
\left|\frac{1}{4\pi|\bs r-\bs r'|}-\frac{1}{4\pi}\sum\limits_{n=0}^{p}\frac{P_n(\cos\gamma_t)}{r_t'}\Big(\frac{r_t}{r_t'}\Big)^n\right|\leq\frac{1}{4\pi(r_t'-r_t)}\Big(\frac{r_t}{r_t'}\Big)^{p+1}, \quad r_t<r_t',
\end{equation}
for any $p\geq 1$ can be obtained by using the fact $|P_n(x)|\leq 1$ for all $x\in[-1, 1]$.
\begin{figure}[ht!]\label{3dspherical}
	\centering
	\includegraphics[scale=1.1]{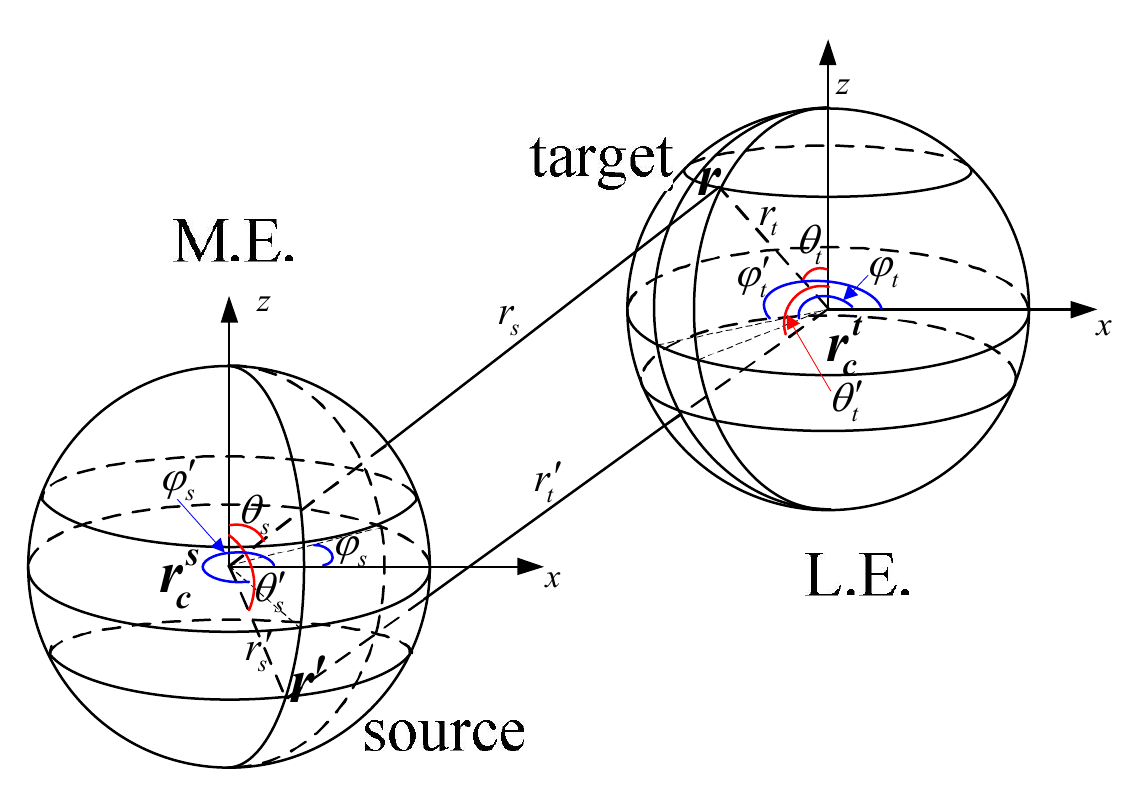}
	\caption{Spherical coordinates used in multipole and local expansions.}
\end{figure}

Note that $P_n(\cos\gamma_s)$, $P_n(\cos\gamma_t)$ still mix the source and target information ($\bs r$ and $\bs r'$) together. The following addition theorems (cf. \cite{grengard1988rapid, epton1995multipole}) will be used to derive source/target separated ME, LE and corresponding shifting and translation operators. As in \cite{wang2019fastlaplace}, we will  re-present the theorems using scaled spherical harmonics
\begin{equation}\label{sphericalharmonics}
Y_n^m(\theta,\varphi)=(-1)^m\sqrt{\frac{2n+1}{4\pi}\frac{(n-m)!}{(n+m)!}}P_n^{m}(\cos\theta)e^{\ri m\varphi}:=\widehat P_n^{m}(\cos\theta)e^{\ri m\varphi},
\end{equation}
where $P_n^m(x)$ (resp. $\widehat P_n^m(x)$) is the associated (resp. normalized) Legendre function of degree $n$ and order $m$. We also use notations
\begin{equation}\label{constantdef}
c_n=\sqrt{\frac{2n+1}{4\pi}},\quad A_n^m=\frac{(-1)^nc_n}{\sqrt{(n-m)!(n+m)!}},\quad |m|\leq  n,
\end{equation}
in the rest part of this paper.
\begin{theorem}\label{addthmleg}
	{\bf (Addition theorem for Legendre polynomials)} Let $P$ and $Q$ be points with spherical coordinates $(r,\theta,\varphi)$ and $(\rho,\alpha,\beta)$, respectively, and let $\gamma$ be the angle subtended between them. Then
	\begin{equation}
	P_n(\cos\gamma)=\frac{4\pi}{2n+1}\sum\limits_{m=-n}^n\overline{Y_n^{m}(\alpha,\beta)}Y_n^m(\theta,\varphi).
	\end{equation}
\end{theorem}

\begin{theorem}\label{theorem:firstaddition}
	Let $Q=(\rho,\alpha,\beta)$ be the center of expansion of an arbitrary spherical harmonic of negative degree. Let the point $P=(r,\theta,\varphi)$, with $r>\rho$, and $P-Q=(r', \theta', \varphi')$. Then
	\begin{equation*}
	\frac{Y_{n'}^{m'}(\theta', \varphi')}{r'^{n'+1}}=\sum\limits_{n=0}^{\infty}\sum\limits_{m=-n}^n\frac{(-1)^{|m+m'|-|m'|}A_n^mA_{n'}^{m'}\rho^nY_n^{-m}(\alpha,\beta)}{c_n^2A_{n+n'}^{m+m'}}\frac{Y_{n+n'}^{m+m'}(\theta,\varphi)}{r^{n+n'+1}}.
	\end{equation*}
\end{theorem}
\begin{theorem}\label{theorem:secondaddition}
	Let $Q=(\rho,\alpha,\beta)$ be the center of expansion of an arbitrary spherical harmonic of negative degree. Let the point $P=(r,\theta,\varphi)$, with $r<\rho$, and $P-Q=(r', \theta', \varphi')$. Then
	\begin{equation*}
	\frac{Y_{n'}^{m'}(\theta', \varphi')}{r'^{n'+1}}=\sum\limits_{n=0}^{\infty}\sum\limits_{m=-n}^n\frac{(-1)^{n'+|m|}A_n^mA_{n'}^{m'}\cdot Y_{n+n'}^{m'-m}(\alpha,\beta)}{c_n^2A_{n+n'}^{m'-m}\rho^{n+n'+1}}r^nY_n^{m}(\theta,\varphi).
	\end{equation*}
\end{theorem}
\begin{theorem}\label{theorem:fourthaddition}
	Let $Q=(\rho,\alpha,\beta)$ be the center of expansion of an arbitrary spherical harmonic of negative degree. Let the point $P=(r,\theta,\varphi)$ and $P-Q=(r', \theta', \varphi')$. Then
	\begin{equation*}
	r'^{n'}Y_{n'}^{m'}(\theta', \varphi')=\sum\limits_{n=0}^{n'}\sum\limits_{m=-n}^n\frac{(-1)^{n-|m|+|m'|-|m'-m|}c_{n'}^2 A_n^mA_{n'-n}^{m'-m}\cdot \rho^nY_n^{m}(\alpha,\beta) }{c_{n}^2c_{n'-n}^2A_{n'}^{m'}r^{n-n'}}Y_{n'-n}^{m'-m}(\theta,\varphi).
	\end{equation*}	
\end{theorem}
In the above theorems, the definition $A_n^m=0$, $Y_n^m(\theta,\varphi)\equiv 0$ for $|m|>n$ is used.

Applying Legendre addition theorem to expansions \eqref{expansionbeforeme} and \eqref{localexpansionbeforeme} gives ME
\begin{equation}\label{mefreespace}
\frac{1}{4\pi|\bs r-\bs r'|}=\sum\limits_{n=0}^{\infty}\sum\limits_{m=-n}^nM_{nm}\frac{Y_n^m(\theta_s,\varphi_s)}{r_s^{n+1}},
\end{equation}
and LE
\begin{equation}\label{lefreespace}
\frac{1}{4\pi|\bs r-\bs r'|}=\sum\limits_{n=0}^{\infty}\sum\limits_{m=-n}^nL_{nm}r^n_tY_n^m(\theta_t,\varphi_t),
\end{equation}
where
\begin{equation}\label{freespaceexpcoef}
M_{nm}=\frac{1}{4\pi c_n^2}r'^n_s\overline{Y_n^{m}(\theta'_s,\varphi'_s)},\quad L_{nm}=\frac{1}{4\pi c_n^2}r'^{-n-1}_t\overline{Y_n^{m}(\theta'_t,\varphi'_t)}.
\end{equation}

Further, applying Theorem \ref{theorem:secondaddition} in ME \eqref{mefreespace} provides a translation from ME \eqref{mefreespace} to LE \eqref{lefreespace} which is given by
\begin{equation}\label{metole}
L_{nm}=\sum\limits_{\nu=0}^{\infty}\sum\limits_{\mu=-\nu}^{\nu}\frac{(-1)^{\nu+|m|}A_{\nu}^{\mu}A_n^mY_{n+\nu}^{\mu-m}(\theta_{st}, \varphi_{st})}{c_{\nu}^2A_{n+\nu}^{\mu-m}r_{st}^{n+\nu+1}}M_{\nu\mu},
\end{equation}
where $(r_{st}, \theta_{st}, \varphi_{st})$ is the spherical coordinate of $\bs r_c^s-\bs r_c^t$.

Given two new centers $\tilde{\bs r}_c^s$ and $\tilde{\bs r}_c^t$ close to $\bs r_c^s$ and $\bs r_c^t$, respectively. By using the addition Theorems \ref{theorem:firstaddition} and \ref{theorem:fourthaddition} in \eqref{mefreespace}-\eqref{lefreespace} and rearranging terms in the results, we obtain
\begin{equation*}
\begin{split}
&\sum\limits_{\nu=0}^{\infty}\sum\limits_{\mu=-\nu}^{\nu}M_{\nu\mu}\frac{Y_{\nu}^{\mu}(\theta_s,\varphi_s)}{r_s^{\nu+1}}\\
=&\sum\limits_{\nu=0}^{\infty}\sum\limits_{\mu=-\nu}^{\nu}\sum\limits_{n'=0}^{\infty}\sum\limits_{m'=-n'}^{n'}M_{\nu\mu}\frac{(-1)^{|m'+\mu|-|\mu|}A_{n'}^{m'}A_{\nu}^{\mu}r_{ss}^{n'}Y_{n'}^{-m'}(\theta_{ss},\varphi_{ss})}{c_{n'}^2A_{n'+\nu}^{m'+\mu}}\frac{Y_{n'+\nu}^{m'+\mu}(\tilde\theta_s, \tilde\varphi_s)}{\tilde r_s^{n'+\nu+1}}\\
=&\sum\limits_{n=0}^{\infty}\sum\limits_{m=-n}^{n}\sum\limits_{\nu=0}^{n}\sum\limits_{\mu=-\nu}^{\nu}M_{\nu\mu}\frac{(-1)^{|m|-|\mu|}A_{n-\nu}^{m-\mu}A_{\nu}^{\mu}r_{ss}^{n-\nu}Y_{n-\nu}^{\mu-m}(\theta_{ss},\varphi_{ss})}{c_{n-\nu}^2A_{n}^{m}}\frac{Y_{n}^{m}(\tilde\theta_s, \tilde\varphi_s)}{\tilde r_s^{n+1}},
\end{split}
\end{equation*}
and
\begin{equation*}
\begin{split}
&\sum\limits_{\nu=0}^{\infty}\sum\limits_{\mu=-\nu}^{\nu}L_{\nu\mu}r^{\nu}_tY_{\nu}^{\mu}(\theta_t,\varphi_t)\\
=&\sum\limits_{\nu=0}^{\infty}\sum\limits_{\mu=-\nu}^{\nu}\sum\limits_{n'=0}^{\nu}\sum\limits_{m'=-n'}^{n'}L_{\nu\mu} \frac{(-1)^{n'-|m'|+|\mu|-|\mu-m'|}c_{\nu}^2A_{n'}^{m'}A_{\nu-n'}^{\mu-m'}r_{tt}^{n'}Y_{n'}^{m'}(\theta_{tt},\varphi_{tt})}{c_{n'}^2c_{\nu-n'}^2A_{\nu}^{\mu}\tilde r_t^{n'-\nu}}Y_{\nu-n'}^{\mu-m'}(\tilde\theta_t, \tilde\varphi_t)\\
=&\sum\limits_{n=0}^{\infty}\sum\limits_{m=-n}^{n}\sum\limits_{\nu=n}^{\infty}\sum\limits_{\mu=-\nu}^{\nu}L_{\nu\mu}\frac{(-1)^{\nu-n-|\mu-m|+|\mu|-|m|}c_{\nu}^2A_{\nu-n}^{\mu-m}A_{n}^{m}r_{tt}^{\nu-n}Y_{\nu-n}^{\mu-m}(\theta_{tt},\varphi_{tt})}{c_{\nu-n}^2c_{n}^2A_{\nu}^{\mu}}\tilde r_t^{n}Y_{n}^{m}(\tilde\theta_t, \tilde\varphi_t),
\end{split}
\end{equation*}
where $(\tilde r_s, \tilde\theta_s, \tilde\varphi_s)$, $(\tilde r_t, \tilde\theta_t, \tilde\varphi_t)$, $(r_{ss}, \theta_{ss}, \varphi_{ss})$ and $(r_{tt},\theta_{tt}, \varphi_{tt})$ are the spherical coordinates of $\bs r-\tilde{\bs r}_c^s$, $\bs r-\tilde{\bs r}_c^t$, $\bs r_c^s-\tilde{\bs r}_c^s$ and $\bs r_c^t-\tilde{\bs r}_c^t$, $(\tilde r_s, \tilde\theta_s, \tilde\varphi_s)$. The above formulas implies that the coefficients
\begin{equation}
\tilde{M}_{nm}=\frac{1}{4\pi c_n^2}\tilde r'^n_s\overline{Y_{n}^{m}(\tilde\theta'_s,\tilde\varphi'_s)},\quad \tilde L_{nm}=\frac{1}{4\pi c_n^2}\tilde r'^{-n-1}_t\overline{Y_{n}^{m}(\tilde\theta'_t,\tilde\varphi'_t)},
\end{equation}
of the shifted ME and LE at new centers $\tilde{\bs r}_c^t$ and $\tilde{\bs r}_c^{s}$ can be obtained via center shifting
\begin{align}
\displaystyle\tilde{M}_{nm}
=&\sum\limits_{\nu=0}^{n}\sum\limits_{\mu=-\nu}^{\nu}\frac{(-1)^{|m|-|\mu|}A_{n-\nu}^{m-\mu}A_{\nu}^{\mu}r_{ss}^{n-\nu}Y_{n-\nu}^{\mu-m}(\theta_{ss},\varphi_{ss})}{c_{n-\nu}^2A_{n}^{m}}M_{\nu\mu},\label{metome}\\
\displaystyle\tilde L_{nm}=&\sum\limits_{\nu=n}^{\infty}\sum\limits_{\mu=-\nu}^{\nu} \frac{(-1)^{\nu-n-|\mu-m|+|\mu|-|m|}c_{\nu}^2A_{\nu-n}^{\mu-m}A_n^mr_{tt}^{\nu-n}Y_{\nu-n}^{\mu-m}(\theta_{tt}, \varphi_{tt})}{c_{\nu-n}^2c_n^2A_{\nu}^{\mu}}L_{\nu\mu}.\label{letole}
\end{align}

Besides using the addition theorems, we have proposed a new derivation for \eqref{mefreespace} and \eqref{lefreespace} by using the integral representation of $1/|\bs r-\bs r'|$.  Moreover, the methodology has been further applied to derive multipole and local expansions for the reaction components of the Green's function in layered media (cf. \cite{wang2019fastlaplace}).

\subsection{Multipole expansions for general reaction component} Consider a general reaction component $u_{\ell\ell'}^{\mathfrak{ab}}(\bs r, \bs r')$ given in \eqref{generalcomponents}. By inserting the source center $\bs r^s_c=(x_c^{s},y_c^{s}, z_c^{s} )$, the exponential kernels in \eqref{generalcomponents} have source/target separations
\begin{equation}\label{sourcetargetseparationsc1}
e^{\ri\bs k\cdot\tau^{\mathfrak a 1}_{\ell\ell'}(\bs r, \bs r')}=e^{\ri\bs k\cdot\bs\tau_{\ell\ell'}^{\mathfrak a 1}(\bs r, \bs r^s_c)}e^{\ri\bs k\cdot\bs\tau(-\bs r_s')}\quad
e^{\ri\bs k\cdot\bs\tau_{\ell\ell'}^{\mathfrak a 2}(\bs r, \bs r')}=e^{\ri\bs k\cdot\bs\tau_{\ell\ell'}^{\mathfrak a 2}(\bs r, \bs r^s_c)}e^{-\ri\bs k\cdot\bs r_s'},\quad \mathfrak a, \mathfrak b=1, 2.
\end{equation}
 Here, $\bs r_s'=\bs r'-\bs r_c^s$, $\bs\tau(\bs r)=(x, y, -z)$ is the reflection of any $\bs r=(x, y, z)\in \mathbb R^3$ according to $xy$-plane. Obviously, the reflection $\bs\tau(\bs r)$ satisfies
\begin{equation}\label{reflection}
|\bs\tau(\bs r)|=|\bs r|,\quad \bs\tau(\bs r+\bs r')=\bs\tau(\bs r)+\bs\tau(\bs r'),\quad\bs\tau(a\bs r)=a\bs\tau(\bs r),\quad \forall \bs r,\bs r'\in\mathbb R^3, \;\forall a\in\mathbb R.
\end{equation}
Moreover, applying source/target separations \eqref{sourcetargetseparationsc1} and the Taylor expansions
\begin{equation*}
\begin{split}
e^{\ri\bs k\cdot\bs\tau^{\mathfrak a 1}_{\ell\ell'}(\bs r, \bs r')}=e^{\ri\bs k\cdot\bs\tau_{\ell\ell'}^{\mathfrak a 1}(\bs r, \bs r^s_c)}\sum\limits_{n=0}^{\infty}\frac{[\ri\bs k\cdot\tau(-\bs r_s')]^n}{n!},\quad
e^{\ri\bs k\cdot\bs\tau_{\ell\ell'}^{\mathfrak a 2}(\bs r, \bs r')}=e^{\ri\bs k\cdot\bs\tau_{\ell\ell'}^{\mathfrak a 2}(\bs r, \bs r^s_c)}\sum\limits_{n=0}^{\infty}\frac{[-\ri\bs k\cdot\bs r_s']^n}{n!},
\end{split}
\end{equation*}
in \eqref{generalcomponents}  gives expansions
\begin{equation}\label{melayerupgoingtaylor}
\begin{split}
u_{\ell\ell'}^{\mathfrak{a}1}(\bs r, \bs r')=&\sum\limits_{n=0}^{\infty}\frac{1}{8\pi^2 }\int_{-\infty}^{\infty}\int_{-\infty}^{\infty}k_{\rho}^{n-1}e^{\ri\bs k\cdot\bs\tau_{\ell\ell'}^{\mathfrak a 1}(\bs r, \bs r^s_c)}\frac{[\ri\bs k\cdot{\bs\tau(-\bs r_s')}]^n}{n!}\sigma_{\ell\ell'}^{\mathfrak{a}1}(k_{\rho})dk_x dk_y, \\
u_{\ell\ell'}^{\mathfrak{a}2}(\bs r, \bs r')=&\sum\limits_{n=0}^{\infty}\frac{1}{8\pi^2 }\int_{-\infty}^{\infty}\int_{-\infty}^{\infty}k_{\rho}^{n-1}e^{\ri\bs k\cdot\bs\tau_{\ell\ell'}^{\mathfrak a 2}(\bs r, \bs r^s_c)}\frac{[-\ri\bs k\cdot{\bs r}_s']^n}{n!}\sigma_{\ell\ell'}^{\mathfrak{a}2}(k_{\rho})dk_x dk_y,
\end{split}
\end{equation}
for $\mathfrak a=1, 2$. Here, we have directly exchanged the order of the infinite summations and improper integrals. Rigorous theoretical proof will be presented in Theorem \ref{Thm:reactmeconvergence}.

To derive ME for the general reaction component $u_{\ell\ell'}^{\mathfrak{ab}}(\bs r, \bs r')$, we shall use the following limit version of the extended Legendre addition theorem (cf. \cite{wang2019fastlaplace}).
\begin{theorem}\label{lemma3}
	Let $\bs k_0=(\cos\alpha, \sin\alpha, \ri)$ be a vector with complex entry, $\theta, \varphi$ be the azimuthal angle and polar angles of a unit vector $\hat{\bs r}$. Then
	\begin{equation}\label{polynomialadd}
	\frac{(\ri\bs k_0\cdot\hat{\bs r})^n}{n!}=\sum\limits_{m=-n}^nC_n^m\widehat P_n^m(\cos\theta)e^{\ri m(\alpha-\varphi)},
	\end{equation}
	where
	\begin{equation}\label{constantcnm}
	C_n^m=\ri^{2n-m}\sqrt{\frac{4\pi}{(2n+1)(n+m)!(n-m)!}}.
	\end{equation}
\end{theorem}


By applying Theorem \ref{lemma3} in expansions \eqref{melayerupgoingtaylor}  and then using identities
\begin{equation}\label{sphrelations}
Y_n^{m}(\pi-\theta,\varphi)=(-1)^{n+m}Y_n^{m}(\theta,\varphi),\quad Y_n^{m}(\theta,\pi+\varphi)=(-1)^{m}Y_n^{m}(\theta,\varphi),
\end{equation}
we obtain MEs
\begin{equation}\label{melayerupgoingimage1}
\begin{split}
u_{\ell\ell'}^{\mathfrak{ab}}(\bs r, \bs r')=\sum\limits_{n=0}^{\infty}\sum\limits_{m=-n}^{n}  M_{nm}^{\mathfrak{ab}}{\mathcal F}_{nm}^{\mathfrak{ab}}(\bs r, \bs r_c^s), \quad M_{nm}^{\mathfrak{ab}}=\frac{1}{4\pi c_n^2} r_s'^n\overline{Y_n^{m}(\theta_s',\varphi_s')},
\end{split}
\end{equation}
at source centers $\bs r_c^s$. Here, ${\mathcal F}_{nm}^{\mathfrak{ab}}(\bs r, \bs r_c^s)$ are represented by Sommerfeld-type integrals
\begin{equation}\label{mebasis1}
\begin{split}
{\mathcal F}_{nm}^{\mathfrak a1}(\bs r, \bs r_c^s)=&\frac{(-1)^{m}c_n^2C_n^m}{2\pi}\int_{-\infty}^{\infty}\int_{-\infty}^{\infty}e^{\ri\bs k\cdot\bs\tau_{\ell\ell'}^{\mathfrak a1}(\bs r, \bs r_c^s)}\sigma_{\ell\ell'}^{1\mathfrak b}(k_{\rho})k_{\rho}^{n-1}e^{\ri m\alpha}dk_x dk_y,\\
{\mathcal F}_{nm}^{\mathfrak a2}(\bs r, \bs r_c^s)=&\frac{(-1)^nc_n^2C_n^m}{2\pi}\int_{-\infty}^{\infty}\int_{-\infty}^{\infty}e^{\ri\bs k\bs\tau_{\ell\ell'}^{\mathfrak a2}(\bs r, \bs r_c^s)}\sigma_{\ell\ell'}^{2\mathfrak b}(k_{\rho})k_{\rho}^{n-1}e^{\ri m\alpha}dk_x dk_y.
\end{split}
\end{equation}

\section{Exponential convergence of the MEs and LEs, shifting and translation operators}
In this section, we prove the exponential convergence of the approximations used in the FMM for 3-dimensional Laplace equation in layered media. Let $\mathscr{P}_{\ell}=\{(Q_{\ell j},\boldsymbol{r}_{\ell j}),$ $j=1,2,\cdots
,N_{\ell}\}$, $\ell=0, 1, \cdots, L$ be $L$ groups of source charges distributed in a multi-layer medium with $L+1$ layers (see Fig. \ref{layerstructure}). The group of charges in $\ell$-th layer is denoted by $\mathscr{P}_{\ell}$. The FMM provides a fast algorithm to compute interactions
\begin{equation}\label{totalinteraction}
\begin{split}
\Phi_{\ell}(\boldsymbol{r}_{\ell i})
=&\Phi_{\ell}^{\text{free}}(\boldsymbol{r}_{\ell i})+\sum\limits_{\ell^{\prime}=0}^{L-1}[\Phi_{\ell\ell^{\prime}}^{11
}(\boldsymbol{r}_{\ell i})+\Phi_{\ell\ell^{\prime}}^{21
}(\boldsymbol{r}_{\ell i})]+\sum\limits_{\ell^{\prime}=1}^{L}[\Phi_{\ell\ell^{\prime}}^{12
}(\boldsymbol{r}_{\ell i})+\Phi_{\ell\ell^{\prime}}^{22
}(\boldsymbol{r}_{\ell i})],
\end{split}
\end{equation}
where
\begin{equation}\label{freereactioncomponents}
\begin{split}
&  \Phi_{\ell}^{\text{free}}(\boldsymbol{r}_{\ell i}):=\sum\limits_{j=1,j\neq
	i}^{N_{\ell}}\frac{Q_{\ell j}}{4\pi|\boldsymbol{r}_{\ell
		i}-\boldsymbol{r}_{\ell j}|},\quad  \Phi_{\ell\ell^{\prime}}^{\mathfrak{ab}}(\boldsymbol{r}_{\ell i}):=\sum
\limits_{j=1}^{N_{\ell^{\prime}}}Q_{\ell^{\prime}j}u_{\ell\ell^{\prime}%
}^{\mathfrak{ab}}(\boldsymbol{r}_{\ell i},\boldsymbol{r}_{\ell^{\prime}j}%
),
\end{split}
\end{equation}
are free space and reaction field components, respectively. Far field approximations are used for both free space and reaction field components. Below, we first review the well-known theoretical results of the FMM for the free space components. It is mostly for the integrity of the theory and comparison with the convergence results that we will prove for the approximations of the reaction field components.

\subsection{Exponential convergence of ME and LE, shifting and translation operators for the free space components} Let $\Phi_{\ell,{\rm in}}^{\rm free}(\bs r)$ and $\Phi_{\ell,{\rm out}}^{\rm free}(\bs r)$ be the free space components of the potentials induced by all particles inside a given source box $B_s$ centered at $\bs r_c^s$ and all particles far away from a given target box $B_t$ centered at $\bs r_c^t$ (see. Fig. \ref{freespacemeandle}), i.e.,
\begin{equation}
\Phi_{\ell,{\rm in}}^{\rm free}(\bs r)=\sum\limits_{j\in \mathcal J}\frac{Q_{\ell j}}{4\pi|\boldsymbol{r}-\boldsymbol{r}_{\ell j}|},\quad \Phi_{\ell,{\rm out}}^{\rm free}(\bs r)=\sum\limits_{j\in \mathcal K}\frac{Q_{\ell j}}{4\pi|\boldsymbol{r}-\boldsymbol{r}_{\ell j}|},
\end{equation}
where $\mathcal J$ and $\mathcal K$ are the sets of indices of particles inside $B_s$ and of particles far away from $B_t$, respectively.
The FMM for free space components use ME
\begin{equation}\label{meapprox}
\Phi_{\ell,{\rm in}}^{\rm free}(\bs r)=\sum\limits_{n=0}^{\infty}\sum\limits_{m=-n}^nM_{nm}^{\rm in}\frac{Y_n^m(\theta_s,\varphi_s)}{r_s^{n+1}},
\end{equation}
at any target points far away from $B_s$ and LE
\begin{equation}\label{leapprox}
\Phi_{\ell,{\rm out}}^{\rm free}(\bs r)=\sum\limits_{n=0}^{\infty}\sum\limits_{m=-n}^nL_{nm}^{\rm out}r_t^nY_n^m(\theta_t,\varphi_t),
\end{equation}
inside $B_t$, where $(r_s, \theta_s, \phi_s)$ and $(r_t, \theta_t, \phi_t)$ are spherical coordinates of $\bs r-\bs r_c^s$ and $\bs r-\bs r_c^t$, respectively. The coefficients are given by
\begin{equation}\label{freespaceexpcoefbox}
M_{nm}^{\rm in}=\frac{c_n^{-2}}{4\pi }\sum\limits_{j\in\mathcal J}Q_{\ell j}(r_{\ell j}^s)^n\overline{Y_n^{m}(\theta_{\ell j}^s,\varphi_{\ell j}^s)},\quad L_{nm}^{\rm out}=\frac{c_n^{-2}}{4\pi }\sum\limits_{j\in\mathcal K}Q_{\ell j}(r_{\ell j}^t)^{-n-1}\overline{Y_n^{m}(\theta_{\ell j}^t,\varphi_{\ell j}^t)},
\end{equation}
where $(r_{\ell j}^s, \theta_{\ell j}^s, \phi_{\ell j}^s)$ and $(r_{\ell j}^t, \theta_{\ell j}^t, \phi_{\ell j}^t)$ are spherical coordinates of $\bs r_{\ell j}-\bs r_c^s$ and $\bs r_{\ell j}-\bs r_c^t$, respectively.
These expansions can be obtained by applying expansions \eqref{mefreespace}-\eqref{lefreespace} to the free space Green's function involved in the summation \eqref{freereactioncomponents}. By using Legendre addition theorem and estimates \eqref{meerror} and \eqref{leerror}, there holds the following error estimates (cf. \cite{grengard1988rapid}).
\begin{figure}[ht!]\label{freespacemeandle}
	\centering
	\includegraphics[scale=0.8]{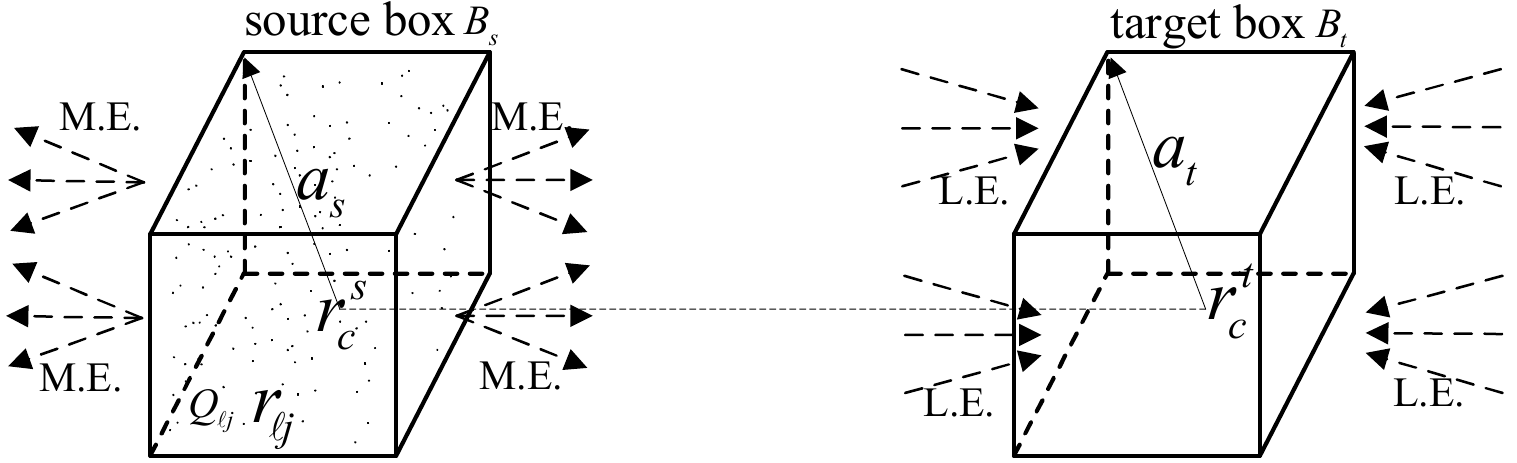}
	\caption{An illustration of the source and target box for the free space component in the $\ell$-th layer.}
\end{figure}
\begin{theorem}\label{meconvergence}
	Denote the radius of the circumscribed sphere of the source box $B_s$ by $a_s$. Then, the ME \eqref{meapprox} has error estimate
	\begin{equation}
	\Big|\Phi_{\ell,{\rm in}}^{\rm free}(\bs r)-\sum\limits_{n=0}^{p}\sum\limits_{m=-n}^nM_{nm}^{\rm in}\frac{Y_n^m(\theta_s,\varphi_s)}{r_s^{n+1}}\Big|\leq\frac{1}{4\pi}\frac{Q_{\mathcal J}}{r_s-a_s}\Big(\frac{a_s}{r_s}\Big)^{p+1},\quad\forall p\geq 1,
	\end{equation}
	for any $\bs r$ outside the circumscribed sphere, i.e., $|\bs r-\bs r_c^s|>a_s$, where
	\begin{equation}\label{totalchargefree}
	Q_{\mathcal J}=\sum\limits_{j\in\mathcal J}|Q_{\ell j}|.
	\end{equation}
\end{theorem}
\begin{theorem}\label{leconvergence}
	Denote the radius of the circumscribed sphere of the target box $B_t$ by $a_t$. Suppose $\mathcal K$ is the set of indices of all particles $(Q_{\ell j},\bs r_{\ell j})$ such that $|\bs r_{\ell j}-\bs r_c^t|>a_t$, then the LE \eqref{leapprox} has error estimate
	\begin{equation}
	\Big|\Phi_{\ell,{\rm out}}^{\rm free}(\bs r)-\sum\limits_{n=0}^{p}\sum\limits_{m=-n}^nL_{nm}^{\rm out}r_t^nY_n^m(\theta_t,\varphi_t)\Big|\leq\frac{1}{4\pi}\frac{Q_{\mathcal K}}{a_t-r_t}\Big(\frac{r_t}{a_t}\Big)^{p+1},\quad\forall p\geq 1,
	\end{equation}
	for any $\bs r\in B_t$, where
	\begin{equation}
	Q_{\mathcal K}=\sum\limits_{j\in\mathcal K}|Q_{\ell j}|.
	\end{equation}
\end{theorem}

Let $B_s^{\rm parent}$ be a parent box of the source box $B_s$ and $B_t^{\rm child}$ be a child box of the target box $B_t$ in the tree structure. Denote by $\tilde{\bs r}_c^s$ and $\tilde{\bs r}_c^t$ the centers of $B_s^{\rm parent}$ and $B_t^{\rm child}$, respectively. In the FMM, the shifting operations from the ME \eqref{meapprox} at $\bs r_c^s$ to new ME at $\tilde{\bs r}_c^s$ and from the LE \eqref{leapprox} at $\bs r_c^t$ to new LE at $\tilde{\bs r}_c^t$ are required. Denote the ME and LE at new centers $\tilde{\bs r}_c^s$ and $\tilde{\bs r}_c^t$ by \begin{equation}\label{meleapproxshifting}
\Phi_{\ell,{\rm in}}^{\rm free}(\bs r)=\sum\limits_{n=0}^{\infty}\sum\limits_{m=-n}^n\tilde M_{nm}^{\rm in}\frac{Y_n^m(\tilde\theta_s,\tilde\varphi_s)}{\tilde r_s^{n+1}},\quad
\Phi_{\ell,{\rm out}}^{\rm free}(\bs r)=\sum\limits_{n=0}^{\infty}\sum\limits_{m=-n}^n\tilde L_{nm}^{\rm out}\tilde r_t^nY_n^m(\tilde\theta_t,\tilde\varphi_t).
\end{equation}
Recall shifting operators \eqref{metome}-\eqref{letole}, we have
\begin{align}
\tilde{M}_{nm}^{\rm in}
=&\sum\limits_{\nu=0}^{n}\sum\limits_{\mu=-\nu}^{\nu}\frac{(-1)^{|m|-|\mu|}A_{n-\nu}^{m-\mu}A_{\nu}^{\mu}r_{ss}^{n-\nu}Y_{n-\nu}^{\mu-m}(\theta_{ss},\varphi_{ss})}{c_{n-\nu}^2A_{n}^{m}}M_{\nu\mu}^{\rm in},\label{metomenpart}\\
\tilde L_{nm}^{\rm out}=&\sum\limits_{\nu=n}^{\infty}\sum\limits_{\mu=-\nu}^{\nu} \frac{(-1)^{\nu-n-|\mu-m|+|\mu|-|m|}c_{\nu}^2A_{\nu-n}^{\mu-m}A_n^mr_{tt}^{\nu-n}Y_{\nu-n}^{\mu-m}(\theta_{tt}, \varphi_{tt})}{c_{\nu-n}^2c_n^2A_{\nu}^{\mu}}L_{\nu\mu}^{\rm out},\label{letolenpart}
\end{align}
$(\tilde r_s,\tilde\theta_s,\tilde\varphi_s)$, $(r_{ss},\theta_{ss},\phi_{ss})$, $(\tilde r_t,\tilde\theta_t,\tilde\varphi_t)$ and $(r_{tt},\theta_{tt},\phi_{tt})$ are the spherical coordinates of $\bs r-\tilde{\bs r}_c^s$, $\bs r_c^s-\tilde{\bs r}_c^s$, $\bs r-\tilde{\bs r}_t^s$ and $\tilde{\bs r}_c^t-\bs r_c^t$, respectively. According to the ME to ME translation \eqref{metomenpart}, we see that any ME coefficients $\tilde M_{nm}^{\rm in}$ in the ME at $\tilde{\bs r}_c^s$ can be computed exactly by the ME coefficients  $\{ M_{\nu\mu}^{\rm in}\}_{\nu=0}^n$ in the ME at $\bs r_c^s$. Therefore, the ME obtained via shifting operator \eqref{metomenpart} is actually the the unique ME of $\Phi_{\ell,{\rm in}}^{\rm free}(\bs r)$ at $\tilde{\bs r}_c^s$ (cf. \cite{grengard1988rapid}). As in Theorem \ref{meconvergence}, the following error estimate holds.
\begin{theorem}
	Denote the radius of the circumscribed sphere of the source box $B_s$ by $a_s$. For any $|\bs r-\tilde{\bs r}_c^s|>a_s+r_{ss}$, the first expansion in \eqref{meleapproxshifting} has error estimate
	\begin{equation}
	\Big|\Phi_{\ell,{\rm in}}^{\rm free}(\bs r)-\sum\limits_{n=0}^{p}\sum\limits_{m=-n}^n\tilde M_{nm}^{\rm in}\frac{Y_n^m(\tilde\theta_s,\tilde\varphi_s)}{\tilde r_s^{n+1}}\Big|\leq \frac{1}{4\pi}\frac{Q_{\mathcal J}}{\tilde r_s-(a_s+r_{ss})}\Big(\frac{a_s+r_{ss}}{\tilde r_s}\Big)^{p+1},
	\end{equation}
	where $Q_{\mathcal J}$ is defined in \eqref{totalchargefree}.
\end{theorem}

Although, the LE to LE shifting operator \eqref{letolenpart} has an infinite summation, the shifting operation remains exact with finite sum when we are shifting a truncated LE to a new center. In practice, the truncated LE
\begin{equation}
\Phi_{\ell,{\rm out}}^{\rm free}(\bs r)\approx\widetilde\Phi_{\ell,{\rm out}}^{\rm free}(\bs r):=\sum\limits_{n=0}^{p}\sum\limits_{m=-n}^n L_{nm}^{\rm out} r_t^nY_n^m(\theta_t,\varphi_t),
\end{equation}
is used in the FMM. It can be seen as an infinite sum with $L_{nm}^{\rm out}=0$ for $n>p$. Then by \eqref{letolenpart}, we have $\tilde L_{nm}^{\rm out}=0$ for $n>p$. The shifting LE in \eqref{meleapproxshifting} reduce to finite summation:
\begin{equation}\label{approxletole}
\widetilde\Phi_{\ell,{\rm out}}^{\rm free}(\bs r)=\sum\limits_{n=0}^{p}\sum\limits_{m=-n}^n\tilde L_{nm}^{p}\tilde r_t^nY_n^m(\tilde\theta_t,\tilde\varphi_t)
\end{equation}
where
\begin{equation}\label{truncatedLEtoLE}
\tilde L_{nm}^{p}=\sum\limits_{\nu=n}^{p}\sum\limits_{\mu=-\nu}^{\nu} \frac{(-1)^{\nu-n-|\mu-m|+|\mu|-|m|}c_{\nu}^2A_{\nu-n}^{\mu-m}A_n^mr_{tt}^{\nu-n}Y_{\nu-n}^{\mu-m}(\theta_{tt}, \varphi_{tt})}{c_{\nu-n}^2c_n^2A_{\nu}^{\mu}}L_{\nu\mu}^{\rm out}.
\end{equation}
Therefore, the truncated LE to LE shifting \eqref{truncatedLEtoLE} used in the FMM implementation is exact.

Suppose target box $B_t$ is far away from the source box $B_s$. Recall the translation operator \eqref{metole}, the LE expansion coefficient in \eqref{leapprox} can be calculated from ME coefficients via
\begin{equation}\label{metolenotruncate}
L_{nm}^{\rm out}=\sum\limits_{|\nu|=0}^{\infty}\sum\limits_{\mu=-\nu}^{\nu}\frac{(-1)^{\nu+|m|}A_{\nu}^{\mu}A_n^mY_{n+\nu}^{\mu-m}(\theta_{st}, \varphi_{st})}{c_{\nu}^2A_{n+\nu}^{\mu-m}r_{st}^{n+\nu+1}}M_{\nu\mu}^{\rm in}.
\end{equation}
Again, \eqref{metolenotruncate} can not be directly used in the FMM due to the infinite summation. In the FMM, the formulas in \eqref{metolenotruncate} for local expansion coefficients $L_{nm}^{\rm out}$ are further truncated which gives approximated local expansion coefficients
\begin{equation}\label{metoletruncate}
L_{nm}^{p}=\sum\limits_{|\nu|=0}^{p}\sum\limits_{\mu=-\nu}^{\nu}\frac{(-1)^{\nu+|m|}A_{\nu}^{\mu}A_n^mY_{n+\nu}^{\mu-m}(\theta_{st}, \varphi_{st})}{c_{\nu}^2A_{n+\nu}^{\mu-m}r_{st}^{n+\nu+1}}M_{\nu\mu}^{\rm in}.
\end{equation}
We find that the detailed proof of the error estimate for the truncated M2L translation has not been presented in the literature. Therefore, we present a proof as follow:
\begin{theorem}\label{metoleconvergence}
	Suppose $B_s$ and $B_t$ are well separated cubic boxes and denote the radii of their circumscribed spheres by $a_s$ and $a_t$. The well separateness of the boxes means that $|\bs r_c^s-\bs r_c^t|>a_s+ca_t$ with $c>1$. Then
	\begin{equation}\label{metoleerror}
	\Big|\Phi_{\ell,{\rm out}}^{\rm free}(\bs r)-\sum\limits_{n=0}^{p}\sum\limits_{m=-n}^n L_{nm}^{p} r_t^nY_n^m(\theta_t,\varphi_t)\Big|\leq \frac{1}{4\pi}  \frac{Q_{\mathcal J}}{(c-1)a_t}\Big(\frac{a_s+a_t}{a_s+ca_t}\Big)^{p+1},\forall \bs r\in B_t,
	\end{equation}
	where $Q_{\mathcal J}$ is defined in \eqref{totalchargefree}.
\end{theorem}
\begin{proof}
	By the assumption $|\bs r_c^s-\bs r_c^t|>a_s+ca_t$ ($c>1$), we have $|\bs r'-\bs r_c^s|+|\bs r-\bs r_c^t|<|\bs r_c^t-\bs r_c^s|$ for any $\bs r\in B_t$, $\bs r'\in B_s$.  Then, as in \eqref{expansionbeforeme}-\eqref{localexpansionbeforeme}, we have Taylor expansion
	\begin{equation}\label{M2LTaylorexp}
	\frac{1}{4\pi|\bs r-\bs r'|}=\frac{1}{4\pi|\bs r_t-\bs r_s'-(\bs r_c^s-\bs r_c^t)|}=\frac{1}{4\pi }\sum\limits_{n'=0}^{\infty}\frac{P_{n'}(\xi)}{ |\bs r_c^t-\bs r_c^s|}\Big(\frac{|\bs r_t-\bs r_s'|}{ |\bs r_c^t-\bs r_c^s|}\Big)^{n'},
	\end{equation}
	where
	\begin{equation}
	\xi=-\frac{(\bs r_s'-\bs r_t)\cdot(\bs r_c^s-\bs r_c^t)}{|\bs r_s'-\bs r_t||\bs r_c^s-\bs r_c^t|},\quad\bs r_s'=\bs r'-\bs r_c^s, \quad \bs r_t=\bs r-\bs r_c^t.
	\end{equation}
	Truncate the expansion \eqref{M2LTaylorexp} and denote the approximation by
	\begin{equation}\label{truncatedM2LTaylorexp}
	\psi^p(\bs r, \bs r')=\frac{1}{4\pi }\sum\limits_{n'=0}^{p}\frac{P_{n'}(\xi)}{ |\bs r_c^t-\bs r_c^s|}\Big(\frac{|\bs r_t-\bs r_s'|}{ |\bs r_c^t-\bs r_c^s|}\Big)^{n'}.
	\end{equation}
	Then, we directly have error estimate
	\begin{equation}\label{taylorexperr}
	\begin{split}
	\Big|\frac{1}{4\pi|\bs r-\bs r'|}-\psi^p(\bs r, \bs r')\Big|\leq&\frac{1}{4\pi(|\bs r_c^s-\bs r_c^t|-|\bs r_t-\bs r_s'|)}\left(\frac{|\bs r_t-\bs r_s'|}{ |\bs r_c^s-\bs r_c^t|}\right)^{p+1}\\
	\leq& \frac{1}{4\pi(c-1)a_t}\Big(\frac{a_s+a_t}{a_s+ca_t}\Big)^{p+1}.
	\end{split}
	\end{equation}

	Applying identity $P_n(-x)=(-1)^nP_n(x)$ and Legendre addition theorem in \eqref{truncatedM2LTaylorexp} gives
	\begin{equation}
	\psi^p(\bs r, \bs r')=\sum\limits_{n'=0}^{p}\frac{1}{2n'+1}\frac{(-1)^{n'}}{r_{st}^{n'+1}}\sum\limits_{m'=-n'}^{n'}\overline{Y_{n'}^{m'}(\theta_{st},\varphi_{st})}|\bs r_s'-\bs r_t|^{n'}Y_{n'}^{m'}(\widehat{\bs r_s'-\bs r_t}),
	\end{equation}
	where $(r_{st},\theta_{st},\varphi_{st})$ is the spherical coordinates of $\bs r_c^s-\bs r_c^t$.
	Further, applying addition theorem \eqref{theorem:fourthaddition} and then rearranging the resulted summation, we obtain
	\begin{equation*}
	\begin{split}
	\psi^p(\bs r, \bs r')=&\sum\limits_{n'=0}^{p}\sum\limits_{m'=-n'}^{n'}\frac{\overline{Y_{n'}^{m'}(\theta_{st},\varphi_{st})}}{4\pi r_{st}^{n'+1}}\sum\limits_{n=0}^{n'}\sum\limits_{m=-n}^{n}B_{n'm'}^{nm}(r_s')^{n'-n}Y_{n'-n}^{m'-m}(\theta_s',\varphi_s')r_t^{n}Y_{n}^{m}(\theta_t,\varphi_t)\\
	=&\sum\limits_{n=0}^p\sum\limits_{m=-n}^{n}\sum\limits_{n'=n}^{p}\sum\limits_{m'=-n'}^{n'}\frac{\overline{Y_{n'}^{m'}(\theta_{st},\varphi_{st})}}{4\pi r_{st}^{n'+1}}B_{n'm'}^{nm}(r_s')^{n'-n}Y_{n'-n}^{m'-m}(\theta_s',\varphi_s')r_t^{n}Y_{n}^{m}(\theta_t,\varphi_t)\\
	=&\sum\limits_{n=0}^p\sum\limits_{m=-n}^{n}\Big[\sum\limits_{\nu=0}^{p}\sum\limits_{\mu=-\nu}^{\nu}\frac{\overline{Y_{n+\nu}^{m+\mu}(\theta_{st},\varphi_{st})}}{4\pi r_{st}^{n+\nu+1}}B_{n+\nu,m+\mu}^{nm}(r_s')^{\nu}Y_{\nu}^{\mu}(\theta_s',\varphi_s')\Big]r_t^{n}Y_{n}^{m}(\theta_t,\varphi_t),
	\end{split}
	\end{equation*}
	where
	\begin{equation*}
	B_{n'm'}^{nm}=\frac{(-1)^{n'+n-|m|+|m'|-|m'-m|} A_{n}^{m}A_{n'-n}^{m'-m} }{c_{n}^2c_{n'-n}^2A_{n'}^{m'}},
	\end{equation*}
	$(r_t,\theta_t,\varphi_t)$ and $(r'_s, \theta'_s,\varphi'_s)$, are the spherical coordinates of $\bs r-\bs r_c^t$ and $\bs r'-\bs r_c^s$, respectively.
	Apparently, $\psi^p(\bs r, \bs r')$ is a truncated LE at target center $\bs r_c^t$ with coefficients given by
	\begin{equation}
	\hat L_{nm}^p=\sum\limits_{\nu=0}^{p}\sum\limits_{\mu=-\nu}^{\nu}\frac{\overline{Y_{n+\nu}^{m+\mu}(\theta_{st},\varphi_{st})}}{4\pi r_{st}^{n+\nu+1}}B_{n+\nu,m+\mu}^{nm}(r_s')^{\nu}Y_{\nu}^{\mu}(\theta_s',\varphi_s').
	\end{equation}
	By identity $Y_{\nu}^{\mu}(\theta,\varphi)=(-1)^{\mu}\overline{Y_{\nu}^{-\mu}(\theta,\varphi)}$,
	the coefficients can be re-expressed as
	\begin{equation}
	\begin{split}
	\hat L_{nm}^p=&\sum\limits_{\nu=0}^{p}\sum\limits_{\mu=-\nu}^{\nu}\frac{\overline{Y_{n+\nu}^{m-\mu}(\theta_{st},\varphi_{st})}}{r_{st}^{n+\nu+1}}\tilde A_{n+\nu,m-\mu}^{nm}(r_s')^{\nu}(-1)^{\mu}\overline{Y_{\nu}^{\mu}(\theta_s',\varphi_s')}\\
	=&\sum\limits_{\nu=0}^{p}\sum\limits_{\mu=-\nu}^{\nu}\frac{(-1)^{\nu-|m|} A_{n}^{m}A_{\nu}^{-\mu}Y_{n+\nu}^{\mu-m}(\theta_{st},\varphi_{st}) }{c_{n}^2A_{n+\nu}^{m-\mu}r_{st}^{n+\nu+1}}M_{\nu}^{\mu}.
	\end{split}
	\end{equation}
	Noting that $A_n^m=A_{n}^{-m}$, the above coefficients is exactly the truncated M2L  coefficients given in \eqref{metoletruncate}. As a result, we have
	\begin{equation}
	\sum\limits_{n=0}^{p}\sum\limits_{m=-n}^n L_{nm}^{p} r_t^nY_n^m(\theta_t,\varphi_t)=\sum\limits_{j\in\mathcal J}Q_{\ell j}\psi^p(\bs r, \bs r_{\ell j}),
	\end{equation}
	and the error estimate \eqref{metoleerror} follows by applying \eqref{taylorexperr} term by term.
\end{proof}

\subsection{Exponential convergence of the ME for reaction components and the conception of equivalent polarization source} The exponential convergence in the Theorem \ref{meconvergence} is a direct result of the error estimate \eqref{meerror}. However, it is much more difficult to derive error estimates for the MEs \eqref{melayerupgoingimage1} of the reaction components of layered Green's function. Here, we first present the main theorem which is the key to prove the exponential convergence of the MEs, LEs and M2L translation operators in this and the next subsections. For the smoothness of the presentation, the detailed proof will be postponed to subsection 3.4.

Let us consider the convergence and error estimates of the MEs in \eqref{melayerupgoingimage1}. According to its derivation in section 3.2, we only need to prove the convergence and error estimates of the expansions in \eqref{melayerupgoingtaylor}. For this purpose, define general integral
\begin{equation}\label{generalintegraldef}
\mathcal I(\bs r;\sigma)=\int_{-\infty}^{\infty}\int_{-\infty}^{\infty}\frac{1}{k_{\rho}}e^{\ri\bs k\cdot\bs r}\sigma(k_{\rho})dk_xdk_y, \quad \forall \bs r=(x, y, z)\in \mathbb R^3,
\end{equation}
where $k_{\rho}=\sqrt{k_x^2+k_y^2}$, $\bs k=(k_x, k_y, \ri k_{\rho})$, $\sigma(k_{\rho})$ is a given density function.
Applying Taylor expansion gives
\begin{equation}\label{twointegralexp}
\begin{split}
&\mathcal I(\bs r+\bs r';\sigma)=\int_{-\infty}^{\infty}\int_{-\infty}^{\infty}e^{\ri\bs k\cdot\bs r}\sum\limits_{n=0}^{\infty}\frac{1}{k_{\rho}}\frac{(\ri \bs k\cdot{\bs r}')^n}{n!}\sigma(k_{\rho})dk_xdk_y,\\
&\mathcal I(\bs r+\bs r'+\bs r'';\sigma)=\int_{-\infty}^{\infty}\int_{-\infty}^{\infty}\frac{1}{k_{\rho}}e^{\ri\bs k\cdot\bs r}\sum\limits_{n=0}^{\infty}\sum\limits_{\nu=0}^{\infty}\frac{(\ri \bs k\cdot{\bs r}')^n(\ri \bs k\cdot{\bs r}'')^{\nu}}{n!\nu!}\sigma(k_{\rho})dk_xdk_y,
\end{split}
\end{equation}
for any $\bs r'=(x', y', z')$, $\bs r''=(x'', y'', z'')$ in $\mathbb R^3$.
Suppose $z>0$, $z+z'>0$, $z+z'+z''>0$, and the density function $\sigma(k_{\rho})$ is not increasing exponentially as $k_{\rho}\rightarrow\infty$, then the integrals in \eqref{generalintegraldef} and \eqref{twointegralexp} are convergent. We first present the conclusion that the improper integral and the infinite summation in \eqref{twointegralexp} can exchange order and the resulted series have exponential convergence under suitable conditions. Detailed proof will be given in section 4.5. For the sake of brevity, we denote
\begin{equation}\label{InInuintegrals}
\begin{split}
&\mathcal I_n(\bs r, \bs r';\sigma)=\int_{-\infty}^{\infty}\int_{-\infty}^{\infty}\frac{1}{k_{\rho}}e^{\ri\bs k\cdot\bs r}\frac{(\ri \bs k\cdot{\bs r}')^n}{n!}\sigma(k_{\rho})dk_xdk_y,\\
&\mathcal I_{n\nu}(\bs r,\bs r',\bs r'';\sigma)=\int_{-\infty}^{\infty}\int_{-\infty}^{\infty}\frac{1}{k_{\rho}}e^{\ri\bs k\cdot\bs r}\frac{(\ri \bs k\cdot{\bs r}')^n(\ri \bs k\cdot{\bs r}'')^{\nu}}{n!\nu!}\sigma(k_{\rho})dk_xdk_y.
\end{split}
\end{equation}

\begin{theorem}\label{Thm:generalintegral}
	Suppose the density function $\sigma(k_{\rho})$ is analytic and has a bound $|\sigma(k_{\rho})|\leq\bs M_{\sigma}$ in the right half complex plane $\{k_{\rho}:\mathfrak{Re}k_{\rho}>0\}$, $\bs r=(x, y, z), \bs r'=(x', y', z'), \bs r''=(x'', y'', z'')\in \mathbb R^3$ such that $z>0$, $z+z'>0$, $z+z'+z''>0$ and $|\bs r|>|\bs r'|$, $|\bs r|>|\bs r'|+|\bs r''|$. Then, the following expansions
	\begin{equation}\label{finaltwointegralexp}
	\mathcal I(\bs r+\bs r';\sigma)=\sum\limits_{n=0}^{\infty}\mathcal I_n(\bs r, \bs r';\sigma),\quad\mathcal I(\bs r+\bs r'+\bs r'';\sigma)=\sum\limits_{n=0}^{\infty}\sum\limits_{\nu=0}^{\infty}\mathcal I_{n\nu}(\bs r,\bs r',\bs r'';\sigma),
	\end{equation}
	hold. Moreover, the truncation error estimates are given by
	\begin{equation}\label{reactionintegralestimate1}
	\Big|\mathcal I(\bs r+ \bs r';\sigma)-\sum\limits_{n=0}^{p}\mathcal I_n(\bs r, \bs r';\sigma)\Big|\leq \frac{2\pi\bs M_{\sigma}}{|\bs r|-|\bs r'|}\Big|\frac{\bs r'}{\bs r}\Big|^{p+1},
	\end{equation}
	and
	\begin{equation}\label{reactionintegralestimate2}
	\Big|\mathcal I(\bs r+ \bs r'+\bs r'';\sigma)-\sum\limits_{n=0}^{p}\sum\limits_{\nu=0}^{p}\mathcal I_{n\nu}(\bs r,\bs r',\bs r'';\sigma)\Big|\leq \frac{4\pi\bs M_{\sigma}}{|\bs r|-|\bs r'|-|\bs r''|}\left(\frac{|\bs r'|+|\bs r''|}{|\bs r|}\right)^{p+1}.
	\end{equation}
\end{theorem}

Applying the above theorem, we can prove the convergence of the MEs in \eqref{melayerupgoingimage1}.
\begin{theorem}\label{Thm:reactmeconvergence}
	Given $\bs r=(x, y, z)$ and $\bs r'=(x', y', z')$ be two points in the $\ell$-th and $\ell'$-th layer, i.e., $d_{\ell}<z<d_{\ell-1}, d_{\ell'}<z'<d_{\ell'-1}$, respectively. Suppose $\bs r_c^s$ is a source center in the $\ell'$-th layer such that  $|\bs r'-\bs r_c^s|<|\tau_{\ell\ell'}^{\mathfrak{ab}}(\bs r, \bs r_c^s)|$, then the expansions in \eqref{melayerupgoingimage1} hold and have the following error estimate
	\begin{equation}\label{errorestimatereaction}
	\Big|u_{\ell\ell'}^{\mathfrak{ab}}(\bs r,\bs r')-\sum\limits_{n=0}^{p}\sum\limits_{m=-n}^nM_{nm}^{\mathfrak{ab}}{\mathcal F}_{nm}^{\mathfrak{ab}}(\bs r, \bs r_c^s)\Big|\leq\frac{(4\pi)^{-1}\bs M_{\sigma_{\ell\ell'}^{\mathfrak{ab}}}}{|\tau_{\ell\ell'}^{\mathfrak{ab}}(\bs r, \bs r_c^s)|-|\bs r'-\bs r_c^s|}\left(\frac{|\bs r'-\bs r_c^s|}{|\tau_{\ell\ell'}^{\mathfrak{ab}}(\bs r, \bs r_c^s)|}\right)^{p+1},
	\end{equation}
	where $\bs M_{\sigma_{\ell\ell'}^{\mathfrak{ab}}}$ is the bound of  $\sigma_{\ell\ell'}^{\mathfrak{ab}}(k_{\rho})$ in the right half complex plane.
\end{theorem}
\begin{proof}
	As in the analysis presented in the last subsection, it is more convenient to prove the expansions \eqref{melayerupgoingtaylor}. Recalling definition \eqref{generalintegraldef} and \eqref{generalcomponents}, we obtain
	\begin{equation}
	u_{\ell\ell'}^{\mathfrak{a}1}(\bs r,\bs r')=\frac{1}{8\pi^2 }\mathcal I(\tau_{\ell\ell'}^{\mathfrak a 1}(\bs r,\bs r_c^s)+ {\tau(-\bs r'_s)};\sigma_{\ell\ell'}^{\mathfrak{a}1}),\;\;
	u_{\ell\ell'}^{\mathfrak{a}2}(\bs r,\bs r')=\frac{1}{8\pi^2 }\mathcal I(\tau_{\ell\ell'}^{\mathfrak a 2}(\bs r,\bs r_c^s) -\bs r'_s;\sigma_{\ell\ell'}^{\mathfrak{a}2}).
	\end{equation}
	From definition \eqref{coordmapping}, we can see that $\tau_{\ell\ell'}^{\mathfrak{ab}}(\bs r, \bs r_c^s)$ always have positive $z$-coordinates given $\bs r$ in the $\ell$-th layer and $\bs r_c^s$ in the $\ell'$-th layer. Moreover, Proposition \ref{matrixentrylemma} shows that the density function $\sigma_{\ell\ell'}^{\mathfrak{ab}}(k_{\rho})$ is analytic and bounded in the right half complex plane $\{k_{\rho}:\mathfrak{Re}k_{\rho}>0\}$. Together with the assumption $|\bs r'-\bs r_c^s|<|\tau_{\ell\ell'}^{\mathfrak{ab}}(\bs r, \bs r_c^s)|$, theorem \ref{Thm:generalintegral} with density function $\sigma_{\ell\ell'}^{\mathfrak{ab}}(k_{\rho})$ and coordinates groups $\{\bs\tau_{\ell\ell'}^{\mathfrak{a}1}(\bs r, \bs r_c^s), \bs\tau(-\bs r_s')\}$, $\{\bs\tau_{\ell\ell'}^{\mathfrak{a}2}(\bs r, \bs r_c^s), -\bs r_s'\}$ can be applied. Therefore, we obtain
	\begin{equation}\label{reformula}
	\begin{split}
	u_{\ell\ell'}^{\mathfrak{a}1}(\bs r,\bs r')=&\frac{1}{8\pi^2 }\sum\limits_{n=0}^{\infty}\mathcal I_n(\tau_{\ell\ell'}^{\mathfrak a 1}(\bs r,\bs r_c^s), {\tau(-\bs r'_s)};\sigma_{\ell\ell'}^{\mathfrak{a}1}),\\
	u_{\ell\ell'}^{\mathfrak{a}2}(\bs r,\bs r')=&\frac{1}{8\pi^2 }\sum\limits_{n=0}^{\infty}\mathcal I_n(\tau_{\ell\ell'}^{\mathfrak a 2}(\bs r,\bs r_c^s), -\bs r'_s;\sigma_{\ell\ell'}^{\mathfrak{a}2}).
	\end{split}
	\end{equation}
	At the mean time, the error estimate \eqref{errorestimatereaction} follows by applying \eqref{reactionintegralestimate1} in \eqref{reformula}.
\end{proof}

The convergence results in the above indicates an important fact that the error of the truncated ME is not determined by the Euclidean distance between source center $\bs r_c^s$ and target $\bs r$ as in the free space case (see Theorem \ref{meconvergence}). Actually, the distances along $z$-direction have been replaced by summations of the distances between $\bs r$, $\bs r_c^s$ and corresponding nearest interfaces of the layered media.

Nevertheless, there are two special cases, i.e., $|\tau_{\ell\ell+1}^{12}(\bs r,\bs r_c^s)|=|\bs r-\bs r_c^s|$ if $\bs r$ and $\bs r_c^s$ are in the $\ell$-th and $(\ell+1)$-th layer and $|\tau_{\ell\ell-1}^{21}(\bs r,\bs r_c^s)|=|\bs r-\bs r_c^s|$ if $\bs r$ and $\bs r_c^s$ are in the $\ell$-th and $(\ell-1)$-th layer. Therefore, the MEs of $u_{\ell\ell+1}^{12}(\bs r, \bs r')$ and $u_{\ell\ell-1}^{21}(\bs r, \bs r')$ have the same convergence behavior as that of free space components.
\begin{figure}[ht!]
	\center
	\subfigure[$u_{\ell\ell'}^{11}$]{\includegraphics[scale=0.6]{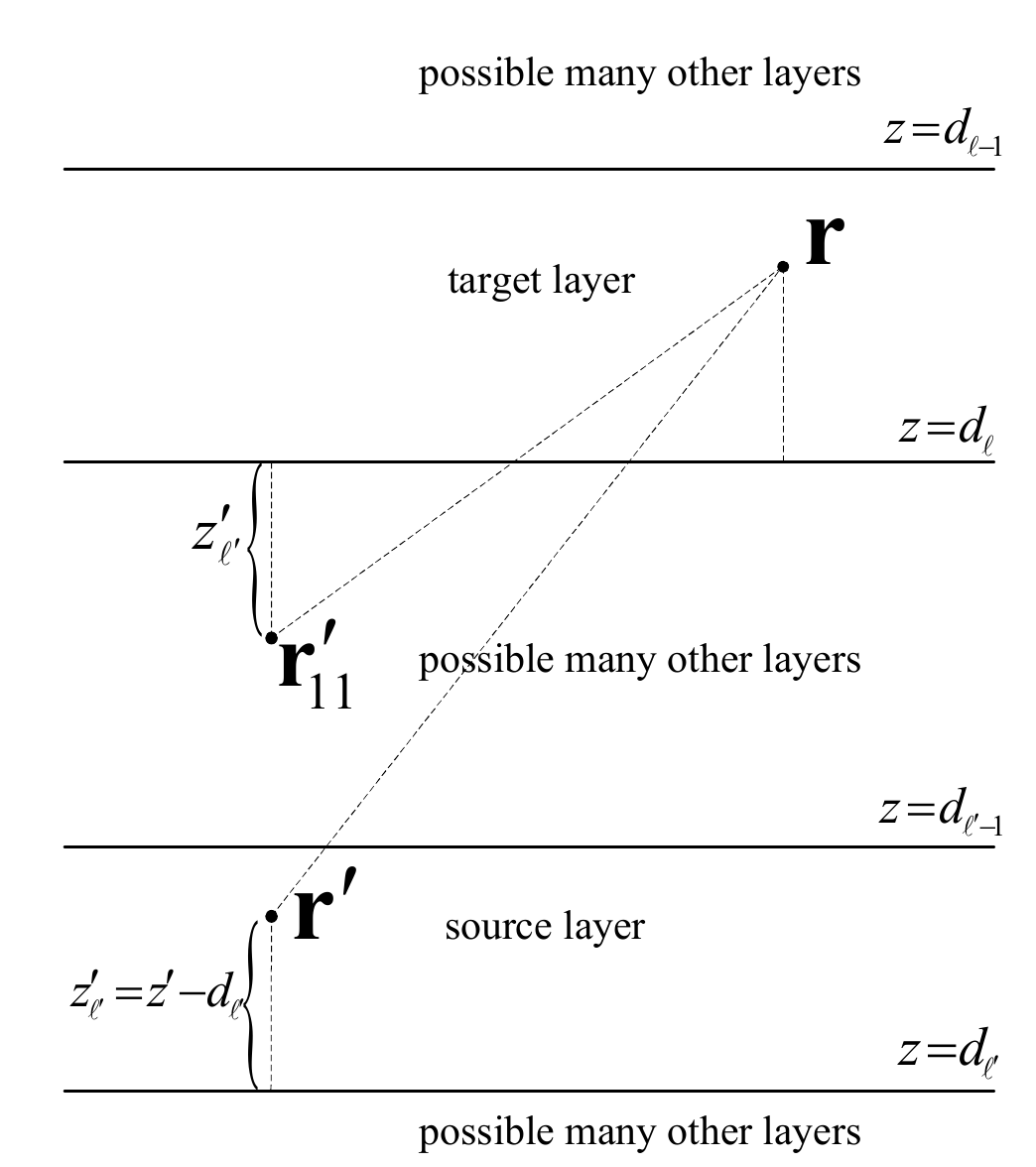}}\quad
	\subfigure[$u_{\ell\ell'}^{21}$]{\includegraphics[scale=0.6]{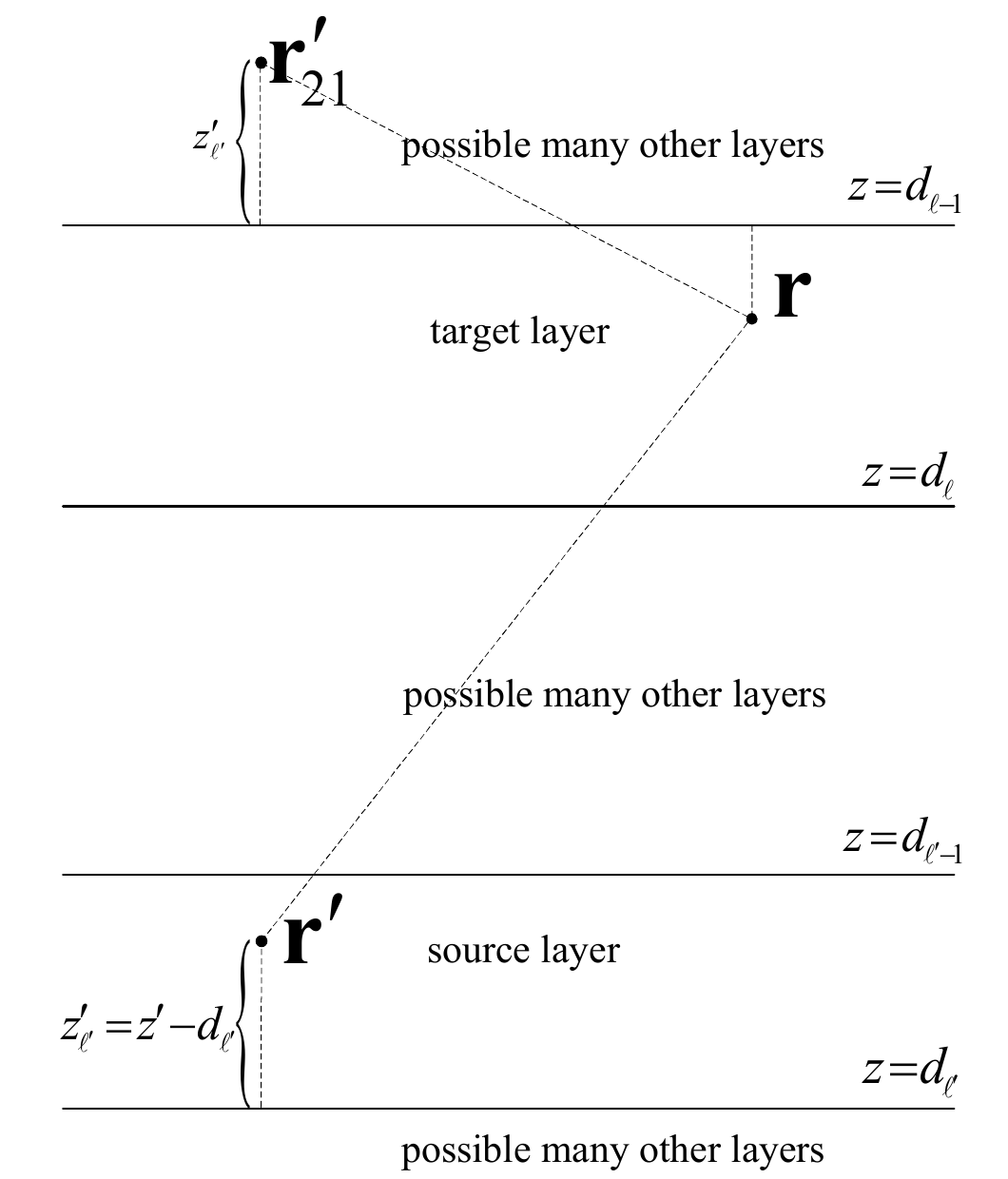}}
	\subfigure[$u_{\ell\ell'}^{12}$]{\includegraphics[scale=0.6]{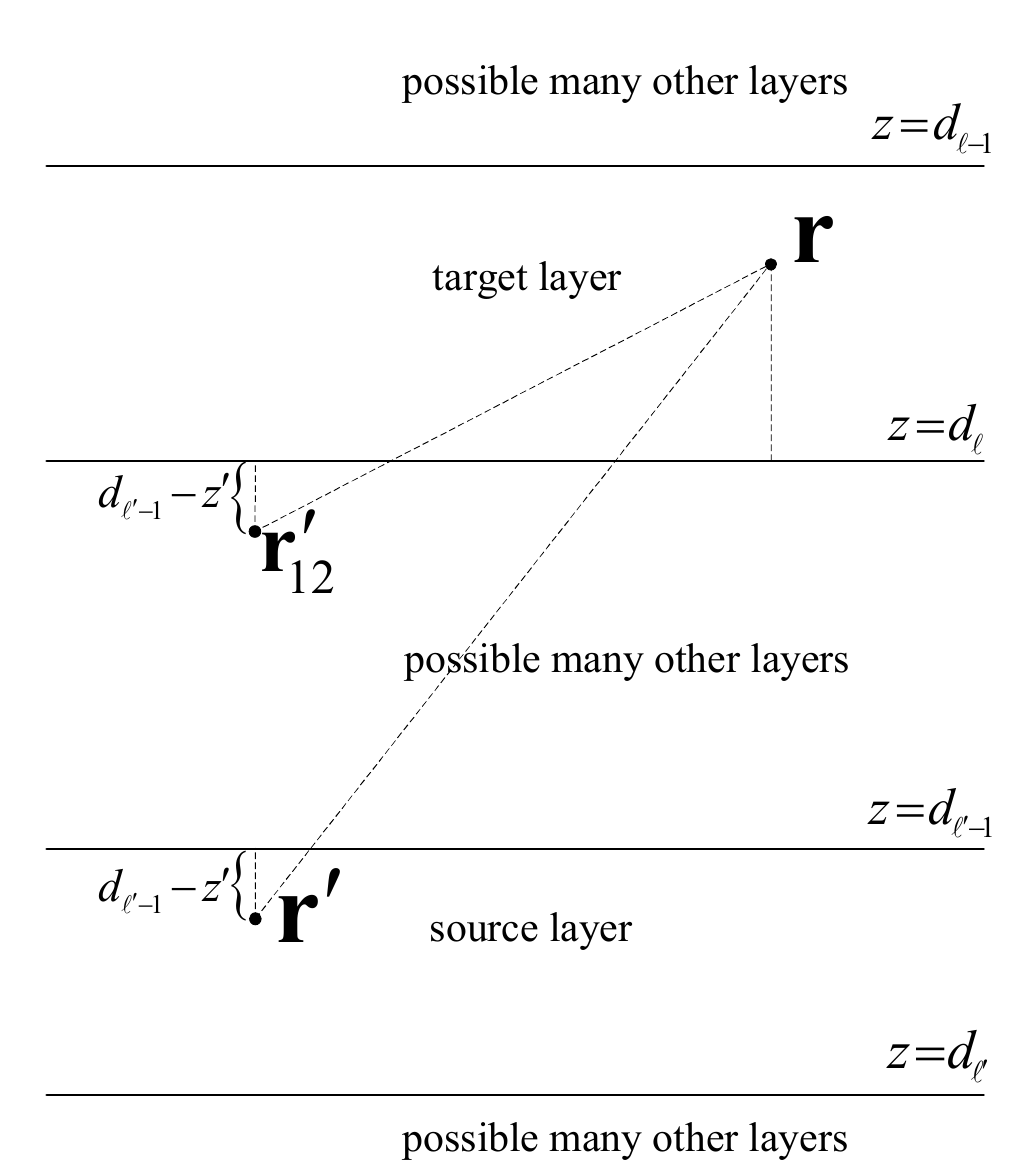}}\quad
	\subfigure[$u_{\ell\ell'}^{22}$]{\includegraphics[scale=0.6]{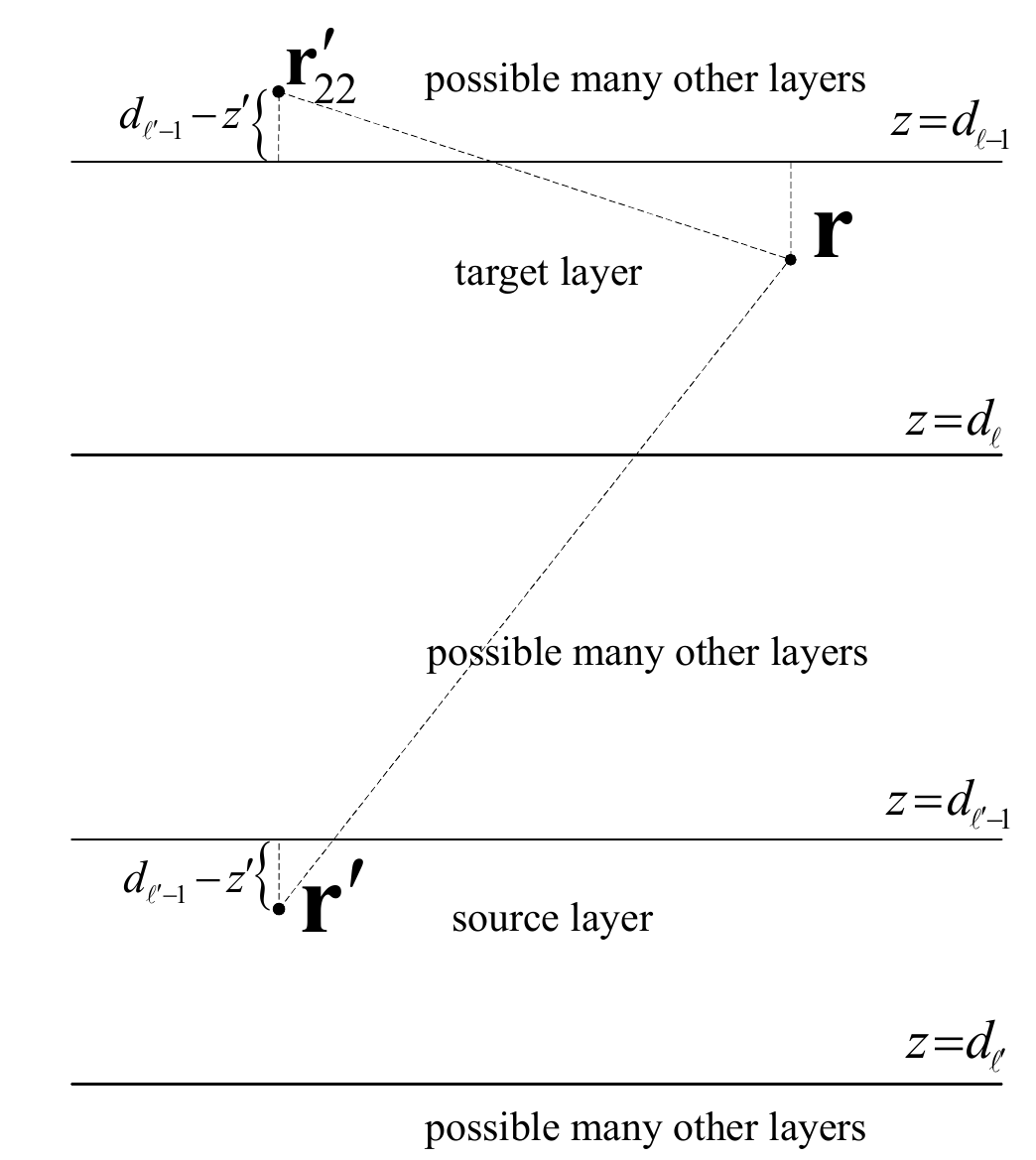}}
	\caption{Location of equivalent polarization sources associated to $u_{\ell\ell'}^{\mathfrak{ab}}$.}%
	\label{sourceimages}%
\end{figure}

On the other hand, the key idea of using hierarchical tree structure in FMM relies on using the Euclidean distance between source and target to determine either direct calculation or truncated ME is used for the computation of the interactions. In applying the free space FMM framework to handle the reaction field components of the layered Green's function, the main problem is the MEs given in \eqref{melayerupgoingimage1} are generally not compatible with the hierarchical tree structure design. In our previous work \cite{wang2019fastlaplace,wang2019fast}, we have introduced the conception of equivalent polarization sources to overcome this problem. The idea was inspired by our theoretical analysis for 2-dimensional Helmholtz equation (cf. \cite{zhang2020exponential}) and numerical tests for 3-dimensional Helmholtz/Laplace equations in layered media (cf. \cite{wang2019fastlaplace,wang2019fast}). Here, the theoretical results in \eqref{errorestimatereaction} further verify the necessity of using the equivalent polarization sources.

According to the convergence results for MEs in Theorem \ref{Thm:reactmeconvergence}, we introduce equivalent polarization sources for the four types of reaction fields (see. Fig. \ref{sourceimages})
\begin{equation}\label{eqpolarizedsource}
\begin{split}
&\bs r'_{11}:=(x', y', d_{\ell}-(z'-d_{\ell'})),\quad\quad\bs r'_{12}:=(x', y', d_{\ell}-(d_{\ell'-1}-z')),\\
&\bs r'_{21}:=(x', y', d_{\ell-1}+(z'-d_{\ell'})),\quad \bs r'_{22}:=(x', y', d_{\ell-1}+(d_{\ell'-1}-z')).
\end{split}
\end{equation}
With the equivalent polarization sources, we define reaction potentials
\begin{equation}\label{reactioncomponentsepsource}
\begin{split}
&\tilde u_{\ell\ell'}^{1\mathfrak b}(\bs r, \bs r'_{1\mathfrak b})=\frac{1}{8\pi^2 }\int_{-\infty}^{\infty}\int_{-\infty}^{\infty}\frac{1}{k_{\rho}}e^{\ri\bs k\cdot(\bs r-\bs r_{1\mathfrak b}')}\sigma_{\ell\ell'}^{1\mathfrak b}(k_{\rho})dk_x dk_y,\\
&\tilde u_{\ell\ell'}^{2\mathfrak b}(\bs r, \bs r'_{2\mathfrak b})=\frac{1}{8\pi^2 }\int_{-\infty}^{\infty}\int_{-\infty}^{\infty}\frac{1}{k_{\rho}}e^{\ri\bs k\cdot\bs\tau(\bs r-\bs r_{2\mathfrak b}')}\sigma_{\ell\ell'}^{2\mathfrak b}(k_{\rho})dk_x dk_y,
\end{split}
\end{equation}
where $z'_{\mathfrak{ab}}$ denotes the $z$-coordinate of $\bs r'_{\mathfrak{ab}}$, i.e.,
\begin{equation}\label{epsourcezcoord}
\begin{split}
&z_{11}^{\prime}=d_{\ell}-(z'-d_{\ell'}),\qquad z_{12}^{\prime}=d_{\ell}-(d_{\ell'-1}-z'),\\
&z_{21}^{\prime}=d_{\ell-1}+(z'-d_{\ell'}),\quad z_{22}^{\prime}=d_{\ell-1}+(d_{\ell'-1}-z').
\end{split}
\end{equation}
Apparently, we can verify that
\begin{equation}\label{expkernelexp}
\bs\tau_{\ell\ell'}^{1\mathfrak{b}}(\bs r, \bs r')=\bs r-\bs r_{1\mathfrak b}',\quad \bs\tau_{\ell\ell'}^{2\mathfrak{b}}(\bs r, \bs r')=\bs\tau(\bs r-\bs r_{2\mathfrak b}'), \quad \mathfrak b=1, 2.
\end{equation}
Therefore, the reaction components of layered Green's function defined in \eqref{generalcomponents} is equal to the introduced reaction potentials associated to equivalent polarization sources, \textit{i.e.},
\begin{equation}\label{generalcomponentsimag}
u_{\ell\ell'}^{1\mathfrak b}(\bs r, \bs r')=\tilde u_{\ell\ell'}^{1\mathfrak b}(\bs r, \bs r'_{1\mathfrak b}),\quad
u_{\ell\ell'}^{2\mathfrak b}(\bs r, \bs r')=\tilde u_{\ell\ell'}^{2\mathfrak b}(\bs r, \bs r'_{2\mathfrak b}),\quad \mathfrak b=1, 2.
\end{equation}

In the FMM for reaction components (cf. \cite{wang2019fastlaplace}), the expressions \eqref{reactioncomponentsepsource} with equivalent polarization sources  \eqref{eqpolarizedsource} are actually used. MEs, LEs and M2L translations for re-expressed reaction field components \eqref{reactioncomponentsepsource} are adopted in the FMM and we have verified numerically that the convergence of the MEs, LEs and M2L translations of \eqref{reactioncomponentsepsource} are determined by the Euclidean distance between target and equivalent polarization source. In the next two subsections, we first review the MEs, LEs and M2Ls for the re-expressed reaction field components \eqref{reactioncomponentsepsource} and then prove that all of them have exponential convergence with rates depends on the Euclidean distance between targets and equivalent polarization sources.

\subsection{MEs, LEs, and M2L translations for reaction filed components using new expressions with equivalent polarization sources} By the definition \eqref{generalintegraldef} and the linear features \eqref{reflection} of $\bs\tau(\bs r)$, the reaction components in \eqref{reactioncomponentsepsource} can be represented as
\begin{equation}\label{reactcompintegralrep}
\begin{split}
&\tilde u_{\ell\ell'}^{1\mathfrak b}(\bs r, \bs r'_{1\mathfrak b})=\frac{1}{8\pi^2 }\mathcal I(\bs r-\bs r_c^{1\mathfrak{b}}-(\bs r_{1\mathfrak{b}}'-\bs r_c^{1\mathfrak{b}});\sigma_{\ell\ell'}^{1\mathfrak{b}})=\frac{1}{8\pi^2 }\mathcal I(\bs r-\bs r_c^t+(\bs r_c^t-\bs r_{1\mathfrak{b}}');\sigma_{\ell\ell'}^{1\mathfrak{b}}),\\
&\tilde u_{\ell\ell'}^{2\mathfrak b}(\bs r, \bs r'_{1\mathfrak b})=\frac{1}{8\pi^2 }\mathcal I(\bs\tau(\bs r-\bs r_c^{2\mathfrak{b}})-\bs\tau(\bs r_{2\mathfrak{b}}'-\bs r_c^{2\mathfrak{b}});\sigma_{\ell\ell'}^{2\mathfrak{b}})=\frac{1}{8\pi^2 }\mathcal I(\bs\tau(\bs r-\bs r_c^t)-\bs\tau(\bs r_c^t-\bs r_{2\mathfrak{b}}');\sigma_{\ell\ell'}^{2\mathfrak{b}}),
\end{split}
\end{equation}
where $\bs r^{\mathfrak {ab}}_c=(x_c^{\mathfrak{ab}},y_c^{\mathfrak{ab}}, z_c^{\mathfrak{ab}} )$ and $\bs r_c^t=(x_c^t, y_c^t, z_c^t)$ are given equivalent polarization source and target centers such that
\begin{equation}\label{imagecentercond}
z_c^{1\mathfrak b}<d_{\ell},  \quad z_c^{2\mathfrak b}>d_{\ell-1},\quad d_{\ell}<z_c^t<d_{\ell-1} .
\end{equation}
These restriction can be met in practice, as we are considering targets in the $\ell$-th layer and the equivalent polarized coordinates are always located either above the interface $z=d_{\ell-1}$ or below the interface $z=d_{\ell}$.

By \eqref{epsourcezcoord} and conditions in \eqref{imagecentercond}, we have
\begin{equation}\label{coordrestrictions}
z-z_c^{1\mathfrak b}>0,\quad z_c^t-z_{1\mathfrak b}'>0,\quad z_c^{2\mathfrak{b}}-z>0,\quad z_{2\mathfrak{b}}'-z_c^t>0,\quad z-z_{1\mathfrak b}'>0,\quad z_{2\mathfrak{b}}'-z>0.
\end{equation}
Assume the centers $\bs r_c^{\mathfrak{ab}}$ and $\bs r_c^t$ satisfy $|\bs r-\bs r_c^{\mathfrak{ab}}|>|\bs r_{\mathfrak{ab}}'-\bs r_c^{\mathfrak{ab}}|$, and $|\bs r-\bs r_c^t|<|\bs r_c^t-\bs r_{\mathfrak{ab}}'|$, then \eqref{coordrestrictions} and Proposition \ref{Prop:densityprop} implies that theorem \ref{Thm:generalintegral} can be applied to give expansions for the integrals in \eqref{reactcompintegralrep}, \textit{i.e.},
\begin{equation}\label{reactmetaylor}
\begin{split}
\tilde u_{\ell\ell'}^{1\mathfrak b}(\bs r, \bs r'_{1\mathfrak b})=&\frac{1}{8\pi^2 }\sum\limits_{n=0}^{\infty}\int_{-\infty}^{\infty}\int_{-\infty}^{\infty}e^{\ri\bs k\cdot(\bs r-\bs r_c^{1\mathfrak b})}\frac{[-\ri\bs k\cdot (\bs r_{1\mathfrak b}'-\bs r_c^{1\mathfrak b})]^n}{n!k_{\rho}}\sigma_{\ell\ell'}^{1\mathfrak b}(k_{\rho})dk_x dk_y,\\
\tilde u_{\ell\ell'}^{2\mathfrak b}(\bs r, \bs r'_{2\mathfrak b})=&\frac{1}{8\pi^2 }\sum\limits_{n=0}^{\infty}\int_{-\infty}^{\infty}\int_{-\infty}^{\infty}e^{\ri\bs k\cdot\tau(\bs r-\bs r_c^{2\mathfrak b})}\frac{[-\ri \bs k\cdot \tau(\bs r_{2\mathfrak b}'-\bs r_c^{2\mathfrak b})]^n}{n!k_{\rho}}\sigma_{\ell\ell'}^{2\mathfrak b}(k_{\rho})dk_x dk_y
\end{split}
\end{equation}
and
\begin{equation}\label{reactletaylor}
\begin{split}
\tilde u_{\ell\ell'}^{1\mathfrak b}(\bs r, \bs r'_{1\mathfrak b})=&\frac{1}{8\pi^2 }\sum\limits_{n=0}^{\infty}\int_{-\infty}^{\infty}\int_{-\infty}^{\infty}e^{\ri\bs k\cdot(\bs r_c^t-\bs r_{1\mathfrak b}')}\frac{[\ri\bs k\cdot(\bs r-\bs r_c^t)]^n}{n!k_{\rho}}\sigma_{\ell\ell'}^{1\mathfrak b}(k_{\rho})dk_x dk_y,\\
\tilde u_{\ell\ell'}^{2\mathfrak b}(\bs r, \bs r'_{2\mathfrak b})=&\frac{1}{8\pi^2 }\sum\limits_{n=0}^{\infty}\int_{-\infty}^{\infty}\int_{-\infty}^{\infty}e^{\ri\bs k\cdot\tau(\bs r_c^t-\bs r_{2\mathfrak b}')}\frac{[\ri\bs k\cdot\tau(\bs r-\bs r_c^t)]^n}{n!k_{\rho}}\sigma_{\ell\ell'}^{2\mathfrak b}(k_{\rho})dk_x dk_y.
\end{split}
\end{equation}

Further, applying Proposition \ref{lemma3} in expansions \eqref{reactmetaylor} and using identities \eqref{sphrelations} again to simplify the obtained results, we obtain MEs
\begin{equation}\label{melayerupgoingimage}
\begin{split}
\tilde u_{\ell\ell'}^{\mathfrak{ab}}(\bs r, \bs r'_{\mathfrak{ab}})=\sum\limits_{n=0}^{\infty}\sum\limits_{m=-n}^{n}  M_{nm}^{\mathfrak{ab}}\widetilde{\mathcal F}_{nm}^{\mathfrak{ab}}(\bs r, \bs r_c^{\mathfrak{ab}}), \quad M_{nm}^{\mathfrak{ab}}=\frac{c_n^{-2}}{4\pi} (r_c^{\mathfrak{ab}})^n\overline{Y_n^{m}(\theta_c^{\mathfrak{ab}},\varphi_c^{\mathfrak{ab}})},
\end{split}
\end{equation}
at equivalent polarization source centers $\bs r_c^{\mathfrak{ab}}$ and LEs
\begin{equation}\label{lelayerimage}
\begin{split}
\tilde u_{\ell\ell'}^{\mathfrak{ab}}(\bs r, \bs r'_{\mathfrak{ab}})=\sum\limits_{n=0}^{\infty}\sum\limits_{m=-n}^{n} L_{nm}^{\mathfrak{ab}}r_t^nY_n^m(\theta_t,\varphi_t)
\end{split}
\end{equation}
at target center $\bs r_c^t$, respectively. Here, $\widetilde{\mathcal F}_{nm}^{\mathfrak{ab}}(\bs r, \bs r_c^{\mathfrak{ab}})$ are represented by Sommerfeld-type integrals
\begin{equation}\label{mebasis}
\begin{split}
\widetilde{\mathcal F}_{nm}^{1\mathfrak b}(\bs r, \bs r_c^{1\mathfrak b})=&\frac{(-1)^{n}c_n^2C_n^m}{2\pi}\int_{-\infty}^{\infty}\int_{-\infty}^{\infty}e^{\ri\bs k\cdot(\bs r-\bs r_c^{1\mathfrak b})}\sigma_{\ell\ell'}^{1\mathfrak b}(k_{\rho})k_{\rho}^{n-1}e^{\ri m\alpha}dk_x dk_y,\\
\widetilde{\mathcal F}_{nm}^{2\mathfrak b}(\bs r, \bs r_c^{2\mathfrak b})=&\frac{(-1)^mc_n^2C_n^m}{2\pi}\int_{-\infty}^{\infty}\int_{-\infty}^{\infty}e^{\ri\bs k\cdot\tau(\bs r-\bs r_c^{2\mathfrak b})}\sigma_{\ell\ell'}^{2\mathfrak b}(k_{\rho})k_{\rho}^{n-1}e^{\ri m\alpha}dk_x dk_y,
\end{split}
\end{equation}
and the local expansion coefficients are given by
\begin{equation}\label{lecoeffimage}
\begin{split}
L_{nm}^{1\mathfrak b}=&\frac{C_n^m}{8\pi^2 }\int_{-\infty}^{\infty}\int_{-\infty}^{\infty}e^{\ri\bs k\cdot(\bs r_c^t-\bs r_{1\mathfrak b}')}\sigma_{\ell\ell^{\prime}}^{1\mathfrak b}(k_{\rho})k_{\rho}^{n-1}e^{-\ri m\alpha}dk_x dk_y,\\
L_{nm}^{2\mathfrak b}=&\frac{(-1)^{n+m}C_n^m}{8\pi^2 }\int_{-\infty}^{\infty}\int_{-\infty}^{\infty}e^{\ri\bs k\cdot\tau(\bs r_c^t-\bs r_{2\mathfrak b}')}\sigma_{\ell\ell^{\prime}}^{2\mathfrak b}(k_{\rho})k_{\rho}^{n-1}e^{-\ri m\alpha}dk_x dk_y.
\end{split}
\end{equation}


Next, we discuss the center shifting and translation operators for ME \eqref{melayerupgoingimage} and LE \eqref{lelayerimage}. A desirable feature of the expansions of reaction components discussed above is that the formula \eqref{melayerupgoingimage} for the ME coefficients and the formula \eqref{lelayerimage} for the LE have exactly the same form as the formulas of ME coefficients and LE for free space Green's function. Therefore, we can see that center shifting for multipole and local expansions are exactly the same as free space case given in \eqref{metome}.

We only need to derive the translation operator from ME \eqref{melayerupgoingimage} to LE \eqref{lelayerimage}.
As in \eqref{reactcompintegralrep}, the reaction components in \eqref{reactioncomponentsepsource} can be represented as
\begin{equation}\label{reactcompme2lerep}
\begin{split}
&\tilde u_{\ell\ell'}^{1\mathfrak b}(\bs r, \bs r'_{1\mathfrak b})=\frac{1}{8\pi^2 }\mathcal I(\bs r-\bs r_c^t+(\bs r_c^t-\bs r_c^{1\mathfrak{b}})-(\bs r_{1\mathfrak{b}}'-\bs r_c^{1\mathfrak{b}});\sigma_{\ell\ell'}^{1\mathfrak{b}}),\\
&\tilde u_{\ell\ell'}^{2\mathfrak b}(\bs r, \bs r'_{1\mathfrak b})=\frac{1}{8\pi^2 }\mathcal I(\bs\tau(\bs r-\bs r_c^t)+\bs\tau(\bs r_c^t-\bs r_c^{2\mathfrak{b}})-\bs\tau(\bs r_{2\mathfrak{b}}'-\bs r_c^{2\mathfrak{b}});\sigma_{\ell\ell'}^{2\mathfrak{b}}).
\end{split}
\end{equation}
Apparently, from \eqref{imagecentercond}, we have
\begin{equation}\label{MELEcentercond}
z_c^t-z_c^{1\mathfrak b}>0, \quad z_c^{2\mathfrak b}-z_c^t>0.
\end{equation}
Assume the given centers $\bs r_c^{\mathfrak{ab}}$ and $\bs r_c^t$ satisfy $|\bs r_c^t-\bs r_c^{\mathfrak{ab}}|>|\bs r_{\mathfrak{ab}}'-\bs r_c^{\mathfrak{ab}}|+|\bs r-\bs r_c^t|$, then \eqref{coordrestrictions}, \eqref{MELEcentercond} and Proposition \ref{Prop:densityprop} implies that Theorem \ref{Thm:generalintegral} can be applied to give expansions for the integrals in \eqref{reactcompme2lerep}, \textit{i.e.},
\begin{equation}\label{reactme2letaylor}
\begin{split}
\tilde u_{\ell\ell'}^{1\mathfrak b}(\bs r, \bs r'_{1\mathfrak b})=&\frac{1}{8\pi^2 }\sum\limits_{n=0}^{\infty}\sum\limits_{\nu=0}^{\infty}\mathcal I_{n\nu}(\bs r_c^t-\bs r_c^{1\mathfrak{b}}, \bs r-\bs r_c^t, -(\bs r_{1\mathfrak{b}}'-\bs r_c^{1\mathfrak{b}});\sigma_{\ell\ell'}^{1\mathfrak b}),\\
\tilde u_{\ell\ell'}^{2\mathfrak b}(\bs r, \bs r'_{2\mathfrak b})=&\frac{1}{8\pi^2 }\sum\limits_{n=0}^{\infty}\sum\limits_{\nu=0}^{\infty}\mathcal I_{n\nu}(\bs\tau(\bs r_c^t-\bs r_c^{2\mathfrak{b}}), \bs\tau(\bs r-\bs r_c^t), -\bs\tau(\bs r_{1\mathfrak{b}}'-\bs r_c^{1\mathfrak{b}});\sigma_{\ell\ell'}^{2\mathfrak b}).
\end{split}
\end{equation}
Applying Proposition \ref{lemma3} to the integrand of $\mathcal{I}_{n\nu}(\bs r, \bs r', \bs r'',\sigma)$, we obtain
the translation from ME \eqref{melayerupgoingimage} to LE \eqref{lelayerimage}, i.e.,
\begin{equation}\label{metoleimage1}
L_{nm}^{\mathfrak{ab}}=\sum\limits_{n'=0}^{\infty}\sum\limits_{|m'|=0}^{n'}T_{nm,n'm'}^{\mathfrak{ab}}M_{n'm'}^{\mathfrak{ab}},
\end{equation}
where the translation operators are given as follows
\begin{equation}\label{metoleimage2}
\begin{split}
T_{nm,n'm'}^{1\mathfrak{b}}=&\frac{D_{nmn'm'}^1}{2\pi}\int_{-\infty}^{\infty}\int_{-\infty}^{\infty}e^{\ri\bs k\cdot(\bs r_c^t-\bs r_c^{1\mathfrak b})}\sigma_{\ell\ell'}^{1\mathfrak{b}}(k_{\rho})k_{\rho}^{n+n'-1}e^{\ri (m'-m)\alpha}dk_x dk_y,\\
T_{nm,n'm'}^{2\mathfrak{b}}=&\frac{D_{nmn'm'}^2}{2\pi}\int_{-\infty}^{\infty}\int_{-\infty}^{\infty}e^{\ri\bs k\cdot\tau(\bs r_c^t-\bs r_c^{2\mathfrak b})}\sigma_{\ell\ell'}^{2\mathfrak{b}}(k_{\rho})k_{\rho}^{n+n'-1}e^{\ri (m'-m)\alpha}dk_x dk_y,
\end{split}
\end{equation}
where
$$D_{nmn'm'}^1=(-1)^{n'}c_{n'}^2C_n^mC_{n'}^{m'},\quad D_{nmn'm'}^2=(-1)^{n+m+m'}c_{n'}^2C_n^mC_{n'}^{m'}.$$

\subsection{Exponential convergence of MEs, LEs and corresponding translation operators for reaction components} Let $\Phi_{\ell\ell',{\rm in}}^{\mathfrak{ab}}(\bs r)$ and $\Phi_{\ell\ell',{\rm out}}^{\mathfrak{ab}}(\bs r)$ be general reaction components of potentials induced by all equivalent polarizaion sources inside a given source box $B_s^{\mathfrak{ab}}$ centered at $\bs r_c^{\mathfrak{ab}}$ and far away from a given target box $B_t$ centered at $\bs r_c^t$, i.e.,
\begin{equation}\label{reactionpotentials}
\Phi_{\ell\ell',{\rm in}}^{\mathfrak{ab}}(\bs r)=\sum\limits_{j\in \mathcal J}Q_{\ell' j}\tilde u_{\ell\ell'}^{\mathfrak{ab}}(\bs r, \bs r_{\ell'j}^{\mathfrak{ab}}),\quad \Phi_{\ell\ell',{\rm out}}^{\mathfrak{ab}}(\bs r)=\sum\limits_{j\in \mathcal K}Q_{\ell' j}\tilde u_{\ell\ell'}^{\mathfrak{ab}}(\bs r, \bs r_{\ell'j}^{\mathfrak{ab}}),
\end{equation}
where $\mathcal J$ and $\mathcal K$ are the sets of indices of equivalent polarizaion sources inside $B_s^{\mathfrak{ab}}$ and far away from $B_t$, respectively.
The FMM for the reaction component $\Phi_{\ell\ell'}^{\mathfrak{ab}}(\bs r)$ use ME
\begin{equation}\label{mereact}
\Phi_{\ell\ell',{\rm in}}^{\mathfrak{ab}}(\bs r)=\sum\limits_{n=0}^{\infty}\sum\limits_{m=-n}^nM_{nm}^{\mathfrak{ab},\rm in}\widetilde{\mathcal F}_{nm}^{\mathfrak{ab}}(\bs r, \bs r_c^{\mathfrak{ab}}),
\end{equation}
at any target points far away from $B_s^{\mathfrak{ab}}$ and LE
\begin{equation}\label{lereactapprox}
\Phi_{\ell\ell',{\rm out}}^{\mathfrak{ab}}(\bs r)=\sum\limits_{n=0}^{\infty}\sum\limits_{m=-n}^nL_{nm}^{\mathfrak{ab},\rm out}r_t^nY_n^m(\theta_t,\varphi_t),
\end{equation}
inside $B_t$, where the coefficients are given by
\begin{equation}\label{reactexpcoefbox}
\begin{split}
M_{nm}^{\mathfrak{ab},\rm in}=&\frac{c_n^{-2}}{4\pi}\sum\limits_{j\in\mathcal J}Q_{\ell' j}(r_{\ell' j}^{\mathfrak{ab}})^n\overline{Y_n^{m}(\theta_{\ell' j}^{\mathfrak{ab}},\varphi_{\ell' j}^{\mathfrak{ab}})},\\
L_{nm}^{1\mathfrak b,\rm out}=&\frac{C_n^m}{8\pi^2 }\int_{-\infty}^{\infty}\int_{-\infty}^{\infty}e^{\ri\bs k\cdot(\bs r_c^t- \bs r_{\ell'j}^{1\mathfrak b})}\sigma_{\ell\ell^{\prime}}^{1\mathfrak b}(k_{\rho})k_{\rho}^{n-1}e^{-\ri m\alpha}dk_x dk_y,\\
L_{nm}^{2\mathfrak b,\rm out}=&\frac{(-1)^{n+m}C_n^m}{8\pi^2 }\int_{-\infty}^{\infty}\int_{-\infty}^{\infty}e^{\ri\bs k\cdot\tau(\bs r_c^t-\bs r_{\ell'j}^{2\mathfrak b})}\sigma_{\ell\ell^{\prime}}^{2\mathfrak b}(k_{\rho})k_{\rho}^{n-1}e^{-\ri m\alpha}dk_x dk_y,
\end{split}
\end{equation}
$(r_t, \theta_t, \phi_t)$ and $(r_{\ell' j}^{\mathfrak{ab}}, \theta_{\ell' j}^{\mathfrak{ab}}, \phi_{\ell' j}^{\mathfrak{ab}})$ are the spherical coordinates of $\bs r-\bs r_c^t$ and $\bs r_{\ell' j}^{\mathfrak{ab}}-\bs r_c^{\mathfrak{ab}}$.
These expansions can be obtained by applying expansions \eqref{melayerupgoingimage}-\eqref{lelayerimage} to each $\tilde u_{\ell\ell'}^{\mathfrak{ab}}(\bs r, \bs r_{\ell'j}^{\mathfrak{ab}})$ involved in the summations in \eqref{reactionpotentials}.
\begin{figure}[ht!]\label{reactionmeandle}
	\centering
	\includegraphics[scale=0.6]{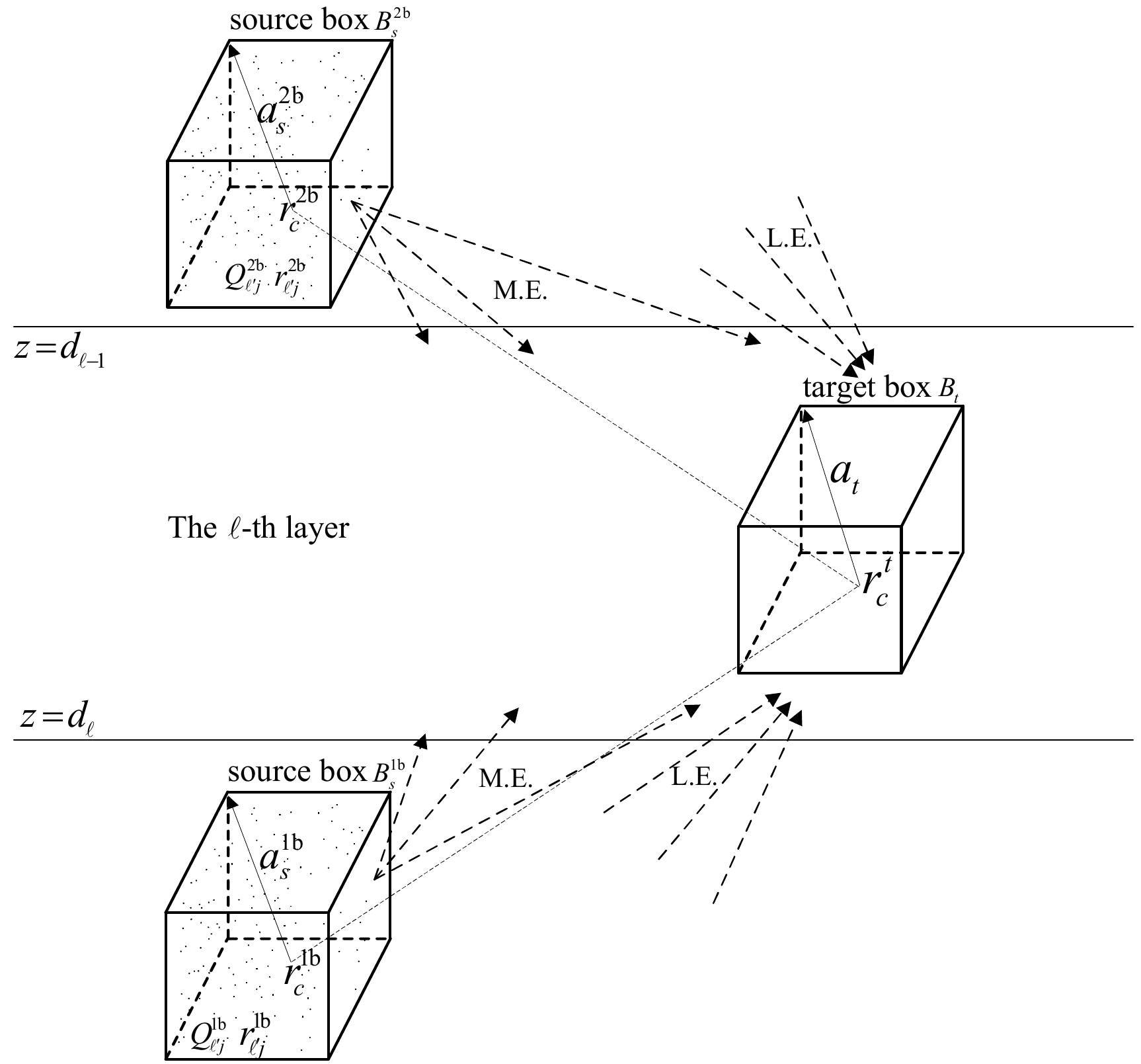}
	\caption{An illustration of equivalent polarization source and target box for the reaction components in the $\ell$-th layer due to sources in $\ell'$-th layer.}
\end{figure}
\begin{theorem}\label{reactmeconvergence1}
	Suppose $a_s^{\mathfrak{ab}}$ is the radius of the circumscribed sphere of the source box $B_s^{\mathfrak{ab}}$, $\bs r$ is a point outside the circumscribed sphere of $B_s^{\mathfrak{ab}}$, i.e., $r_s^{\mathfrak ab}:=|\bs r-\bs r_c^{\mathfrak{ab}}|>a_s^{\mathfrak ab}$, then ME \eqref{mereact} has error estimate
	\begin{equation}\label{reactmereactapprox}
	\Big|\Phi_{\ell\ell',{\rm in}}^{\mathfrak{ab}}(\bs r)-\sum\limits_{n=0}^{p}\sum\limits_{m=-n}^nM_{nm}^{\mathfrak{ab},\rm in}\widetilde{\mathcal F}_{nm}^{\mathfrak{ab}}(\bs r, \bs r_c^{\mathfrak{ab}})\Big|\leq \frac{1}{4\pi}\frac{Q_{\mathcal J}\bs M_{\sigma_{\ell\ell'}^{\mathfrak{ab}}}}{r_s^{\mathfrak{ab}}-a_s^{\mathfrak{ab}}}\Big(\frac{a_s^{\mathfrak{ab}}}{r_s^{\mathfrak{ab}}}\Big)^{p+1},
	\end{equation}
	where $\bs M_{\sigma_{\ell\ell'}^{\mathfrak{ab}}}$ is the bound of  $\sigma_{\ell\ell'}^{\mathfrak{ab}}(k_{\rho})$ in the right half complex plane,
	\begin{equation}\label{totalcharge}
	Q_{\mathcal J}=\sum\limits_{j\in\mathcal J}|Q_{\ell' j}|.
	\end{equation}
\end{theorem}
\begin{proof}
	As the MEs \eqref{melayerupgoingimage} are obtained by directly applying Proposition \ref{lemma3} to the Taylor expansions  \eqref{reactmetaylor}, we have
	\begin{equation}\label{reactmereformula}
	\begin{split}
	\sum\limits_{n=0}^{p}\sum\limits_{m=-n}^nM_{nm}^{1\mathfrak{b},\rm in}\widetilde{\mathcal F}_{nm}^{1\mathfrak{b}}(\bs r, \bs r_c^{1\mathfrak{b}})
	=&\frac{1}{8\pi^2 }\sum\limits_{j\in\mathcal J}Q_{\ell' j}\sum\limits_{n=0}^{p}\mathcal I_n(\bs r-\bs r_c^{1\mathfrak{b}},-({\bs r_{\ell' j}^{1\mathfrak{b}}-\bs r_c^{1\mathfrak{b}}});\sigma_{\ell\ell'}^{1\mathfrak{b}}),\\
	\sum\limits_{n=0}^{p}\sum\limits_{m=-n}^nM_{nm}^{2\mathfrak{b},\rm in}\widetilde{\mathcal F}_{nm}^{2\mathfrak{b}}(\bs r, \bs r_c^{2\mathfrak{b}})
	=&\frac{1}{8\pi^2 }\sum\limits_{j\in\mathcal J}Q_{\ell' j}\sum\limits_{n=0}^{p}\mathcal I_n(\tau(\bs r-\bs r_c^{2\mathfrak b}), -\tau({\bs r_{\ell' j}^{2\mathfrak{b}}-\bs r_c^{2\mathfrak{b}}});\sigma_{\ell\ell'}^{2\mathfrak{b}}).
	\end{split}
	\end{equation}
	By conditions in \eqref{coordrestrictions}, we have $z-z_c^{1\mathfrak{b}}>0$ and $z_c^{2\mathfrak b}-z>0$. Together with the assumption $|\bs r-\bs r_c^{\mathfrak{ab}}|>a_s^{\mathfrak{ab}}\geq|\bs r_{\ell'j}^{\mathfrak{ab}}-\bs r_c^{\mathfrak{ab}}|$, we can apply the truncation error estimates  \eqref{reactionintegralestimate1} to obtain
	\begin{equation*}
	\begin{split}
	&\Big|\Phi_{\ell\ell',{\rm in}}^{1\mathfrak{b}}(\bs r)-\sum\limits_{n=0}^{p}\sum\limits_{m=-n}^nM_{nm}^{1\mathfrak{b},\rm in}\widetilde{\mathcal F}_{nm}^{1\mathfrak{b}}(\bs r, \bs r_c^{1\mathfrak{b}})\Big|\\
	\leq&\frac{1}{8\pi^2}\sum_{j\in\mathcal J}|Q_{\ell'j}|\Big|\mathcal I(\bs r, \bs r_{\ell'j}^{1\mathfrak{b}}; \sigma_{\ell\ell'}^{1\mathfrak{b}})-\sum\limits_{n=0}^{p}\mathcal I_n(\bs r-\bs r_c^{1\mathfrak{b}},-({\bs r_{\ell' j}^{1\mathfrak{b}}-\bs r_c^{1\mathfrak{b}}});\sigma_{\ell\ell'}^{1\mathfrak{b}})\Big|,\\
	\leq&\sum_{j\in\mathcal J}|Q_{\ell'j}|\frac{ (4\pi)^{-1}\bs M_{\sigma_{\ell\ell'}^{1\mathfrak b}}}{(|\bs r-\bs r_c^{1\mathfrak{b}}|-|{\bs r_{\ell' j}^{1\mathfrak{b}}-\bs r_c^{1\mathfrak{b}}}|)}\Big|\frac{{\bs r_{\ell' j}^{1\mathfrak{b}}-\bs r_c^{1\mathfrak{b}}}}{\bs r-\bs r_c^{1\mathfrak{b}}}\Big|^{p+1},
	\end{split}
	\end{equation*}
	and similarly
	\begin{equation*}
	\Big|\Phi_{\ell\ell',{\rm in}}^{2\mathfrak{b}}(\bs r)-\sum\limits_{n=0}^{p}\sum\limits_{m=-n}^nM_{nm}^{2\mathfrak{b},\rm in}\widetilde{\mathcal F}_{nm}^{2\mathfrak{b}}(\bs r, \bs r_c^{2\mathfrak{b}})\Big|\leq\sum_{j\in\mathcal J}\frac{(4\pi)^{-1}|Q_{\ell'j}| \bs M_{\sigma_{\ell\ell'}^{2\mathfrak b}}}{(|\bs r-\bs r_c^{2\mathfrak{b}}|-|{\bs r_{\ell' j}^{2\mathfrak{b}}-\bs r_c^{2\mathfrak{b}}}|)}\Big|\frac{{\bs r_{\ell' j}^{2\mathfrak{b}}-\bs r_c^{2\mathfrak{b}}}}{\bs r-\bs r_c^{2\mathfrak{b}}}\Big|^{p+1}.
	\end{equation*}
	Consequently, the error estimate \eqref{reactmereactapprox} follows by applying the above estimates in \eqref{reactmereformula} with the assumption $|\bs r_{\ell'j}^{\mathfrak{ab}}-\bs r_c^{\mathfrak{ab}}|\leq a_s^{\mathfrak ab}<r_s^{\mathfrak ab}$.
\end{proof}

Following a similar proof, we have the error estimate for the truncated LE as follows:
\begin{theorem}\label{reactleconvergence}
	Suppose $a_t$ is the radius of the circumscribed sphere of the target box $B_t$, $\bs r$ is a point inside $B_t$, $\mathcal K$ is the set of indices of all charges $(Q_{\ell'j},\bs r_{\ell' j}^{\mathfrak{ab}})$ such that $|\bs r_{\ell' j}^{\mathfrak{ab}}-\bs r_c^t|>a_t$, then the LE \eqref{lereactapprox} has error estimate
	\begin{equation}\label{reactlereactapprox}
	\Big|\Phi_{\ell\ell',{\rm out}}^{\mathfrak{ab}}(\bs r)-\sum\limits_{n=0}^{p}\sum\limits_{m=-n}^nL_{nm}^{\mathfrak{ab},\rm out}r_t^nY_n^m(\theta_t,\varphi_t)\Big|\leq \frac{1}{4\pi}\frac{Q_{\mathcal K} \bs M_{\sigma_{\ell\ell'}^{\mathfrak{ab}}} }{a_t-r_t}\Big(\frac{r_t}{a_t}\Big)^{p+1},
	\end{equation}
	where $\bs M_{\sigma_{\ell\ell'}^{\mathfrak{ab}}}$ is the bound of  $\sigma_{\ell\ell'}^{\mathfrak{ab}}(k_{\rho})$ in the right half complex plane,
	\begin{equation}
	Q_{\mathcal K}=\sum\limits_{j\in\mathcal K}|Q_{\ell' j}|.
	\end{equation}
\end{theorem}

Now we consider the error estimate for the ME to LE translation. Suppose the target box $B_t$ is far away from the source box $B_s^{\mathfrak{ab}}$. Recall \eqref{lelayerimage}, the LE of the potential $\Phi_{\ell\ell',\rm in}^{\mathfrak{ab}}$ in $B_t$ is given by
\begin{equation}\label{reactmetole}
\begin{split}
\Phi_{\ell\ell',{\rm in}}^{\mathfrak{ab}}(\bs r)=\sum\limits_{n=0}^{\infty}\sum\limits_{m=-n}^n L_{nm}^{\mathfrak{ab},\rm in} r_t^nY_n^m(\theta_t,\varphi_t),\quad \forall \bs r\in B_t,
\end{split}
\end{equation}
while the LE coefficients $L_{nm}^{\rm \mathfrak{ab},in}$ can be calculated from ME coefficients via the ME to LE translation operator \eqref{metoleimage1} as follows
\begin{equation}\label{reactmetolenotruncate}
L_{nm}^{\rm \mathfrak{ab},in}=\sum\limits_{\nu=0}^{\infty}\sum\limits_{\mu=-\nu}^{\nu}T_{nm,\nu\mu}^{\mathfrak{ab}}M_{\nu\mu}^{\mathfrak{ab},\rm in}.
\end{equation}
As in the FMM for free space components, \eqref{reactmetole} is not the approximation used in the implementation. In fact, the formulas \eqref{reactmetolenotruncate} for LE coefficients $L_{nm}^{\mathfrak{ab},\rm in}$ are truncated which gives approximated LE coefficients
\begin{equation}\label{reactmetoletruncate}
L_{nm}^{\mathfrak{ab},p}=\sum\limits_{\nu=0}^{p}\sum\limits_{\mu=-\nu}^{\nu}T_{nm,\nu\mu}^{\mathfrak{ab}}M_{\nu\mu}^{\mathfrak{ab},\rm in}.
\end{equation}
Thus, approximate LEs
\begin{equation}
\Phi_{\ell\ell',{\rm in}}^{\mathfrak{ab}}(\bs r)\approx \Phi_{\ell\ell',{\rm in}}^{\mathfrak{ab},p}(\bs r):=\sum\limits_{n=0}^{p}\sum\limits_{m=-n}^n L_{nm}^{\mathfrak{ab},p} r_t^nY_n^m(\theta_t,\varphi_t)
\end{equation}
with approximate LE coefficients defined in \eqref{reactmetoletruncate} are obtained after M2L translation. Recalling representation \eqref{reactcompme2lerep} and expansion \eqref{reactme2letaylor}, the approximate LEs $\Phi_{\ell\ell',{\rm in}}^{\mathfrak{ab},p}(\bs r)$ have representations
\begin{equation}\label{metoleapprox}
\begin{split}
\Phi_{\ell\ell',{\rm in}}^{1\mathfrak{b},p}(\bs r)&=\frac{1}{8\pi^2}\sum\limits_{j\in\mathcal J}Q_{\ell' j}\sum\limits_{n=0}^{p}\sum\limits_{\nu=0}^{p}\mathcal I_{n\nu}(\bs r_c^t-\bs r_c^{1\mathfrak{b}}, \bs r-\bs r_c^t, -(\bs r_{\ell'j}^{1\mathfrak{b}}-\bs r_c^{1\mathfrak{b}});\sigma_{\ell\ell'}^{1\mathfrak{b}}),\\
\Phi_{\ell\ell',{\rm in}}^{2\mathfrak{b},p}(\bs r)&=\frac{1}{8\pi^2}\sum\limits_{j\in\mathcal J}Q_{\ell' j}\sum\limits_{n=0}^{p}\sum\limits_{\nu=0}^{p}\mathcal I_{n\nu}(\tau(\bs r_c^t-\bs r_c^{2\mathfrak b}),\bs\tau({\bs r-\bs r_c^t}), -\bs\tau({\bs r_{\ell'j}^{2\mathfrak{b}}-\bs r_c^{2\mathfrak{b}}});\sigma_{\ell\ell'}^{2\mathfrak{b}}).
\end{split}
\end{equation}
Obviously, they are rectangular truncation of the double Taylor series.
\begin{theorem}\label{reactmetoleconvergence}
	Suppose $a_s^{\mathfrak{ab}}$ and $a_t$ are the radii of the circumscribed spheres of two well separated boxes $B_s^{\mathfrak{ab}}$ and $B_t$, respectively. The well separateness of the boxes means that $|\bs r_c^t-\bs r_c^{\mathfrak{ab}}|>a_s^{\mathfrak{ab}}+ca_t$ with some $c>1$. Then, the ME to LE translation has error estimate
	\begin{equation}\label{me2leerrorest}
	\Big|\Phi_{\ell\ell',{\rm in}}^{\mathfrak{ab}}(\bs r)-\Phi_{\ell\ell',{\rm in}}^{\mathfrak{ab},p}(\bs r)\Big|\leq \frac{1}{2\pi}\frac{Q_{\mathcal J} \bs M_{\sigma_{\ell\ell'}^{\mathfrak{ab}}}}{(c-1)a_t}\Big(\frac{a_s+a_t}{a_s+ca_t}\Big)^{p+1},\forall \bs r\in B_t,
	\end{equation}
	where $\bs M_{\sigma_{\ell\ell'}^{\mathfrak{ab}}}$ is the bound of  $\sigma_{\ell\ell'}^{\mathfrak{ab}}(k_{\rho})$ in the right half complex plane, $Q_{\mathcal J}$ is defined in \eqref{totalcharge}.
\end{theorem}
\begin{proof}
	By expression \eqref{metoleapprox} and truncation error estimate \eqref{reactionintegralestimate2}, we obtain
	\begin{equation}
	\begin{split}
	&\big|\Phi_{\ell\ell',{\rm in}}^{1\mathfrak{b}}(\bs r)-\Phi_{\ell\ell',{\rm in}}^{1\mathfrak{b},p}(\bs r)\big|\\
	\leq&\frac{1}{8\pi^2} \sum\limits_{j\in\mathcal K}|Q_{\ell'j}|\Big|\mathcal I(\bs r-\bs r_{\ell'j}^{1\mathfrak{b}};\sigma_{\ell\ell'}^{1\mathfrak{b}})-\sum\limits_{n=0}^{p}\sum\limits_{\nu=0}^{p}\mathcal I_{n\nu}(\bs r_c^t-\bs r_c^{1\mathfrak{b}}, \bs r-\bs r_c^t, -(\bs r_{\ell'j}^{1\mathfrak{b}}-\bs r_c^{1\mathfrak{b}});\sigma_{\ell\ell'}^{1\mathfrak{b}})\Big|\\
	\leq& \sum\limits_{j\in\mathcal K}\frac{|Q_{\ell'j}|\bs M_{\sigma_{\ell\ell'}^{1\mathfrak{b}}}}{2\pi(|\bs r_c^t-\bs r_c^{1\mathfrak{b}}|-|\bs r-\bs r_c^t|-|\bs r_{\ell'j}^{1\mathfrak{b}}-\bs r_c^{1\mathfrak{b}}|)}\left(\frac{|\bs r-\bs r_c^t|+|\bs r_{\ell'j}^{1\mathfrak{b}}-\bs r_c^{1\mathfrak{b}}|}{|\bs r_c^t-\bs r_c^{1\mathfrak{b}}|}\right)^{p+1}.
	\end{split}
	\end{equation}
	Similar error estimate can also be obtained for the reaction component $\Phi_{\ell\ell',{\rm in}}^{2\mathfrak{b}}(\bs r)$ by following the same derivations.
	Consequently, the error estimate \eqref{me2leerrorest} follows by further applying the assumptions $|\bs r_{\ell'j}^{\mathfrak{ab}}-\bs r_c^{\mathfrak{ab}}|<a_s^{\mathfrak{ab}}$ and $|\bs r-\bs r_c^t|<a_t$ and $|\bs r_c^t-\bs r_c^{\mathfrak{ab}}|>a_s^{\mathfrak{ab}}+ca_t$.
\end{proof}

\begin{rem}
	The error estimates in Theorems \ref{reactmeconvergence1}-\ref{reactmetoleconvergence} are almost the same as the ones in Theorems \ref{meconvergence}, \ref{leconvergence} and \ref{metoleconvergence} except the bound $\bs M_{\sigma_{\ell\ell'}^{\mathfrak{ab}}}$ of $\sigma_{\ell\ell'}^{\mathfrak{ab}}$.
\end{rem}
\subsection{Detailed proof for the Theorem \ref{Thm:generalintegral}} The proof consists of the following three steps:

{\bf Step 1: Rotation according to the azimuthal angle of $\bs r$.} By the assumptions $z>0, z+z'>0$ and $z+z'+z''>0$, all improper integrals used in the Theorem \ref{Thm:generalintegral} are convergent.
Denote by $(\rho, \phi)$ the polar coordinate of $(x, y)$ and define rotational transform $\xi=k_x\cos\phi+k_y\sin\phi$, $\eta=k_x\sin\phi-k_y\cos\phi$, i.e., $\xi+\ri\eta=e^{\ri\phi}(k_x-\ri k_y)$. It is obvious that $k_{\rho}^2=\xi^2+\eta^2$ and
\begin{equation}\label{gnkernel}
\frac{(\ri \bs k\cdot\tilde{\bs r})^q}{q!}=\frac{\ri^{q}\tilde r^q}{q!}(\xi^2+\eta^2)^{\frac{q}{2}}\Big[\frac{\xi\cos(\phi-\beta)+\eta\sin(\phi-\beta)}{\sqrt{\xi^2+\eta^2}}\sin\alpha+\ri\cos\alpha\Big]^q:=\hat g_{q}(\xi, \eta,\phi;\tilde{\bs r}),
\end{equation}
for any $\tilde{\bs r}=(\tilde r\sin\alpha\cos\beta,\tilde r\sin\alpha\sin\beta,\tilde r\cos\alpha)\in\mathbb R^3$. Therefore, \eqref{twointegralexp} can be re-expressed as
\begin{equation}\label{integralEk}
\begin{split}
\mathcal I(\bs r+ \bs r';\sigma)=&\int_{-\infty}^{\infty}\int_{-\infty}^{\infty}\sum\limits_{n=0}^{\infty}\hat g_{n}(\xi, \eta,\phi;{\bs r}')e^{\ri\xi\rho-\zeta z}\sigma(\zeta)d\xi d\eta\\
=&\int_{0}^{\infty}\int_{-\infty}^{\infty}\sum\limits_{n=0}^{\infty}\big[\hat g_{n}(\xi, \eta,\phi;{\bs r}')+\hat g_{n}(\xi, -\eta,\phi;{\bs r}')\big]e^{\ri\xi\rho-\zeta z}\sigma(\zeta)d\xi d\eta,
\end{split}
\end{equation}
and
\begin{equation}\label{integralFk}
\begin{split}
\mathcal I(\bs r+\bs r'&+\bs r'';\sigma)=\int_{-\infty}^{\infty}\int_{-\infty}^{\infty}\sum\limits_{n=0}^{\infty}\sum\limits_{\nu=0}^{\infty}g_{n\nu}(\xi, \eta,\phi;{\bs r}',{\bs r}'')e^{\ri\xi\rho-\zeta z}\sigma(\zeta)d\xi d\eta\\
=&\int_{0}^{\infty}\int_{-\infty}^{\infty}\sum\limits_{n=0}^{\infty}\sum\limits_{\nu=0}^{\infty}\big[g_{n\nu}(\xi, \eta,\phi;{\bs r}',{\bs r}'')+g_{n\nu}(\xi, -\eta,\phi;{\bs r}',{\bs r}'')\big]e^{\ri\xi\rho-\zeta z}\sigma(\zeta)d\xi d\eta,
\end{split}
\end{equation}
where
\begin{equation}\label{gnnukernel}
\zeta=\sqrt{\xi^2+\eta^2},\quad g_{n\nu}(\xi, \eta,\phi;{\bs r}',{\bs r}''):=\frac{1}{\zeta}\hat g_n(\xi, \eta,\phi;{\bs r}')\hat g_{\nu}(\xi, \eta,\phi;{\bs r}'').
\end{equation}
Define
\begin{equation}\label{Endeflimit}
\begin{split}
\mathcal E_{n}(\bs r, \bs r',\sigma)&=\int_{0}^{\infty}\int_{-\infty}^{\infty}\hat g_{n}(\xi,\eta,\phi;{\bs r}')e^{\ri\xi\rho-\zeta z}\sigma(\zeta)d\xi d\eta,\\
\widetilde{\mathcal E}_{n}(\bs r, \bs r', \sigma)&=\int_{0}^{\infty}\int_{-\infty}^{\infty}\hat g_{n}(\xi,-\eta,\phi;{\bs r}')e^{\ri\xi\rho-\zeta z}\sigma(\zeta)d\xi d\eta,
\end{split}
\end{equation}
and
\begin{equation}\label{Fnnudeflimit}
\begin{split}
\mathcal F_{n\nu}(\bs r, \bs r', \bs r'',\sigma)&=\int_{0}^{\infty}\int_{-\infty}^{\infty}g_{n\nu}(\xi,\eta,\phi;{\bs r}',{\bs r}'')e^{\ri\xi\rho-\zeta z}\sigma(\zeta)d\xi d\eta,\\
\widetilde{\mathcal F}_{n\nu}(\bs r, \bs r', \bs r'',\sigma)&=\int_{0}^{\infty}\int_{-\infty}^{\infty}g_{n\nu}(\xi,-\eta,\phi;{\bs r}',{\bs r}'')e^{\ri\xi\rho-\zeta z}\sigma(\zeta)d\xi d\eta.
\end{split}
\end{equation}
Then, the integrals in \eqref{InInuintegrals} have representations
\begin{equation}\label{InnuEnFnu}
\begin{split}
\mathcal I_n(\bs r+ \bs r';\sigma)=&\mathcal E_{n}(\bs r, \bs r',\sigma)+\widetilde{\mathcal E}_{n}(\bs r, \bs r', \sigma),\\
\mathcal I_{n\nu}(\bs r+\bs r'+\bs r'';\sigma)=&\mathcal F_{n\nu}(\bs r, \bs r', \bs r'',\sigma)+\widetilde{\mathcal F}_{n\nu}(\bs r, \bs r', \bs r'',\sigma).
\end{split}
\end{equation}

{\bf Step 2: Contour deformation.} In the following analysis, we will deform the contour of the inner integral in \eqref{integralEk}-\eqref{integralFk}. As the integrands involve square root function $\zeta(\xi)=\sqrt{\xi^2+\eta^2}$, we choose branch as follow
\begin{equation}\label{squarerootbranch}
\sqrt{z}=\sqrt{\frac{|z|+z_1}{2}}+\ri\;{\rm sign}(z_2)\sqrt{\frac{|z|-z_1}{2}},\quad\forall z=z_1+\ri z_2\in\mathbb C.
\end{equation}
With this branch, $\zeta(\xi)=\sqrt{\xi^2+\eta^2}$ for any fixed $\eta\geq 0$ has branch cut along $\{\ri\xi: \xi>\eta\}$ and  $\{\ri\xi: \xi<\eta\}$ (the red lines in Fig. \ref{contour}) in the complex $\xi$-plane and is analytic with respect to $\xi$ in the complex domain $\mathbb C\setminus(\{\ri\xi: \xi\geq\eta\}\cup \{\ri\xi: \xi\leq\eta\})$.
The contour deformation will be based on the following lemma:
\begin{lemma}\label{contourchangelemma}
	Denote by $\Omega_{\Gamma}^+\subset\mathbb C$ the complex domain between real axis and the contour $\Gamma$ defined by the parametric $\xi_{\pm}(t)$ in \eqref{contourdef}. Let $f(\xi)$ be an analytic function in $\Omega_{\Gamma}^+$ and satisfy $|f(\xi)|\leq C|\xi|^m$ for some integer $m$ and some constant $C>0$. Then for any $\rho\geq 0, z>0$ and $\eta>0$, there holds
	\begin{equation}\label{integraltransformed}
	\int_{-\infty}^{\infty}f(\xi)e^{\ri\xi\rho-\sqrt{\eta^2+\xi^2}z}d\xi=\ri\int_{1}^{\infty}[f(\xi_{+}(t))\Lambda_+(t)+f(\xi_{-}(t)\Lambda_{-}(t)]\frac{e^{-\eta rt}}{\sqrt{t^2-1}}dt,
	\end{equation}
	where $r=\sqrt{\rho^2+z^2}$, and $\xi_{\pm}(t), \Lambda_{\pm}(t)$ are defined by the Cagniard-de Hoop transform
	\begin{equation}\label{contourdef}
	\xi_{\pm}(t)=\frac{\eta}{r}\big(\ri\rho t\pm z\sqrt{t^2-1}\big),\quad \Lambda_{\pm}(t)=\frac{\eta}{r}(\rho\sqrt{t^2-1}\mp\ri zt).
	\end{equation}
\end{lemma}
\begin{proof}
	Define a hyperbolic integral path $\Gamma=\Gamma_+\cup\Gamma_-$, where
	$$\Gamma_{\pm}=\{\xi_{\pm}(t): t\geq 1\}.$$
	For any $R>0$, let $O^+_R$ and $O^-_R$ be the parts of the circle $\{\xi:|\xi|=R\}$ that are bounded by the real axis and $\Gamma_{\pm}$, respectively (see Fig. \ref{contour}).
	\begin{figure}[ht!]\label{contour}
		\centering
		\includegraphics[scale=0.8]{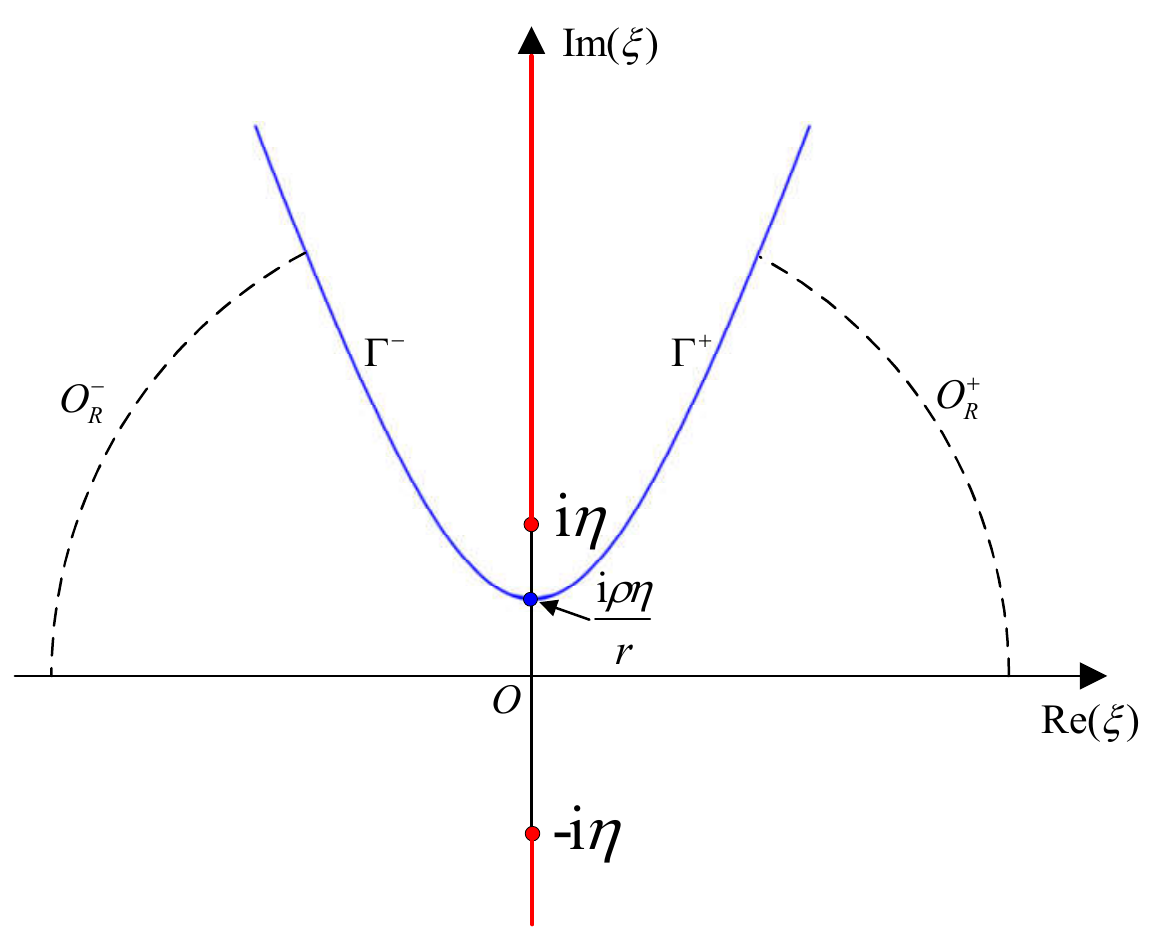}
		\caption{The Cagniard-de Hoop transform from the real axis to $\Gamma_+\cup\Gamma_-$.}
	\end{figure}
	Denote by $\xi(t_{R}^{\pm})=Re^{\ri\theta_R^{\pm}}$ the intersections of $O_R^{\pm}$ and $\Gamma^{\pm}$. Then, $0<\theta_R^+<\frac{\pi}{2}$, $\frac{\pi}{2}<\theta_R^-<\pi$ and
	\begin{equation}\label{circularintegral}
	\begin{split}
	\int_{O_R^+}f(\xi)e^{\ri\xi\rho-\sqrt{\eta^2+\xi^2}z}d\xi=\ri\int_0^{\theta_R^+}f(Re^{\ri\theta})e^{-R\rho\sin\theta-\mathfrak{Re}\zeta(\theta)z}e^{\ri (R\rho\cos\theta-\mathfrak{Im}\zeta(\theta)z)}Re^{\ri\theta}d\theta,\\
	\int_{O_R^-}f(\xi)e^{\ri\xi\rho-\sqrt{\eta^2+\xi^2}z}d\xi=\ri\int_{\theta_R^-}^{\pi}f(Re^{\ri\theta})e^{-R\rho\sin\theta-\mathfrak{Re}\zeta(\theta)z}e^{\ri (R\rho\cos\theta-\mathfrak{Im}\zeta(\theta)z)}Re^{\ri\theta}d\theta,
	\end{split}
	\end{equation}
	where $\zeta(\theta)=\sqrt{\eta^2+R^2e^{2\ri\theta}}$. Choosing the branch according to \eqref{squarerootbranch} gives a lower bound
	\begin{equation}
	\mathfrak{Re}\;\zeta(\theta)=\sqrt{\frac{\sqrt{(\eta^2+R^2\cos2\theta)^2+R^4\sin^22\theta}+\eta^2+R^2\cos 2\theta}{2}}\geq R|\cos\theta|.
	\end{equation}
	Noting that $0<\theta_R^+<\frac{\pi}{2}$, $\frac{\pi}{2}<\theta_R^-<\pi$ and $z>0$, we have
	\begin{equation}
	\begin{split}
	\cos\theta_R^+=&\frac{z\sqrt{t^2_R-1}}{\sqrt{\rho^2t^2_R+z^2(t^2_R-1)}}>\frac{z}{r}\sqrt{1-\frac{1}{t^2_R}} \geq \frac{\sqrt{3}}{2}\frac{z}{r},\\
	\cos\theta_R^-=&\frac{-z\sqrt{t^2_R-1}}{\sqrt{\rho^2t^2_R+z^2(t^2_R-1)}}< -\frac{z}{r}\sqrt{1-\frac{1}{t^2_R}} \leq -\frac{\sqrt{3}}{2}\frac{z}{r},
	\end{split}
	\end{equation}
	for all $R$ such that $t_R\geq 2$. Thus
	\begin{equation}
	\begin{split}
	&\mathfrak{Re}\;\zeta(\theta)\geq R\cos\theta_R^+> \frac{\sqrt{3}}{2}\frac{Rz}{r},\quad 0\leq\theta\leq\theta_R^+,\\
	&\mathfrak{Re}\;\zeta(\theta)\geq -R\cos\theta_R^-> \frac{\sqrt{3}}{2}\frac{Rz}{r},\quad\theta_R^-\leq\theta\leq\pi,
	\end{split}
	\end{equation}
	if $t_R\geq 2$. Applying the above estimates in \eqref{circularintegral} and then using the assumption $z>0$, we obtain
	\begin{equation}
	\Big|\int_{O_R^{\pm}}f(\xi)e^{\ri\xi\rho-\sqrt{\eta^2+\xi^2}z}d\xi\Big|\leq CR^{m+1}e^{-\frac{\sqrt{3}z^2}{2r}R}\rightarrow 0,\quad R\rightarrow \infty.
	\end{equation}
	By choosing the branch \eqref{squarerootbranch}, the square root function $\sqrt{\eta^2+\xi^2}$ have branch cut along $\{\ri\xi| \xi>\eta\}$ and $\{\ri\xi| \xi<-\eta\}$ (see Fig. \ref{contour}). Therefore, $f(\xi)e^{\ri\xi\rho-\sqrt{\eta^2+\xi^2}z}$ is analytic in the domain $\Omega_{\Gamma}^+$ for any fixed $\eta>0$.   By Cauchy's theorem, \eqref{integraltransformed} follows from the facts
	\begin{equation}
	\int_{-\infty}^{\infty}f(\xi)e^{\ri\xi\rho-\sqrt{\eta^2+\xi^2}z}d\xi=\int_{\Gamma}f(\xi)e^{\ri\xi\rho-\sqrt{\eta^2+\xi^2}z}d\xi,\quad \forall\eta>0,
	\end{equation}
	and
	\begin{equation}
	\frac{\rm d\xi_{\pm}(t)}{\rm dt}=\frac{\eta}{r\sqrt{t^2-1}}(\ri\rho\sqrt{t^2-1}\pm zt)=\frac{\ri \Lambda_{\pm}(t)}{\sqrt{t^2-1}}.
	\end{equation}
\end{proof}

In order to deform the contour of the inner integrals from the real axis to the contour $\Gamma$ defined in lemma \ref{contourchangelemma}, $\eta$ is not allowed to touch $0$. Therefore, we define sequences
\begin{equation}\label{halfintegralEk}
\begin{split}
\mathcal E^k(\bs r, \bs r';\sigma)=&\int_{\frac{1}{k}}^{\infty}\int_{-\infty}^{\infty}\sum\limits_{n=0}^{\infty}\hat g_{n}(\xi, \eta,\phi;{\bs r}')e^{\ri\xi\rho-\zeta z}\sigma(\zeta)d\xi d\eta,\\
\widetilde{\mathcal E}^k(\bs r, \bs r';\sigma)=&\int_{\frac{1}{k}}^{\infty}\int_{-\infty}^{\infty}\sum\limits_{n=0}^{\infty}\hat g_{n}(\xi, -\eta,\phi;{\bs r}')e^{\ri\xi\rho-\zeta z}\sigma(\zeta)d\xi d\eta,
\end{split}
\end{equation}
and
\begin{equation}\label{halfintegralFk}
\begin{split}
{\mathcal F}^k(\bs r, \bs r', \bs r'';\sigma)=&\int_{\frac{1}{k}}^{\infty}\int_{-\infty}^{\infty}\sum\limits_{n=0}^{\infty}\sum\limits_{\nu=0}^{\infty}g_{n\nu}(\xi, \eta,\phi;{\bs r}',{\bs r}'')e^{\ri\xi\rho-\zeta z}\sigma(\zeta)d\xi d\eta,\\
\widetilde{\mathcal F}^k(\bs r, \bs r', \bs r'';\sigma)=&\int_{\frac{1}{k}}^{\infty}\int_{-\infty}^{\infty}\sum\limits_{n=0}^{\infty}\sum\limits_{\nu=0}^{\infty}g_{n\nu}(\xi, -\eta,\phi;{\bs r}',{\bs r}'')e^{\ri\xi\rho-\zeta z}\sigma(\zeta)d\xi d\eta,
\end{split}
\end{equation}
for $k=1, 2, \cdots $. Further, their limit values are denoted by
\begin{equation}\label{EFdef}
\begin{split}
&\mathcal E(\bs r, \bs r',\sigma):=\lim\limits_{k\rightarrow\infty}\mathcal E^k(\bs r, \bs r';\sigma),\quad\mathcal F(\bs r, \bs r', \bs r'';\sigma):=\lim\limits_{k\rightarrow\infty}\mathcal F^k(\bs r, \bs r', \bs r'';\sigma),\\
& \widetilde{\mathcal E}(\bs r, \bs r';\sigma):=\lim\limits_{k\rightarrow\infty}\widetilde{\mathcal E}^k(\bs r, \bs r';\sigma),\quad \widetilde{\mathcal F}(\bs r, \bs r', \bs r'';\sigma):=\lim\limits_{k\rightarrow\infty}\widetilde{\mathcal F}^k(\bs r, \bs r', \bs r'';\sigma).
\end{split}
\end{equation}
Then,
\begin{equation}\label{InInnuExplimit}
\begin{split}
&\mathcal I(\bs r+\bs r';\sigma)=\mathcal E(\bs r, \bs r',\sigma)+ \widetilde{\mathcal E}(\bs r, \bs r';\sigma),\\
&\mathcal I(\bs r+\bs r'+\bs r'';\sigma)=\mathcal F(\bs r, \bs r', \bs r'';\sigma)+\widetilde{\mathcal F}(\bs r, \bs r', \bs r'';\sigma).
\end{split}
\end{equation}
Accordingly, we also define
\begin{equation}\label{EnFnnudef}
\begin{split}
\mathcal E_{n}^k(\bs r, \bs r',\sigma)&=\int_{\frac{1}{k}}^{\infty}\int_{-\infty}^{\infty}\hat g_{n}(\xi,\eta,\phi;{\bs r}')e^{\ri\xi\rho-\zeta z}\sigma(\zeta)d\xi d\eta,\\
\widetilde{\mathcal E}_{n}^k(\bs r, \bs r', \sigma)&=\int_{\frac{1}{k}}^{\infty}\int_{-\infty}^{\infty}\hat g_{n}(\xi,-\eta,\phi;{\bs r}')e^{\ri\xi\rho-\zeta z}\sigma(\zeta)d\xi d\eta,\\
\mathcal F_{n\nu}^k(\bs r, \bs r', \bs r'',\sigma)&=\int_{\frac{1}{k}}^{\infty}\int_{-\infty}^{\infty}g_{n\nu}(\xi,\eta,\phi;{\bs r}',{\bs r}'')e^{\ri\xi\rho-\zeta z}\sigma(\zeta)d\xi d\eta,\\
\widetilde{\mathcal F}_{n\nu}^k(\bs r, \bs r', \bs r'',\sigma)&=\int_{\frac{1}{k}}^{\infty}\int_{-\infty}^{\infty}g_{n\nu}(\xi,-\eta,\phi;{\bs r}',{\bs r}'')e^{\ri\xi\rho-\zeta z}\sigma(\zeta)d\xi d\eta,
\end{split}
\end{equation}
while the integrals in \eqref{Endeflimit}-\eqref{Fnnudeflimit} are their limit values, i.e.,
\begin{equation}\label{EnFnnudeflimit}
\begin{split}
\mathcal E_{n}(\bs r, \bs r',\sigma)=\lim\limits_{k\rightarrow\infty}\mathcal E_{n}^k(\bs r, \bs r',\sigma),\quad \mathcal F_{n\nu}(\bs r, \bs r', \bs r'',\sigma)=\lim\limits_{k\rightarrow\infty}\mathcal F_{n\nu}^k(\bs r, \bs r', \bs r'',\sigma),\\
\widetilde{\mathcal E}_{n}(\bs r, \bs r', \sigma)=\lim\limits_{k\rightarrow\infty}\widetilde{\mathcal E}_{n}^k(\bs r, \bs r', \sigma),\quad \widetilde{\mathcal F}_{n\nu}(\bs r, \bs r', \bs r'',\sigma)=\lim\limits_{k\rightarrow\infty}\widetilde{\mathcal F}_{n\nu}^k(\bs r, \bs r', \bs r'',\sigma).
\end{split}
\end{equation}

\begin{lemma}
	Suppose $z>0$, and $\sigma(k_{\rho})$ is analytic and bounded in the right half complex plane, then
	\begin{equation}\label{EnFnnudefcontourdeformed1}
	\begin{split}
	\mathcal E_{n}^k(\bs r, \bs r',\sigma)&=\ri\int_{\frac{1}{k}}^{\infty}\int_{1}^{\infty}\hat{h}_n(t,\eta,\phi;\bs r')\sigma(\zeta(t))\frac{e^{-\eta rt}}{\sqrt{t^2-1}}dt d\eta\\
	\widetilde{\mathcal E}_{n}^k(\bs r, \bs r', \sigma)&=\ri\int_{\frac{1}{k}}^{\infty}\int_{1}^{\infty}\hat{h}_n(t,-\eta,\phi;\bs r')\sigma(\zeta(t))\frac{e^{-\eta rt}}{\sqrt{t^2-1}}dt d\eta,
	\end{split}
	\end{equation}
	\begin{equation}\label{EnFnnudefcontourdeformed2}
	\begin{split}
	\mathcal F_{n\nu}^k(\bs r, \bs r', \bs r'',\sigma)&=\ri\int_{\frac{1}{k}}^{\infty}\int_{1}^{\infty}\hat{h}_{n\nu}(t,\eta,\phi;\bs r',\bs r'')\sigma(\zeta(t))\frac{e^{-\eta rt}}{\sqrt{t^2-1}}dt d\eta,\\
	\widetilde{\mathcal F}_{n\nu}^k(\bs r, \bs r', \bs r'',\sigma)&=\ri\int_{\frac{1}{k}}^{\infty}\int_{1}^{\infty}\hat{h}_{n\nu}(t,-\eta,\phi;\bs r',\bs r'')\sigma(\zeta(t))\frac{e^{-\eta rt}}{\sqrt{t^2-1}}dt d\eta,
	\end{split}
	\end{equation}
	where
	\begin{equation}\label{newdeformedkernels}
	\begin{split}
	&\hat{h}_n(t,\eta,\phi;\bs r')=\hat g_{n}(\xi_{+}(t),\eta,\phi;{\bs r}')\Lambda_{+}(t)+\hat g_{n}(\xi_{-}(t),\eta,\phi;{\bs r}')\Lambda_{-}(t),\\
	&{h}_{n\nu}(t,\eta,\phi;\bs r',\bs r'')=\hat g_{n\nu}(\xi_{+}(t),\eta,\phi;{\bs r}')\Lambda_{+}(t)+\hat g_{n\nu}(\xi_{-}(t),\eta,\phi;{\bs r}')\Lambda_{-}(t).
	\end{split}
	\end{equation}
\end{lemma}
\begin{proof}	
	According to the branch \eqref{squarerootbranch} we choose for the square root function, given any $\eta>0$, we have
	\begin{equation}\label{zetaprop}
	\mathfrak{Re}[\zeta(\xi)]=\mathfrak{Re}[\sqrt{\xi^2+\eta^2}]> 0, \quad\forall\xi\in\Omega_{\Gamma}^+.
	\end{equation}
	Together with the assumption $\sigma(k_{\rho})$ is analytic and bounded in the right half complex plane, we obtain $\sigma(\zeta(\xi))$ is analytic and bounded in $\Omega_{\Gamma}^+$.

	On the other hand, the branch \eqref{squarerootbranch} implies that $\hat g_q(\xi,\pm\eta,\phi;\tilde{\bs r})$ defined in \eqref{gnkernel} only have branch cut along $\{\ri\xi: \xi>\eta\}$ and $\{\ri\xi: \xi<-\eta\}$ (see. Fig. \ref{contour}) which has no intersection with $\Omega_{\Gamma}^+$ for any $\eta\geq\frac{1}{k}>0$. As \eqref{zetaprop} has already shown that $\sqrt{\xi^2+\eta^2}\neq 0$ for any given $\eta>0$ and $\xi\in\Omega_{\Gamma}^+$, we can conclude from the expression \eqref{gnkernel} that $\hat g_q(\xi,\pm\eta,\phi;\tilde{\bs r})$ is analytic and satisfies $|\hat g_q(\xi,\pm\eta,\phi;\tilde{\bs r})|\leq C|\xi|^q$ in the domain $\Omega_{\Gamma}^+$.  As a result, we can apply lemma \ref{contourchangelemma} to change the contour of the inner integrals in \eqref{halfintegralEk}-\eqref{halfintegralFk} from real axis to $\Gamma$. 	
\end{proof}
\begin{lemma}
	Suppose $z>0$, $z+z'>0$, and $\sigma(k_{\rho})$ is analytic and bounded in the right half complex plane, then
	\begin{equation}\label{integraldecomposition2-1}
	\begin{split}
	{\mathcal E}^{k}(\bs r,{\bs r}';\sigma)&=\ri\int_{\frac{1}{k}}^{\infty}\int_{1}^{\infty}\sum\limits_{n=0}^{\infty}\hat h_{n}(t,\eta,\phi;{\bs r}')\sigma(\zeta(t))\frac{e^{-\eta rt}}{\sqrt{t^2-1}}dt d\eta,\\
	\widetilde{\mathcal E}^{k}(\bs r,{\bs r}';\sigma)&=\ri\int_{\frac{1}{k}}^{\infty}\int_{1}^{\infty}\sum\limits_{n=0}^{\infty}\hat h_{n}(t,-\eta,\phi;{\bs r}')\sigma(\zeta(t))\frac{e^{-\eta rt}}{\sqrt{t^2-1}}dt d\eta.
	\end{split}
	\end{equation}
	Further, if have $z+z'+z''>0$, then
	\begin{equation}\label{integraldecomposition2-2}
	\begin{split}
	{\mathcal F}^{k}(\bs r,{\bs r}',{\bs r}'';\sigma)&=\ri\int_{\frac{1}{k}}^{\infty}\int_{1}^{\infty}\sum\limits_{n=0}^{\infty}\sum\limits_{\nu=0}^{\infty}h_{n\nu}(t,\eta,\phi;{\bs r}',{\bs r}'')\sigma(\zeta(t))\frac{e^{-\eta rt}}{\sqrt{t^2-1}}dt d\eta,\\
	\widetilde{\mathcal F}^{k}(\bs r,{\bs r}',{\bs r}'';\sigma)&=\ri\int_{\frac{1}{k}}^{\infty}\int_{1}^{\infty}\sum\limits_{n=0}^{\infty}\sum\limits_{\nu=0}^{\infty}h_{n\nu}(t,-\eta,\phi;{\bs r}',{\bs r}'')\sigma(\zeta(t))\frac{e^{-\eta rt}}{\sqrt{t^2-1}}dt d\eta,
	\end{split}
	\end{equation}
	where $\hat{h}_n(t,\eta,\phi;\bs r')$ and ${h}_{n\nu}(t,\eta,\phi;\bs r',\bs r'')$ are defined in \eqref{newdeformedkernels}.
\end{lemma}
\begin{proof}	
	As we have proved that $\sigma(\zeta(\xi))$ is analytic and bounded in $\Omega_{\Gamma}^+$, we will focus on the analysis for functions
	\begin{equation}
	\sum\limits_{n=0}^{\infty}\hat g_{n}(\xi, \pm\eta,\phi;{\bs r}'), \quad \sum\limits_{n=0}^{\infty}\sum\limits_{\nu=0}^{\infty}g_{n\nu}(\xi, \pm\eta,\phi;{\bs r}',\bs r'').
	\end{equation}
	Noting that they are resulted from a rotation of the Taylor expansions of exponential functions, we have
	\begin{equation}
	\sum\limits_{n=0}^{\infty}\hat g_{n}(\xi, \pm\eta,\phi;{\bs r}')=e^{\ri\xi(x'\cos\phi+y'\sin\phi)\pm\ri\eta(x'\sin\phi-y'\cos\phi)-\sqrt{\xi^2+\eta^2}z'}.
	\end{equation}
	Apparently,
	\begin{equation}
	|e^{\ri\xi(x'\cos\phi+y'\sin\phi)\pm\ri\eta(x'\sin\phi-y'\cos\phi)}|\leq 1,\quad\forall \xi\in\Omega_{\Gamma}^+,\;\;\eta\in\mathbb R,
	\end{equation}
	and
	\begin{equation}
	\sum\limits_{n=0}^{\infty}\hat g_{n}(\xi, \pm\eta,\phi;{\bs r}')e^{\ri\xi\rho-\zeta z}=e^{\ri\xi(x'\cos\phi+y'\sin\phi)\pm\ri\eta(x'\sin\phi-y'\cos\phi)}e^{\ri\xi\rho-\zeta (z+z')}
	\end{equation}
	Together with the assumptions $\rho\geq 0$, $z+z'>0$ and the fact $\sigma(\zeta(\xi))$ is analytic and bounded in the domain $\Omega_{\Gamma}^+$ for any $\eta>0$,  we can apply lemma \ref{contourchangelemma} to \eqref{halfintegralEk} to obtain \eqref{integraldecomposition2-1}.
	
	The proof for \eqref{integraldecomposition2-2} can be obtained similarly as $g_{n\nu}(\xi, \pm\eta,\phi;{\bs r}',\bs r'')$ are just the product of $\hat g_{n}(\xi, \pm\eta,\phi;{\bs r}')$ and $\hat g_{\nu}(\xi, \pm\eta,\phi;{\bs r}'')$ as defined in \eqref{gnnukernel}.
\end{proof}

{\bf Step 3: Convergence and error estimate.}
In order to exchange the order of the improper integrals and infinite summations in \eqref{integraldecomposition2-1}-\eqref{integraldecomposition2-2}, estimates for the following integrals
\begin{equation}\label{halfintegralE}
\begin{split}
&{\mathcal E}^{k,\pm}_{n}(r,\phi,{\bs r}';\sigma)=\int_{\frac{1}{k}}^{\infty}\int_{1}^{\infty}\Big|\hat g_{n}(\xi_{\pm}(t),\eta,\phi;{\bs r}')\frac{\Lambda_{\pm}(t)e^{-\eta rt}}{\sqrt{t^2-1}}\sigma(\zeta(t))\Big|dt d\eta,\\
&\widetilde{\mathcal E}^{k,\pm}_{n}(r,\phi,{\bs r}';\sigma)=\int_{\frac{1}{k}}^{\infty}\int_{1}^{\infty}\Big|\hat g_{n}(\xi_{\pm}(t),-\eta,\phi;{\bs r}')\frac{\Lambda_{\pm}(t)e^{-\eta rt}}{\sqrt{t^2-1}}\sigma(\zeta(t))\Big|dt d\eta, \\
&{\mathcal F}^{k,\pm}_{n\nu}(r,\phi,{\bs r}',{\bs r}'';\sigma)=\int_{\frac{1}{k}}^{\infty}\int_{1}^{\infty}\Big|g_{n\nu}(\xi_{\pm}(t),\eta,\phi;{\bs r}',{\bs r}'')\frac{\Lambda_{\pm}(t)e^{-\eta rt}}{\sqrt{t^2-1}}\sigma(\zeta(t))\Big|dt d\eta,\\
&\widetilde{\mathcal F}^{k,\pm}_{n\nu}(r,\phi,{\bs r}',{\bs r}'';\sigma)=\int_{\frac{1}{k}}^{\infty}\int_{1}^{\infty}\Big|g_{n\nu}(\xi_{\pm}(t),-\eta,\phi;{\bs r}',{\bs r}'')\frac{\Lambda_{\pm}(t)e^{-\eta rt}}{\sqrt{t^2-1}}\sigma(\zeta(t))\Big|dt d\eta,
\end{split}
\end{equation}
are needed for any integer $k>0$. Let us first prove estimate for their integrands.

\begin{lemma}\label{lemmagqest}
	Let $\xi_{\pm}(t)$ be the contour defined in \eqref{contourdef}, $\tilde{\bs r}=(\tilde r\sin\alpha\cos\beta,\tilde r\sin\alpha\sin\beta,\tilde r\cos\alpha)\in\mathbb R^3$ is any given vector. Then,
	\begin{equation}\label{gqestimate}
	\begin{split}
	&|\hat g_{q}(\xi_{\pm}(t), \eta,\phi;\tilde{\bs r})|\leq\frac{\tilde r^q|\Lambda_{\pm}(t)|^q}{q!}\Big(\frac{r^2t^2}{r^2t^2-\rho^2}\Big)^{\frac{q}{2}},\quad\forall t>1,\\
	&|\hat g_{q}(\xi_{\pm}(t), -\eta,\phi;\tilde{\bs r})|\leq\frac{\tilde r^q|\Lambda_{\pm}(t)|^q}{q!}\Big(\frac{r^2t^2}{r^2t^2-\rho^2}\Big)^{\frac{q}{2}},\quad\forall t>1,
	\end{split}
	\end{equation}
	hold for any integer $q\geq 0$.
\end{lemma}
\begin{proof}
	Note that
	\begin{equation*}
	\begin{split}
	\frac{\xi_{\pm}(t)\cos(\phi-\beta)+\eta\sin(\phi-\beta)}{\sqrt{\xi_{\pm}(t)^2+\eta^2}}=\frac{1}{2}\Big[\frac{\xi_{\pm}(t)+\ri\eta}{\sqrt{\xi_{\pm}(t)^2+\eta^2}}e^{-\ri(\phi-\beta)}+\frac{\xi_{\pm}(t)-\ri\eta}{\sqrt{\xi_{\pm}(t)^2+\eta^2}}e^{\ri(\phi-\beta)}\Big],\\
	\frac{\xi_{\pm}(t)\cos(\phi-\beta)-\eta\sin(\phi-\beta)}{\sqrt{\xi_{\pm}(t)^2+\eta^2}}=\frac{1}{2}\Big[\frac{\xi_{\pm}(t)+\ri\eta}{\sqrt{\xi_{\pm}(t)^2+\eta^2}}e^{\ri(\phi-\beta)}+\frac{\xi_{\pm}(t)-\ri\eta}{\sqrt{\xi_{\pm}(t)^2+\eta^2}}e^{-  \ri(\phi-\beta)}\Big].
	\end{split}
	\end{equation*}
	From the definitions in \eqref{contourdef}, we have
	\begin{equation}\label{contoureq}
	\xi_{\pm}(t)^2+\eta^2=-\Lambda_{\pm}(t)^2,
	\end{equation}
	and
	\begin{equation*}
	\begin{split}
	\Big|\frac{\xi_{\pm}(t)+\ri\eta}{\sqrt{\xi_{\pm}(t)^2+\eta^2}}\Big|=\frac{1}{|\Lambda_{\pm}(t)|}\frac{\eta}{r}|\ri(\rho t+ r)\pm z\sqrt{t^2-1}|=\sqrt{\frac{rt+\rho}{rt-\rho}},\\
	\Big|\frac{\xi_{\pm}(t)-\ri\eta}{\sqrt{\xi_{\pm}(t)^2+\eta^2}}\Big|=\frac{1}{|\Lambda_{\pm}(t)|}\frac{\eta}{r}|\ri(\rho t- r)\pm z\sqrt{t^2-1}|=\sqrt{\frac{rt-\rho}{rt+\rho}}.
	\end{split}
	\end{equation*}
	Therefore, we have the following concise formulas
	\begin{equation}\label{kernelreexpression}
	\begin{split}
	\frac{\xi_{\pm}(t)\cos(\phi-\beta)+\eta\sin(\phi-\beta)}{\sqrt{\xi_{\pm}(t)^2+\eta^2}}=\frac{(rt+\rho)e^{\ri(\gamma_{\pm}-\phi+\beta)}+(rt-\rho)e^{-\ri(\gamma_{\pm}-\phi+\beta)}}{2(r^2t^2-\rho^2)},\\
	\frac{\xi_{\pm}(t)\cos(\phi-\beta)-\eta\sin(\phi-\beta)}{\sqrt{\xi_{\pm}(t)^2+\eta^2}}=\frac{(rt+\rho)e^{\ri(\gamma_{\pm}+\phi-\beta)}+(rt-\rho)e^{-\ri(\gamma_{\pm}+\phi-\beta)}}{2(r^2t^2-\rho^2)},
	\end{split}
	\end{equation}
	where $\gamma_{\pm}$ denote the phases of the complex numbers $(\xi_{\pm}(t)+\ri\eta)/\sqrt{\xi_{\pm}(t)^2+\eta^2}$, \textit{i.e.},
	\begin{equation}
	\frac{\xi_{\pm}(t)+\ri\eta}{\sqrt{\xi_{\pm}(t)^2+\eta^2}}=\sqrt{\frac{rt+\rho}{rt-\rho}}e^{\ri\gamma_{\pm}}, \quad \frac{\xi_{\pm}(t)-\ri\eta}{\sqrt{\xi_{\pm}(t)^2+\eta^2}}=\frac{\sqrt{\xi_{\pm}(t)^2+\eta^2}}{\xi_{\pm}(t)+\ri\eta}=\sqrt{\frac{rt-\rho}{rt+\rho}}e^{-\ri\gamma_{\pm}}.
	\end{equation}
	By formulations in \eqref{kernelreexpression}, we calculate that
	\begin{equation}\label{gnkernelest1}
	\begin{split}
	&\Big|\frac{\xi_{\pm}(t)\cos(\phi-\beta)+\eta\sin(\phi-\beta)}{\sqrt{\xi_{\pm}(t)^2+\eta^2}}\sin\alpha+\ri\cos\alpha\Big|^2\\
	=&\frac{1}{r^2t^2-\rho^2}\Big|\Big((rt+\rho)e^{\ri\psi_{\pm}}+(rt-\rho)e^{-\ri\psi_{\pm}}\Big)\frac{\sin\alpha}{2}+\ri\sqrt{r^2t^2-\rho^2}\cos\alpha\Big|^2\\
	=&\frac{1}{r^2t^2-\rho^2}\Big|rt\cos(\psi_{\pm})\sin\alpha+\ri \rho \sin(\psi_{\pm})\sin\alpha+\ri\sqrt{r^2t^2-\rho^2}\cos\alpha\Big|^2\\
	=&\frac{1}{r^2t^2-\rho^2}\big((\rho\sin\alpha+\sqrt{r^2t^2-\rho^2}\sin\psi_{\pm}\cos\alpha)^2+(r^2t^2-\rho^2)\cos^2\psi_{\pm}\big)\\
	\leq &\frac{1}{r^2t^2-\rho^2}\big(\rho^2+(r^2t^2-\rho^2)\sin^2\psi_{\pm}+(r^2t^2-\rho^2)\cos^2\psi_{\pm}\big)=\frac{r^2t^2}{r^2t^2-\rho^2},
	\end{split}
	\end{equation}
	where $\psi_{\pm}=\gamma_{\pm}-\phi+\beta$. The inequality is due to the fact
	\begin{equation}
	(a\sin\theta+b\cos\theta)^2=a^2+b^2-(a\cos\theta-b\sin\theta)^2
	\end{equation}
	for any $a, b$ and $\theta$ in $\mathbb R$. Similarly, the following estimate
	\begin{equation}\label{gnkernelest2}
	\Big|\frac{\xi_{\pm}(t)\cos(\phi-\beta)-\eta\sin(\phi-\beta)}{\sqrt{\xi_{\pm}(t)^2+\eta^2}}\sin\alpha+\ri\cos\alpha\Big|^2\leq \frac{r^2t^2}{r^2t^2-\rho^2},
	\end{equation}
	can also be obtained. Then, \eqref{gqestimate} follows by applying estimate \eqref{gnkernelest2} and identity \eqref{contoureq} to the definition in \eqref{gnkernel}.
	
\end{proof}
\begin{lemma}
		Suppose for any $r>\rho\geq 0$, the density function $\sigma(\zeta(t))$ has a uniform bound $|\sigma(\zeta(t))|\leq\bs M_{\sigma}$ along the contour $\Gamma$ defined in lemma \ref{contourchangelemma}. Then, the following estimates
		\begin{equation}\label{integralest}
		\begin{split}
		\mathcal F_{n\nu}^{k,\pm}(r,\phi,{\bs r}',{\bs r}'';\sigma)\leq \frac{\pi \bs M_{\sigma}}{2}\frac{|\bs r'|^n|\bs r''|^{\nu} (n+\nu)!}{r^{n+\nu+1}n!\nu!},\\
		\widetilde{\mathcal F}_{n\nu}^{k,\pm}(r,\phi,{\bs r}',{\bs r}'';\sigma)\leq \frac{\pi \bs M_{\sigma}}{2}\frac{|\bs r'|^n|\bs r''|^{\nu} (n+\nu)!}{r^{n+\nu+1}n!\nu!},
		\end{split}
		\end{equation}
		and
		\begin{equation}\label{integralestsingle}
		\mathcal E_{n}^{k,\pm}(r,\phi,{\bs r}';\sigma)\leq \frac{\pi \bs M_{\sigma}}{2}\frac{|\bs r'|^n}{r^{n+1}},\quad
		\widetilde{\mathcal E}_{n}^{k,\pm}(r,\phi,{\bs r}';\sigma)\leq \frac{\pi \bs M_{\sigma}}{2}\frac{|\bs r'|^n}{r^{n+1}}.
		\end{equation}
		hold for any integers $n, \nu\geq 0$. 		
\end{lemma}
\begin{proof}	
	By \eqref{gnnukernel}, \eqref{contourdef}, identity \eqref{contoureq} and estimates in lemma \ref{lemmagqest}, we have
	\begin{equation*}
	\big|g_{n\nu}(\xi_{\pm}(t),\pm\eta,\phi;{\bs r}',{\bs r}'')\Lambda_{\pm}(t)\big|\leq \frac{|\bs r'|^n|\bs r''|^{\nu}}{n!\nu!}\left(\frac{|\Lambda_{\pm}(t)|rt}{\sqrt{r^2t^2-\rho^2}}\right)^{n+\nu}=\frac{|\bs r'|^n|\bs r''|^{\nu}}{n!\nu!}(\eta t)^{n+\nu}.
	\end{equation*}
	With the above estimates and the bound of $\sigma(\zeta(t))$, we derive from expressions \eqref{halfintegralE} that
	\begin{equation}
	\begin{split}
	\mathcal F^{k,\pm}_{n\nu}&(r,\phi,{\bs r}',{\bs r}'';\sigma)\leq
	\frac{\bs M_{\sigma}|\bs r'|^n|\bs r''|^{\nu}}{n!\nu!}\int_{1}^{\infty}\frac{t^{n+\nu}}{\sqrt{t^2-1}}\int_{0}^{\infty}\eta^{n+\nu}e^{-\eta rt}dt d\eta\\
	&=\frac{\bs M_{\sigma}|\bs r'|^n|\bs r''|^{\nu}(n+\nu)!}{r^{n+\nu+1}n!\nu!}\int_1^{\infty}\frac{1}{t\sqrt{t^2-1}}dt=\frac{\pi \bs M_{\sigma}}{2}\frac{|\bs r'|^n|\bs r''|^{\nu} (n+\nu)!}{r^{n+\nu+1}n!\nu!},
	\end{split}
	\end{equation}
	for any integers $n,\nu\geq 0$. The other estimates in \eqref{integralest} and \eqref{integralestsingle} can be proved similarly.
\end{proof}

\begin{theorem}\label{Th6}
	Suppose $|\bs r|>|\bs r'|+|\bs r''|$, $z>0$, $z+z'+z''>0$, and the density function $\sigma(\zeta)$ is analytic and has a bound $|\sigma(\zeta)|\leq \bs M_{\sigma}$ in the right half complex plane. Then,
	\begin{equation}\label{convergenceFn}
	\begin{split}
	{\mathcal F}(\bs r,{\bs r}',{\bs r}'';\sigma)=\sum\limits_{n=0}^{\infty}\sum\limits_{\nu=0}^{\infty}{\mathcal F}_{n\nu}(\bs r, \bs r', \bs r'',\sigma),\\
	\widetilde{\mathcal F}(\bs r,{\bs r}',{\bs r}'';\sigma)=\sum\limits_{n=0}^{\infty}\sum\limits_{\nu=0}^{\infty}\widetilde{\mathcal F}_{n\nu}(\bs r, \bs r', \bs r'',\sigma),
	\end{split}
	\end{equation}
	where the integrals are defined in \eqref{EFdef} and \eqref{EnFnnudeflimit}.
\end{theorem}
\begin{proof}	
	We only present the proof for the first summation in \eqref{convergenceFn}. Similar analysis can be done for the second one.

	By the estimate \eqref{integralest}, we have
	\begin{equation}\label{Fnnusum}
	\begin{split}
	\sum\limits_{n=0}^{\infty}\sum\limits_{\nu=0}^{\infty}\mathcal F^{k,\pm}_{n\nu}(r,\phi,{\bs r}',{\bs r}'';\sigma)\leq&\frac{\pi \bs M_{\sigma}}{2} \sum\limits_{n=0}^{\infty}\sum\limits_{\nu=0}^{\infty}\frac{|\bs r'|^n|\bs r''|^{\nu} (n+\nu)!}{r^{n+\nu+1}n!\nu!}\\
	=&\frac{\pi \bs M_{\sigma}}{2r}\sum\limits_{n=0}^{\infty}\Big(\frac{|\bs r'|+|\bs r''|}{r}\Big)^n=\frac{\pi \bs M_{\sigma}}{2(r-|\bs r'|-|\bs r''|)},
	\end{split}
	\end{equation}
	for any $r>|\bs r'|+|\bs r''|$. Therefore, we can apply the Fubini theorem to exchange the order of the improper integrals and infinite summations in \eqref{integraldecomposition2-2}. Together with the expressions in \eqref{EnFnnudefcontourdeformed2}, we obtain
	\begin{equation}\label{Fnnusum2}
	\begin{split}
	{\mathcal F}^{k}(\bs r,{\bs r}',{\bs r}'';\sigma)=\sum\limits_{n=0}^{\infty}\sum\limits_{\nu=0}^{\infty}{\mathcal F}_{n\nu}^k(\bs r, \bs r', \bs r'',\sigma).
	\end{split}
	\end{equation}
	Note that \eqref{Fnnusum} holds uniformly with respect to parameter $k$. Therefore, the series in \eqref{Fnnusum2} is also uniform convergent with respect to parameter $k$. Taking limit for $k\rightarrow\infty$ in \eqref{Fnnusum2} and exchanging order of the limit and summations, we obtain the first equality in \eqref{convergenceFn}.
\end{proof}
By following the same analysis above, we have similar conclusions for the simpler cases.
\begin{theorem}\label{Th7}
	Suppose $|\bs r|>|\bs r'|$, $z>0$, $z+z'>0$, and the density function $\sigma(\zeta)$ is analytic and has a bound $|\sigma(\zeta)|\leq \bs M_{\sigma}$ in the right half complex plane. Then,
	\begin{equation}\label{convergenceEn}
	\begin{split}
	{\mathcal E}(\bs r,{\bs r}';\sigma)=\sum\limits_{n=0}^{\infty}\mathcal E_{n}(\bs r, \bs r',\sigma),\quad
	\widetilde{\mathcal E}(\bs r,{\bs r}';\sigma)=\sum\limits_{n=0}^{\infty}\widetilde{\mathcal E}_{n}(\bs r, \bs r', \sigma),
	\end{split}
	\end{equation}
	where the integrals are defined in \eqref{EFdef} and \eqref{EnFnnudeflimit}.
\end{theorem}

Summing up expansions in \eqref{convergenceFn} and \eqref{convergenceEn}, respectively, and recalling identities \eqref{InnuEnFnu} and \eqref{InInnuExplimit}, we complete the proof for \eqref{finaltwointegralexp}.

Next, let us prove the truncation error estimate \eqref{reactionintegralestimate2}. Another error estimate \eqref{reactionintegralestimate1} can be proved similarly. By the definition \eqref{EnFnnudef}, lemma \ref{contourchangelemma} and estimates \eqref{integralest}, we have
\begin{equation}
\begin{split}
|{\mathcal F}_{n\nu}(\bs r, \bs r', \bs r'',\sigma)|&\leq \lim\limits_{k\rightarrow\infty}{\mathcal F}^{k,+}_{n\nu}(|\bs r|,\phi,{\bs r}',{\bs r}'';\sigma)+\lim\limits_{k\rightarrow\infty}{\mathcal F}^{k,-}_{n\nu}(|\bs r|,\phi,{\bs r}',{\bs r}'';\sigma)\\
&\leq \pi \bs M_{\sigma}\frac{|\bs r'|^n|\bs r''|^{\nu} (n+\nu)!}{|\bs r|^{n+\nu+1}n!\nu!},\\
|\widetilde{\mathcal F}_{n\nu}(\bs r, \bs r', \bs r'',\sigma)|&\leq \lim\limits_{k\rightarrow\infty}\widetilde{\mathcal F}^{k,+}_{n\nu}(|\bs r|,\phi,{\bs r}',{\bs r}'';\sigma)+\lim\limits_{k\rightarrow\infty}\widetilde{\mathcal F}^{k,-}_{n\nu}(|\bs r|,\phi,{\bs r}',{\bs r}'';\sigma)\\
&\leq \pi \bs M_{\sigma}\frac{|\bs r'|^n|\bs r''|^{\nu} (n+\nu)!}{|\bs r|^{n+\nu+1}n!\nu!}.
\end{split}
\end{equation}
Therefore
\begin{equation}
|\mathcal I_{n\nu}(\bs r, \bs r', \bs r'';\sigma)|=|{\mathcal F}_{n\nu}(\bs r, \bs r', \bs r'',\sigma)+\widetilde{\mathcal F}_{n\nu}(\bs r, \bs r', \bs r'',\sigma)|\leq 2\pi \bs M_{\sigma}\frac{|\bs r'|^n|\bs r''|^{\nu} (n+\nu)!}{|\bs r|^{n+\nu+1}n!\nu!},
\end{equation}
and
\begin{equation}
\begin{split}
&\Big|\mathcal I(\bs r, \bs r', \bs r'';\sigma)-\sum\limits_{n=0}^{p}\sum\limits_{\nu=0}^{p}\mathcal I_{n\nu}(\bs r, \bs r', \bs r'';\sigma)\Big|\\
\leq&2\pi \bs M_{\sigma} \left[\sum\limits_{n=0}^{p}\sum\limits_{\nu=p+1}^{\infty}\frac{|\bs r'|^n|\bs r''|^{\nu} (n+\nu)!}{|\bs r|^{n+\nu+1}n!\nu!}+\sum\limits_{n=p+1}^{\infty}\sum\limits_{\nu=0}^{\infty}\frac{|\bs r'|^n|\bs r''|^{\nu} (n+\nu)!}{|\bs r|^{n+\nu+1}n!\nu!}\right]\\
\leq&\frac{4\pi \bs M_{\sigma}}{|\bs r|}\sum\limits_{n=p+1}^{\infty}\Big(\frac{|\bs r'|+|\bs r''|}{|\bs r|}\Big)^n	=\frac{4\pi \bs M_{\sigma}}{|\bs r|-|\bs r'|-|\bs r''|} \Big(\frac{|\bs r'|+|\bs r''|}{|\bs r|}\Big)^{p+1}.
\end{split}
\end{equation}

\section{Conclusion}

In this paper, we have shown that the reaction density functions involved in the Green's function of 3-dimensional Laplace equation in multi-layered media are analytic and bounded in the right half complex plane. Based on this theoretical result, we are able to show that  the ME and LE and M2L, M2M, and L2L translation operators for the Green's functions of a 3-dimensional Laplace equation in layered media have exponential convergence similar to the classic FMM for free space problem.

The detailed analysis and estimates done here for the 3-D Laplace equation in layered media will allow us to tackle more challenging tasks in establishing the mathematical foundation for the FMMs we developed for the 3-D Poisson-Boltzmann and Helmholtz equations, and moreover, the Maxwell's equations.
As an immediate future work, we will carry out the error estimate for the FMMs for the 3-dimensional Helmholtz equation in layered media, which will require new techniques to address the effect of the surface waves (poles of density function close to the real axis) on the exponential convergence property of the MEs and LEs and M2L translation operators.

\section*{Acknowledgement}

The research of the first author is partially supported by NSFC (grant 11771137), the Construct Program of the Key Discipline in Hunan Province and a Scientific Research Fund of Hunan Provincial Education Department (No. 16B154).

\bibliographystyle{plain}


\end{document}